\documentclass[12pt,a4paper]{amsart}
\usepackage{amsfonts}
\usepackage{amsthm}
\usepackage{amsmath}
\usepackage{amscd}
\usepackage{amssymb}
\usepackage[latin2]{inputenc}
\usepackage{t1enc}
\usepackage[mathscr]{eucal}
\usepackage{indentfirst}
\usepackage{graphicx}
\usepackage{graphics}
\usepackage{pict2e}
\usepackage{epic}
\numberwithin{equation}{section}
\usepackage[margin=2.9cm]{geometry}
\usepackage{epstopdf} 
\usepackage{enumerate}
\usepackage[shortlabels]{enumitem}
\usepackage{verbatim}
\usepackage{cite}
\usepackage{mathtools}
\usepackage{mathrsfs}
\usepackage{xcolor}
\usepackage{hyperref}
\usepackage{esint}
\usepackage[normalem]{ulem}
\DeclarePairedDelimiter\ceil{\lceil}{\rceil}
\DeclarePairedDelimiter\floor{\lfloor}{\rfloor}

\theoremstyle{plain}
\newtheorem{Th}{Theorem}[section]
\newtheorem{Lemma}[Th]{Lemma}

\newtheorem{Prop}[Th]{Proposition}

\newtheoremstyle{named}{}{}{\itshape}{}{\bfseries}{.}{.5em}{\thmnote{#3}}
\theoremstyle{named}

\theoremstyle{definition}
\newtheorem{Def}[Th]{Definition}

\newtheorem{Rem}[Th]{Remark}

\newcommand{\im}{\operatorname{im}}

\newcommand{\area}{{\rm{area}}}
\newcommand{\vol}{{\rm{vol}}}
\newcommand{\tr}{{\rm{tr}}}
\newcommand{\dist}{{\rm{dist}}}
\newcommand{\ovl}{\overline}
\newcommand{\ww}{\mathtt{w}}
\newcommand{\PSLC}{{\rm PSL}_2(\mathbb{C})}
\newcommand{\PSLR}{{\rm PSL}_2(\mathbb{R})}
\newcommand{\dvol}{\operatorname{dvol}}
\definecolor{ao(english)}{rgb}{0.0, 0.5, 0.0}

\begin{document}

\title[minimal surface entropy and applications of Ricci flow]{minimal surface entropy and applications of Ricci flow on finite-volume hyperbolic 3-manifolds}

\author{Ruojing Jiang}
\address{Massachusetts Institute of Technology, Department of Mathematics, Cambridge, MA 02139} 
\email{ruojingj@mit.edu}

\author{Franco Vargas Pallete}
\address{School of Mathematical and Statistical Sciences, Arizona State University, Tempe, AZ 85287}
\email{fevargas@asu.edu}

 \subjclass[2020]{} 
 \date{}

\begin{abstract}  
This paper studies minimal surface entropy (the exponential asymptotic growth of the number of minimal surfaces up to a given value of area) for negatively curved metrics on hyperbolic $3$-manifolds of finite volume, particularly its comparison to the hyperbolic minimal surface entropy in terms of sectional and scalar curvature.

On one hand, for metrics that are bilipschitz equivalent to the hyperbolic metric and have sectional curvature bounded above by $-1$ and uniformly bounded below, we show that the entropy achieves its minimum if and only if the metric is hyperbolic.

On the other hand, by analyzing the convergence rate of the Ricci flow toward the hyperbolic metric, we prove that among all metrics with scalar curvature bounded below by $-6$ and with non-positive sectional curvature on the cusps, the entropy is maximized at the hyperbolic metric, provided that it is infinitesimally rigid. Furthermore, if the metrics are uniformly $C^0$-close to the hyperbolic metric and asymptotically cusped, then the entropy associated with the Lebesgue measure is uniquely maximized at the hyperbolic metric.
\end{abstract}

\maketitle

\section{Introduction}
On a closed hyperbolic $n$-manifold $M$ ($n\geq3$), Hamenst\"{a}dt \cite{Hamenstadt90} studied the topological entropy of the geodesic flow and proved that the hyperbolic metric attains its minimum among all metric in $M$ with sectional curvature not exceeding $-1$. On \cite{Besson-Courtois-Gallot} Besson, Courtois and Gallot studied the analogous statement under fixed volume, namely how the topological entropy of the geodesic flow is minimized by the hyperbolic metric among all negatively curved metrics on $M$ with the same volume. Recently, Calegari, Marques, and Neves \cite{Calegari-Marques-Neves} introduced the concept of the minimal surface entropy of closed hyperbolic $3$-manifolds, building on the construction and calculation of surface subgroups by Kahn and Markovic \cite{Kahn-Markovic} \cite{Kahn-Markovic_surface_subgroup}. The minimal surface entropy measures the number of essential minimal surfaces in $M$ with respect to different metrics, shifting the focus from one-dimensional objects (geodesics) to two-dimensional minimal surfaces.

Let $\mathbb{H}^3$ denote the hyperbolic $3$-space. In the Poincar\'{e} ball model, let $S^2_\infty$ be the boundary sphere of $\mathbb{H}^3$ at infinity. We write $\overline{\mathbb{H}^3}=\mathbb{H}^3\cup S^2_\infty$.
Suppose that $M=\mathbb{H}^3/\Gamma$ is an orientable $3$-manifold that admits a hyperbolic metric $h_0$. 
Consider a closed surface $S$ immersed in $M$ with genus $g\geq 2$, the surface is said to be \emph{essential} if the immersion is $\pi_1$-injective, and the image of $\pi_1(S)$ in $\Gamma$ is called a \emph{surface subgroup of genus $g$}. 
Let $S(M,g)$ denote the set of surface subgroups of genus at most $g$ up to conjugacy. For $\epsilon>0$, let the subset $S(M,g,\epsilon)\subset S(M,g)$ consist of the conjugacy classes whose limit sets are $(1+\epsilon)$-quasicircles, where a \emph{$K$-quasicircle} in $S^2_\infty$ is the image of a round circle under a $K$-quasiconformal map from $S^2_\infty$ to $S^2_\infty$.
Moreover, set \begin{equation*}
    S_\epsilon(M)=\underset{g\geq 2}{\cup}S(M,g,\epsilon).
\end{equation*} 
In particular, we consider a subset of $S_\epsilon(M)$ defined as follows. Let $\rho$ be a metric on the space of all Radon probability measures on the frame bundle $\mathcal{F}rM$, compatible with the weak-* topology, and let $\mu$ be a probability measure on $\mathcal{F}rM$. For $\Pi\in S_\epsilon(M)$, the induced Radon measure $\mu_\Pi$ on $\mathcal{F}rM$ is obtained by averaging the integral over an area-minimizing surface in the homotopy class corresponding to $\Pi$ with respect to $h_0$. As we will show in Section \ref{subsection_equidistribution}, for sufficiently small $\epsilon$, this minimizer is unique.
Define \begin{equation*}
    S_{\epsilon,\mu}(M):=\{\Pi\in S_\epsilon(M):\rho(\mu_\Pi,\mu)<\epsilon\}.
\end{equation*}
Let $h$ be an arbitrary Riemannian metric on $M$. For any $\Pi\in S(M,g)$, we set \begin{equation*}
    \area_h(\Pi)=\inf \{\area_h(\Sigma):\Sigma\in\Pi\}.
\end{equation*}
The \emph{minimal surface entropy with respect to $h$} is defined as follows by Calegari, Marques, and Neves \cite{Calegari-Marques-Neves}. 
\begin{align*}
    \overline{E}(h) &=\underset{\epsilon\rightarrow 0}{\lim}\,\underset{L\rightarrow\infty}{\limsup}\,\dfrac{\ln\#\{\area_h(\Pi)\leq 4\pi(L-1):\Pi\in S_\epsilon(M)\}}{L\ln L},\\\nonumber
    \underline{E}(h) &=\underset{\epsilon\rightarrow 0}{\lim}\,\underset{L\rightarrow\infty}{\liminf}\,\dfrac{\ln\#\{\area_h(\Pi)\leq 4\pi(L-1):\Pi\in S_\epsilon(M)\}}{L\ln L}.
\end{align*}
We write $E(h)$ if $\underline{E}(h)=\overline{E}(h)$.
Additionally, the \emph{minimal surface entropy associated with measure $\mu$} is introduced by Marques and Neves \cite{marques-neves_conformalcurrentsentropy}: \begin{equation*}
    \overline{E}_{\mu}(h) =\underset{\epsilon\rightarrow 0}{\lim}\,\underset{L\rightarrow\infty}{\limsup}\,\dfrac{\ln\#\{\area_h(\Pi)\leq 4\pi(L-1):\Pi\in S_{\epsilon,\mu}(M)\}}{L\ln L}.
\end{equation*} 
By Prokhorov's theorem, $\overline{E}_{\mu}(h)$ is independent of the metric $\rho$, as long as it induces the weak-* topology. 
$\underline{E}_\mu(h)$ and $E_\mu(h)$ are defined similarly. Clearly, for every metric $h$ on $M$, we have $\overline{E}_{\mu}(h)\leq \overline{E}(h)$.

In \cite{Calegari-Marques-Neves}, Calegari, Marques, and Neves proved that for a closed $3$-manifold $M$ admitting a hyperbolic metric $h_0$, the entropy satisfies $E(h_0)=2$. Moreover, if a Riemannian metric $h$ on $M$ has sectional curvature at most $-1$, then $\underline{E}(h)\geq 2$, with equality if and only if $h$ is isometric to $h_0$.

For closed hyperbolic manifolds of higher dimensions, when the dimension is odd, Hamenst{\"a}dt \cite{Hamenstadt} verified the existence of surface subgroups and constructed essential surfaces that are sufficiently well-distributed. 
Based on this result, the definition of minimal surface entropy can be extended to closed odd-dimensional manifolds with hyperbolic metrics, and an analogue of the theorem of Calegari-Marques-Neves was proved by the first author in \cite{Jiang}.

Concerning the influence of scalar curvature, Lowe \cite{lowe2021areascalarcurvaturehyperbolic} investigated minimal surface entropy using Ricci flow. He considered closed hyperbolic $3$-manifolds that are \emph{infinitesimally rigid}, meaning that the cohomology group \begin{equation*}
    H^1(\pi_1(M),\mathrm{Ad})=0,
\end{equation*}
where $\mathrm{Ad}$ is the adjoint representation of $\pi_1(M)\subset SO(3,1)$ via $so(3,1)\hookrightarrow so(4,1)$. He showed that if $h$ is a metric with scalar curvature $R(h)\geq -6$, then $\overline{E}(h)\leq 2$, with equality if and only if $h$ is isometric to $h_0$.

Subsequently, Lowe and Neves \cite{Lowe-Neves} removed the assumption of infinitesimal rigidity and proved the corresponding result for $\overline{E}_{\mu_{Leb}}$, the entropy associated with the Lebesgue measure $\mu_{Leb}$ on $\mathcal{F}rM$.

\subsection{Main results}
In this paper, we focus on hyperbolic $3$-manifolds of finite volume. By utilizing the construction of surface subgroups by Kahn and Wright \cite{kahn-wright2020}, as well as the existence of closed essential minimal surfaces corresponding to each subgroup, we can calculate the minimal surface entropy of the hyperbolic metric.

\renewcommand{\theTh}{A}
\begin{Th}\label{thm_hyp}
Let $(M,h_0)$ be a hyperbolic $3$-manifolds of finite volume, then we have \begin{equation*}
       E_{\mu_{Leb}}(h_0)= E(h_0)=2.
    \end{equation*}
\end{Th}

However, for a general metric $h$, the manifold $(M,h)$ may not contain an area-minimizing surface corresponding to every surface subgroup. Therefore, we need additional conditions for $h$ to ensure the existence of such surfaces. Metrics satisfying these conditions are called \emph{weakly cusped}, as defined in Definition \ref{def_asymptotically_cusp}, and the existence of minimizers under these metrics is discussed in Section \ref{section_existence_minimizer}.
In particular, any metric with $sec(h)\leq -1$ is automatically weakly cusped, and we have the following result.

\renewcommand{\theTh}{B}
\begin{Th}\label{thm_finite_vol}
Let $(M,h_0)$ be a hyperbolic $3$-manifold of finite volume, and let $h$ be a Riemannian metric on $M$. If the sectional curvature of $h$ is less than or equal to $-1$, then \begin{equation*}
        \underline{E}(h)\geq 2.
    \end{equation*}
    Furthermore, assume that $h$ is bilipschitz equivalent to $h_0$, and that there is a constant $k>1$ such that $sec(h)\geq -k^2$. Then the equality holds if and only if $h$ is isometric to $h_0$.
\end{Th}

Another focus of the paper is the application of Ricci flow to finite-volume hyperbolic $3$-manifolds. We will use it to extend \cite{lowe2021areascalarcurvaturehyperbolic} and \cite{Lowe-Neves} to this setting. Similarly to the compact case, it is natural to ask whether the Hamilton-Perelman results can be extended to noncompact manifolds. First, we need to determine if the existence theorems for Ricci flow apply in this context. Second, we are interested in the stability of the Ricci flow at its fixed point, specifically the hyperbolic metric.
Bessi{\`e}res-Besson-Maillot established the construction of Ricci flow with a specific version of surgery on cusped manifolds in \cite{Bessieres-Besson-Maillot}, called \emph{Ricci flow with bubbling-off}, with assumption that the initial metric has a cusp-like structure. For the second question, their work indicates that, after a finite number of surgeries, the solution converges smoothly to the hyperbolic metric on balls of radius $R$ for all $R>0$ as $t$ approaches infinity. However, this convergence may fail to extend globally on $M$, since the cuspidal ends allow for nontrivial Einstein variations that can alter the asymptotic behavior. Bamler \cite{Bamler} showed that if the initial metric is a small $C^0$ perturbation of the hyperbolic metric, then the Ricci flow converges on any compact sets and remains asymptotic to the same hyperbolic structure for all time. 

In \cite{Jiang-VargasPallete_RF}, the authors provided a more quantitative version of the stability of cusped hyperbolic manifolds under normalized Ricci-DeTurck flow. We impose additional conditions on the initial metric and use Bamler's stability result \cite{Bamler} to rule out trivial Einstein variations. The strategy builds on maximal regularity theory and interpolation techniques, following the approach of Angenent~\cite{Angenent_1990}, which extends the work of Da Prato and Grisvard \cite{DaPratoGrisvard1975}. By working with a pair of densely embedded Banach spaces and an operator that generates a strongly continuous analytic semigroup, we obtain maximal regularity for solutions of the normalized Ricci-DeTurck flow. This framework enables us to derive exponential convergence to the hyperbolic metric, with optimal decay rate given by the spectral estimate of the linearized operator.

On a finite-volume hyperbolic $3$-manifold, the authors showed that if the initial metric is sufficiently close to the hyperbolic metric $h_0$, then the normalized Ricci-DeTurck flow exists for all time and converges exponentially fast to $h_0$ in a weighted H{\"o}lder norm (see Theorem \ref{thm_ricci_flow} below).

Furthermore, the attractivity result implies an inequality for minimal surface entropy when the scalar curvature is bounded below. To introduce the theorem, we need the following definitions. 

\renewcommand{\theTh}{\thesection.\arabic{Th}}
\begin{Def}\label{def_asymptotically_cusp}
A complete Riemannian $h$ on $M$ is said to be: \begin{itemize}
    \item \emph{Asymptotically cusped of order $k$} if there exist a constant $\lambda>0$ and a hyperbolic metric $h_{cusp}$ defined on the cusp $\mathcal{C}=\cup_iT_i\times [0,\infty)$, such that $\lambda h|_{\mathcal{C}}-h_{cusp}$ tends to zero at infinity in $C^k$ norm;
    \item \emph{Weakly cusped} if there exists a compact set $K$ such that $sec(h)\leq 0$ in $M\setminus K$.
\end{itemize}
\end{Def}
Any asymptotically cusped metric of order $k\geq 2$ is weakly cusped.

\renewcommand{\theTh}{C}
\begin{Th}\label{Thm_sar>-6}
Let $(M,h_0)$ be a hyperbolic $3$-manifold of finite volume, and assume that it is infinitesimally rigid. Let $h$ be a weakly cusped metric on $M$. If the scalar curvature of $h$ is greater than or equal to $-6$, then \begin{equation*}
    \overline{E}(h)\leq 2.
\end{equation*}
Furthermore, suppose that $h$ is asymptotically cusped of order at least two, and it satisfies $\Vert Rm(h)\Vert_{C^{1}(M)}<\infty$. 
Then the equality holds if and only if $h$ is isometric to $h_0$.
\end{Th}

By proving the equidistribution result for finite-volume hyperbolic $3$-manifolds (Proposition \ref{prop_measure_convergence} below, which constructs a sequence of Radon probability measures on $\mathcal{F}rM$ obtained from integration over closed essential minimal surfaces and shows that it converges vaguely to $\mu_{Leb}$), we establish the following theorem.
\renewcommand{\theTh}{D}
\begin{Th}\label{Thm_sar>-6_c0perturbation}
Let $(M,h_0)$ be a hyperbolic $3$-manifold of finite volume, and let $h$ be a weakly cusped metric on $M$ that satisfies the following conditions.
\begin{itemize}
    \item $\Vert h-h_0\Vert_{C^0(M)}\leq\epsilon$ for a given constant $\epsilon>0$,
    \item $h$ is asymptotically cusped of order at least two with $\Vert Rm(h)\Vert_{C^{1}(M)}<\infty$.
\end{itemize} 
If the scalar curvature of $h$ is greater than or equal to $-6$, then \begin{equation*}
    \overline{E}_{\mu_{Leb}}(h)\leq 2.
\end{equation*}
Furthermore, the equality holds if and only if $h$ is isometric to $h_0$.
\end{Th}
\renewcommand{\theTh}{\thesection.\arabic{Th}}

\subsection{Organization}
The paper is organized as follows. Section \ref{section_hyperbolic} discusses the equidistribution properties of minimal surfaces with respect to the hyperbolic metric and establishes Theorem \ref{thm_hyp}. In Section \ref{section_existence_minimizer}, we examine conditions on general metrics on $M$ that ensure the existence of minimal surfaces. Theorem \ref{thm_finite_vol} is proved in Section \ref{section_thm_b}. Sections \ref{section_bkg_rf} and \ref{section_rf_conv} provide background on Ricci flow and introduce decay estimates toward the hyperbolic metric, which are used in the proofs of Theorems \ref{Thm_sar>-6} and \ref{Thm_sar>-6_c0perturbation}. Finally, Sections \ref{sec:proofS>-6} and \ref{section_proof_thmD} contain the proofs of Theorems \ref{Thm_sar>-6} and \ref{Thm_sar>-6_c0perturbation}, respectively.

\section*{Acknowledgements}
 FVP thanks Fernando Al-Assal, Yves Benoist, Sami Douba, Ursula Hamenst\"{a}dt, Ben Lowe, Andr\'e Neves and Richard Schwartz for helpful conversations while working on this project, and thanks IHES for their hospitality during a phase of this work. FVP was partially funded by European Union (ERC, RaConTeich, 101116694)\footnote{Views and opinions expressed are however those of the author(s) only and do not necessarily reflect those of the European Union or the European Research Council Executive Agency. Neither the European Union nor the granting authority can be held responsible for them.}

\section{entropy of hyperbolic metrics}\label{section_hyperbolic}
In this section, we discuss the proof of Theorem \ref{thm_hyp}, which calculates the minimal surface entropy for hyperbolic $3$-manifolds of finite volume. 
On such a hyperbolic $3$-manifold $(M,h_0)$, let $\#S(M,g)$ denote the cardinality of $S(M,g)$, and $\#S_{\mu_{Leb}}(M,g,\epsilon)$ denote the cardinality of the subset $S_{\mu_{Leb}}(M,g,\epsilon):= S(M,g,\epsilon)\cap S_{\epsilon,\mu_{Leb}}(M)$. We prove the following proposition, and thus deduce the minimal surface entropy of the hyperbolic metric. 
\begin{Prop}\label{prop_s}
Suppose that $\epsilon>0$ is sufficiently small and $g\in\mathbb{N}$ is sufficiently large. 
   The quantities $\#S(M,g)$ and $\#S_{\mu_{Leb}}(M,g,\epsilon)$ satisfy the following inequality:
    \begin{equation*}
   (c_1g)^{2g}\leq \#S_{\mu_{Leb}}(M,g,\epsilon)\leq \#S(M,g)\leq (c_2g)^{2g},
\end{equation*}
where $c_1>0$ is a constant that depends only on $M$ and $\epsilon$, and $c_2>0$ is a constant depending only on $M$.
\end{Prop}

\subsection{Existence of minimal surfaces}
For closed hyperbolic manifolds, it is known from the work of Schoen and Yau \cite{Schoen-Yau}, Sacks and Uhlenbeck \cite{Sacks-Uhlenbeck} that every surface subgroup produces a least-area surface in its homotopy class. However, this argument does not extend to all noncompact $3$-manifolds, see Example 6.1 in \cite{Hass-Scott} for a counterexample.

In this section, we present the existence results for hyperbolic $3$-manifolds with finitely many cusps. Hass, Rubinstein, and Wang \cite{Hass-Rubinstein-Wang}, and Ruberman \cite{Ruberman} established that in such manifolds, any noncompact essential surface of genus at least two can be homotoped to a least-area surface. Subsequently, Collin, Hauswirth, Mazet, and Rosenberg proved the existence of closed essential minimal surfaces embedded in cusped hyperbolic $3$-manifolds in \cite{Collin-Hauswirth-Mazet-Rosenberg} and \cite{Collin2019CorrigendumT}.
Later, Huang and Wang addressed the case of immersed essential surfaces in \cite{Huang-Wang}, showing that any such surface of genus at least two can be homotoped to an area-minimizing representative.

As a consequence, the minimal surface entropy of a cusped hyperbolic $3$-manifold $(M,h_0)$ can be approximated by counting the least-area closed surfaces up to homotopy. In what follows, we will estimate both upper and lower bounds for $\#S(M,g,\epsilon)$ associated with $h_0$, and use these to prove the theorem.

\subsection{The upper bound in Proposition \ref{prop_s}}\label{section_upper}

Let $S$ be a closed surface of genus $g$, and let $i(S)$ denote the injectivity radius of $S$. Since $S$ is compact, the systole length of $S$, denoted by $sl(S)$, is simply twice the injectivity radius $i(S)$ of $S$. 

Consider a \emph{triangulation} $\tau$ on $S$, which is a connected graph where each component of $S\setminus \tau$ is a triangle. Two vertices of the same triangle are called \emph{adjacent} in $\tau$. 
Let $\tau_1$ and $\tau_2$ be triangulations on $S$, with vertex sets $\mathcal{V}(\tau_1)=\{v_1^1, v_1^2,\cdots, v_1^p\}$ and $\mathcal{V}(\tau_2)=\{v_2^1, v_2^2,\cdots, v_2^p\}$, respectively. Suppose there is a bijection $h:\mathcal{V}(\tau_1)\rightarrow \mathcal{V}(\tau_2)$ such that $h(v_1^i)=v_2^i$ for all $1\leq i\leq p$. This map $h$ induces a triangulation $h(\tau_1)$ on $S$, defined by the rule that $v_2^i$ and $v_2^j$ are adjacent in $h(\tau_1)$ if and only if $v_1^i$ and $v_1^j$ are adjacent in $\tau_1$. We say that $\tau_1$ and $\tau_2$ are \emph{equivalent} if $h(\tau_1)=\tau_2$.

We state the following lemma which refers to Lemma 2.1 and Lemma 2.2 of \cite{Kahn-Markovic}.
\begin{Lemma}\label{lemma_upper_2.1}
    For any $s\leq i(S)$, there exists $k=k(s)>0$ and a triangulation $\tau$ on $S$, such that \begin{enumerate}[(1)]
    \item each edge of $\tau$ is a geodesic arc of length at most $s$,
    \item $\tau$ has at most $kg$ vertices and edges,
    \item the degree of each vertex is at most $k$.
\end{enumerate}

Furthermore, denote the set of all equivalence classes of triangulations on $S$ with genus $g$ satisfying (1)-(3) by $\mathcal{T}(k,g)$. Then, there exists a constant $c$ depending only on $k$, so that 
\begin{equation}\label{equ_upper_1}
    \#\mathcal{T}(k,g)\leq (cg)^{2g}.
\end{equation}
\end{Lemma}

Note that the lemma and the estimate \eqref{equ_upper_1} depend only on the intrinsic property of the closed surface $S$ and the choice of $s$.

Now we take the ambient manifold into consideration. Let $M$ be a hyperbolic $3$-manifold of finite volume, and $f: S\rightarrow M$ be a $\pi_1$-injective immersion that determines a surface subgroup in $S(M,g,\epsilon)$. 
It is possible to establish a hyperbolic structure on $S$ and a homotopy of $f$ that is \emph{pleated} with respect to this structure (meaning that every point of $S$ lies in the interior of a straight line segment, which is mapped to a straight line segment in $M$ with identical length). This pleated map is still denoted as $f$. We refer to 8.10 of \cite{thurston} and Lemma 3.6 of \cite{Calegari_scl} for its construction.

Denote by $sl(M)>0$ the systole of $M$. Since $f$ does not increase the hyperbolic distance and it is parabolic free, we have $sl(M)\leq sl(S)=2i(S)$. Therefore, in Lemma \ref{lemma_upper_2.1}, we can take $s=\frac{sl(M)}{6}<i(S)$.

Furthermore, we claim that $\#S(M,g,\epsilon)$ can be estimated by counting the number of homotopy classes of the pleated immersions. 
Let $f_1$ and $f_2$ be two pleated maps of genus $g$ surfaces $S_1$ and $S_2$ into $M$, respectively. Suppose that the triangulations $\tau(S_1)$ and $\tau(S_2)$ are equivalent with a bijection $h:\mathcal{V}(\tau(S_1))\rightarrow \mathcal{V}(\tau(S_2))$. 
Moreover, $M$ is covered by a family of open balls of radius $\frac{sl(M)}{12}$, denoted by $B_1, B_2, \cdots$.
We assume that for any vertex $v\in \mathcal{V}(\tau(S_1))$, the points $f_1(v)$ and $f_2(h(v))$ of $M$ are contained in the same ball $B_i$. 
Therefore, if $v,v'\in \mathcal{V}(\tau(S_1))$ are adjacent vertices, and if $s_v$ and $s_{v'}$ denote the segments connecting $f_1(v)$ to $f_2(h(v))$ and $f_1(v')$ to $f_2(h(v'))$, respectively, then lengths of $s_v$ and $s_{v'}$ are less than $\frac{sl(M)}{6}$. Moreover, due to the equivalence between $\tau(S_1)$ and $\tau(S_2)$, there are edges $e_1$, $e_2$ connecting $f_1(v)$ to $f_1(v')$, $f_2(h(v))$ to $f_2(h(v'))$, respectively, the lengths are at most $\frac{sl(M)}{6}$. 
So we get a simple closed curve that passes through $f_1(v), f_1(v'), f_2(h(v))$, and $f_2(h(v'))$. By triangle inequality, we have\begin{equation*}
    \ell(\gamma)\leq 3\max\left\{\ell(e_1)+\ell(s_v),\,\ell(e_2)+\ell(s_{v'})\right\}<sl(M).
\end{equation*}
Notice that $\gamma$ cannot shrink homotopically to a closed geodesic $\gamma'$, as otherwise it gives rise to a smaller systole length of $M$. As a result, $\gamma$ must bound a disk. 
Thus, by repeating this argument for any pair of adjacent vertices in $\tau(S_1)$, 
we conclude that $f_1|_{\tau(S_1)}$ and $f_2\circ h|_{\tau(S_2)}$ are homotopic. Since the complementary regions of $\tau(S_1)$ and $\tau(S_2)$ are triangles, the bijective map $h$ can be extended to a homeomorphism $h:S_1\rightarrow S_2$, such that $f_1$ and $f_2\circ h$ are homotopic. We then say $f_1: S_1\rightarrow M$ and $f_2: S_2\rightarrow M$ are \emph{homotopic}. 

In summary, the relation of equivalence on $\mathcal{T}(k,g)$ with images of vertices in prescribed balls of $M$ implies the relation of homotopy on pleated immersions.

Let $\widetilde{S}(M,g)$ be the subset of $S(M,g)$ that includes surfaces of fixed genus $g$. Consider any representative $\tau$ of an element in $\mathcal{T}(k,g)$. Since the pleated surface corresponding to any surface subgroup in $S(M,g)$ cannot be completely contained in the cusp regions, we can select the first vertex $v_1\in\tau$ so that it maps to a ball $B_i$ contained in the thick part of $M$. There are only finitely many such possibilities, which do not depend on $g$. We denote this number by $m$. Next, consider a vertex $v_2\neq v_1$ that bounds an edge $e$ with $v_1$. By (1) of Lemma \ref{lemma_upper_2.1}, the length of $e$ is at most $\frac{sl(M)}{6}$. Furthermore, because the balls that cover $M$ have radius $\frac{sl(M)}{12}$, there is a finite number $n>0$ (independent of $g$), such that $v_2$ can be mapped to at most $n$ options of the balls. Therefore, it follows from (2) that \begin{equation}\label{equ_upper_2}
    \#\widetilde{S}(M,g)\leq mn^{kg-1}\#\mathcal{T}(k,g).
\end{equation}
Finally, combining \eqref{equ_upper_1} and \eqref{equ_upper_2}, we can find $c_2>0$, such that \begin{equation*}
    \#S(M,g)\leq \sum_{i=2}^g\#\widetilde{S}(M,i)\leq (c_2g)^{2g}.
\end{equation*}

\subsection{The lower bound in Proposition \ref{prop_s}}\label{subsection_lower_bound}

To estimate the lower bound of the quantity $\#S_{\mu_{Leb}}(M,g,\epsilon)$, we first need to construct a closed essential surface, and then find the area-minimizing representative in its homotopy class.

\subsubsection{Construction of essential surfaces}
On a cusped hyperbolic $3$-manifold $M=\mathbb{H}^3/\Gamma$, Kahn and Wright developed an essential surface of $M$ in \cite{kahn-wright2020}. The statement is as follows, and we will explain the ideas. 

\begin{Th}[Theorem 1.1, \cite{kahn-wright2020}]\label{thm_Kahn-Wright}
    For any sufficiently small constant $\epsilon>0$, there exists a closed essential surface $\Sigma_\epsilon$ immersed in $M$. Moreover, $\Sigma_\epsilon$ has a representative in its homotopy class that is \emph{$(1+\epsilon)$-quasigeodesic}, this means that the geodesics of $\Sigma_\epsilon$ are $(1+\epsilon,\epsilon)$-quasigeodesics in $M$ ($1+\epsilon$ Lipschitz with an additive $\epsilon$ error).
\end{Th}

When $M$ is a closed hyperbolic $3$-manifold, the result analogous was proved by Kahn and Markovic \cite{Kahn-Markovic_surface_subgroup}. See also \cite{Hamenstadt} and \cite{Kahn-Labourie-Mozes} for related results in more general compact settings. Below, we outline the construction following the framework and notation of \cite{kahn-wright2020}, and describe the properties of the resulting surface $\Sigma_\epsilon$.

\begin{proof}[Sketch of proof]

Let $\delta$ and $R$ be positive constants chosen so that $\delta\rightarrow 0$ and $R\rightarrow\infty$ as the given constant $\epsilon\to 0$. 
Define $\Gamma_{\delta, R}$ to be the space of \emph{$(\delta,R)$-good curves}, consisting of closed geodesics in $M$ whose complex translation lengths are within $2\delta$ of $2R$.

First consider the case where $M$ is compact. The essential surface $\Sigma_\epsilon$ is constructed from building blocks called \emph{$(\delta,R)$-good pair of pants} (Section 2.8 of \cite{kahn-wright2020}), where $R$ and $\delta$ together quantify the size and the twisting number of the pants. Each such pair of pants $P$ has three boundary geodesics in $M$, called \emph{cuffs}. 
Define $\Pi_{\delta,R}$, the space of $(\delta,R)$-good pants, as the set of the equivalence classes of maps $f:P\rightarrow M$ so that each cuff is homotopic to an element in $\Gamma_{\delta,R}$. We say that two representatives $f_1$ and $f_2$ are equivalent if there is an orientation-preserving homeomorphism $h$ on $P$, such that $f_1$ is homotopic to $f_2\circ h$. We still denote the equivalence class $[f(P)]$ (or a representative) as $P\in \Pi_{\delta,R}$. For each $\gamma\in\Gamma_{\delta,R}$, let $\Pi_{\delta,R}(\gamma)$ be the set of good pants with a cuff homotopic to $\gamma$. Based on orientation, $\Pi_{\delta,R}(\gamma)$ decomposes into $\Pi_{\delta,R}^+(\gamma)$ and $\Pi_{\delta,R}^-(\gamma)$, where $\Pi_{\delta,R}^+(\gamma)$ consists of the oriented good pants with a cuff homotopic to $\gamma$, and $\Pi_{\delta,R}^-(\gamma)$ represents those homotopic to the reversal $\gamma^{-1}$. 
The norm squared of the second fundamental form of each such $P$ is uniformly bounded by a constant depending on $R$ and $\delta$, which can be made arbitrarily small with sufficiently large $R$ and small $\delta$ (this can be achieved by taking sufficiently small $\epsilon$).
Moreover, because of the exponential mixing property of the geodesic flow, the pants with a common cuff are equidistributed about the cuff. 

However, when the manifold $M$ is not compact, exponential mixing does not guarantee equidistribution in regions with small injectivity radius (i.e., near cusps). To overcome this, Kahn-Wright introduced the \emph{umbrellas} to replace pants that extend too far into the cusps.

More precisely, we can use the \emph{height} of the pants to measure the signed distance from the cuffs to the boundary of the disjoint horoballs used to model the cusps. The height of $\gamma\in \Gamma_{\delta,R}$ is the maximum height of points in $\gamma$, and the height of $P\in\Pi_{\delta,R}$ is the maximum height of its cuffs (page 516 of \cite{kahn-wright2020}). If the height is below a cutoff number $h_C$ which is a large multiple of $\ln R$, then the pants are sufficiently well-distributed around each cuff, allowing us to attach them suitably along common cuffs. If the height of $P$ is above $h_C$, while one of the cuffs $\gamma$ has height no greater than $h_C$, we need to build an umbrella $U(P,\gamma)$ along $\gamma$ to replace $P$. The other boundaries of $U(P,\gamma)$ are below a target height $h_T$, a much smaller multiple of $\ln R$. Note that the choices of $h_C$ and $h_T$ depend on $\epsilon$, and both go to infinity as $\epsilon$ approaches zero.  
An umbrella is a collection of \emph{$(\delta,R)$-good hamster wheels} $\{H_1,\cdots,H_m\}$ (Section 2.9 of \cite{kahn-wright2020}), each of which is a punctured sphere with two outer boundary geodesics (called \emph{rims}) and $R$ inner boundary geodesics (called \emph{inner cuffs}). 
The parameters $R$ and $\delta$ together measure the size of the hamster wheel, and control the \emph{constant turning normal fields} on the rims, which are smooth unit normal fields with constant slope. 
These hamster wheels are glued recursively: One rim of $H_1$ is glued to $\gamma$ in such a way that its constant turning normal field closely matches that of $P$, and the other rim has height less than the target height $h_T$.
Each $H_{i+1}$ is then added by identifying one of its rims with an inner cuff of $H_i$ that still has height greater than $h_T$. This process continues, decreasing the height at each step, until all remaining inner cuffs have heights below the target $h_T$.
Both good pants and good hamster wheels are referred to as \emph{good components} (Section 2.10 of \cite{kahn-wright2020}).

Furthermore, Section 2.10 of \cite{kahn-wright2020} explains the way to concatenate good components. Specifically, on each pair of good pants (or each good hamster wheel), the shortest geodesic arc that connects two cuffs (or, in the case of a hamster wheel, two inner cuffs) is called an \emph{orthogeodesic}. For good pants, an orthogeodesic intersects a cuff $\gamma$ at an endpoint, this intersection point, together with its normal direction pointing along $\gamma$, determines a point in the unit normal bundle $N^1(\gamma)$, called a \emph{foot}. For hamster wheels, the \emph{feet} (or \emph{formal feet}) $\mathbf{f}_+=(f_+,n_+)$ and $\mathbf{f}_-=(f_-,n_-)$ in $N^1(\gamma)$ are defined as the unique points such that $f_+$ and $f_-$ lie close to the intersections of orthogeodesics with $\gamma$, and $\mathbf{f}_+-\mathbf{f}_-$ is equal to the half-length of $\gamma$ (i.e. the complex distance between the two orthogeodesics along $\gamma$).
Two $(\delta,R)$-good components $Q$ and $Q'$ are considered \emph{well-matched} along a common boundary geodesic $\gamma\in \Gamma_{\delta,R}$ if, as $R\rightarrow\infty$ and $\delta\rightarrow 0$, \begin{itemize}
    \item either the complex distance between their feet on $\gamma$ is close to $1+i\pi$ measured by $R$ and $\delta$ (if $Q$ and $Q'$ are pants or hamster wheels having $\gamma$ as an inner cuff),

    \item or the constant turning normal fields along $\gamma$ form a bend of at most a multiple of $\delta$ (if $Q$ or $Q'$ is a hamster wheel with a rim $\gamma$).
\end{itemize}

Let $A_{\delta,R}^+(\gamma)$ denote the union of $\Pi_{\delta,R}^+(\gamma)$ and a small set of weighted good hamster wheels with a boundary homotopic to $\gamma$, and similarly define $A_{\delta,R}^-(\gamma)$ using $\gamma^{-1}$. Consider the good curves in $\Gamma_{\delta,R}^{\leq h_C}\subset \Gamma_{\delta,R}$, which are those with height at most $h_C$. Theorem 5.2 of \cite{kahn-wright2020} gives a bijection \begin{equation}\label{equ_bijection_matching}
\sigma_\gamma: A_{\delta,R}^+(\gamma)\rightarrow A_{\delta,R}^-(\gamma)
\end{equation}
for every $\gamma\in \Gamma_{\delta,R}^{\leq h_C}$, ensuring that all the oriented good components are well-matched.

Next, consider the ``bad'' cuffs in $\Gamma_{\delta,R}^{> h_C}:=\Gamma_{\delta,R}\setminus \Gamma_{\delta,R}^{\leq h_C}$ whose heights exceed $h_C$.
For $\gamma\in \Gamma_{\delta,R}^{\leq h_C}$, denote by $\Pi_{\delta,R}^{> h_C}(\gamma)\subset \Pi_{\delta,R}(\gamma)$ the subset of $(\delta,R)$-good pants with at least one cuff in $\Gamma_{\delta,R}^{> h_C}$.
Using the Umbrella Theorem (Theorem 4.1 of \cite{kahn-wright2020}), we cut off $P\in \Pi_{\delta,R}^{> h_C}(\gamma)$ and glue in an umbrella $U(P,\gamma)$ along $\gamma$ as explained earlier. \eqref{equ_bijection_matching} implies that all the remaining good components are still well-matched, and the resulting surface is denoted as $\Sigma_\epsilon$. Additionally, by Theorem 2.2 of \cite{kahn-wright2020}, each connected component of $\Sigma_\epsilon$ admits a representative in its homotopy class which is essential and $(1+\epsilon)$-quasigeodesic.

\end{proof}

\subsubsection{Sequences that equidistribute.}\label{subsection_equidistribution}
The essential surface $\Sigma_\epsilon$ may contain a finite number of connected components, each component $\Sigma_\epsilon^i$ is corresponding to a surface subgroup in $S_\epsilon(M)$ and homotoped to a representative $S_\epsilon^i$ minimizing the area. Since the limit set of the surface subgroup is a $(1+\epsilon)$-quasicircle, by \cite{Seppi}, the norm squared of the second fundamental form of $S_\epsilon^i$ is controlled by $O(\epsilon)$. Then it follows from \cite{Uhlenbeck} that $S_\epsilon^i$ is the unique minimal surface in its homotopy class, and it is $(1+O(\epsilon))$-quasigeodesic. 

In this section, we analyze the measures on $\mathcal{F}rM$ (the frame bundle in $M$) induced by these minimal surfaces and discuss how to obtain a sequence that equidistributes.

First, we introduce some notations from \cite{Labourie}.
Let $\mathcal{F}(\mathbb{H}^3,\epsilon)$ be the space of conformal minimal immersions $\Phi:\mathbb{H}^2\rightarrow\mathbb{H}^3$,
such that $\Phi(\partial_\infty\mathbb{H}^2)$ is a $(1+\epsilon)$-quasicircle. 
Define
\begin{equation*}
    \mathcal{F}(M,\epsilon):=\mathcal{F}(\mathbb{H}^3,\epsilon)/\Gamma
\end{equation*} with the quotient topology, where $M = \mathbb{H}^3/\Gamma$. The space $\mathcal{F}(M,\epsilon)$ together with the action of $\PSLR$ by pre-composition \begin{equation}\label{psl2r}
    \mathcal{R}_\gamma:\mathcal{F}(M,\epsilon)\rightarrow\mathcal{F}(M,\epsilon),\quad \mathcal{R}_\gamma(\phi)=\phi\circ\gamma^{-1},\quad\forall\gamma\in \PSLR
\end{equation}
is called the \emph{conformal minimal lamination associated with $M$}.

A \emph{laminar measure} on $\mathcal{F}(M,\epsilon)$ stands for a probability measure which is invariant under the $\PSLR$-action defined as above. As the primary example, let $\Pi_i\in S_{\frac1i}(M)$ be a representation of a Fuchsian group $G_i<\PSLR$ in $\Gamma<\PSLC$. There exists $\phi_i\in\mathcal{F}\left(M,\frac{1}{i}\right)$ that is equivariant with respect to $\Pi_i$, and the unique minimal surface $S_i\in \Pi_i$ (provided that $i$ is sufficiently large) can be denoted by the image of $\phi_i(\mathbb{H}^2/G_i)$. The laminar measure associated with $\phi_i$ is defined as follows.  \begin{equation}\label{def_laminar_measure}
    \delta_{\phi_i}(f)=\frac{1}{\vol(U_i)}\int_{U_i} f(\phi_i\circ\gamma)d\nu_0(\gamma), \quad\forall f\in C^0\Big(\mathcal{F}\Big(M,\frac{1}{i}\Big)\Big),
\end{equation}
where $U_i$ is the fundamental domain of the closed hyperbolic surface $\mathbb{H}^2/G_i$, and $\nu_0$ denotes the bi-invariant measure on $\PSLR$. 

Since $M$ is non-compact the space of laminar measures on $\mathcal{F}(M,\epsilon)$ is not necessarily weakly compact. In light of that, we consider the canonical continuous map $\Omega$ from $\mathcal{F}(M,\epsilon)$ to the frame bundle $\mathcal{F}rM$, and focus on the push-forward measures on $\mathcal{F}rM$ via $\Omega_*$ instead.
When $M$ is compact, $\mathcal{F}rM$ is also compact and the space of probability measures on $\mathcal{F}rM$ is compact in weak-$*$ topology. 
For non-compact manifold $M$, more discussion is provided below.

Define \begin{equation*}
    \Omega: \mathcal{F}(M,\epsilon)\rightarrow \mathcal{F}rM,\quad \Omega(\phi)=\left(\phi(i), \{e_1(\phi), e_2(\phi),e_3(\phi)\}\right),
\end{equation*}
where $i\in\mathbb{H}^2$ while $e_1(\phi), e_2(\phi)$ denotes the image by $D_i\phi$ of the standard orthonormal basis $e_1, e_2$, and $e_3\in T_{\phi(i)}M$ is the unique unit vector so that $\{e_1(\phi),e_2(\phi),e_3(\phi)\}$ is an oriented orthonormal basis.

Consider the subspace $\mathcal{F}(\mathbb{H}^3,0)\subset \mathcal{F}(\mathbb{H}^3,\epsilon)$, the space of isometric immersions $\phi_0:\mathbb{H}^2\rightarrow \mathbb{H}^3$ whose images are totally geodesic disks in $\mathbb{H}^3$. Conversely, each parametrized totally geodesic disk is uniquely determined by $\phi(i)$, and tangent orthogonal unit vectors $e_1(\phi)$, $e_2(\phi)$. 
Let $\Omega_0:\mathcal{F}(M,0)\rightarrow \mathcal{F}rM$ be the restriction of $\Omega$ to $\mathcal{F}(M,0)$, it is therefore a bijection. Using \eqref{psl2r}, we can define the $\PSLR$-action on $\mathcal{F}rM$ as follows. \begin{equation}\label{psl2r_2}
     R_\gamma:\mathcal{F}rM\rightarrow \mathcal{F}rM,\quad R_\gamma=\Omega_0\circ\mathcal{R}_\gamma\circ \Omega_0^{-1}, \quad\forall\gamma\in \PSLR.
\end{equation}
This definition coincides with the homogeneous action of $\PSLR$ on $\mathcal{F}rM$.

The following equidistribution property extends the results for closed hyperbolic $3$-manifolds established in \cite{Labourie}, \cite{Lowe-Neves}, and \cite{AlAssal} (see also \cite{Kahn-Labourie-Mozes} for more general compact manifolds) to the case of finite volume. In the following proposition, $o_i(1)\to 0$ as $i\to\infty$.

\begin{Prop}\label{prop_measure_convergence}
For $i\in\mathbb{N}$, there exists a map $\phi_i\in \mathcal{F}(M,o_i(1))$, equivariant with respect to a surface subgroup $\Pi_i\in S_{o_i(1)}(M)$, such that after passing to a subsequence, $\Omega_*\delta_{\phi_i}$ converges vaguely to the Lebesgue measure $\mu_{Leb}$ on $\mathcal{F}rM$ as $i\rightarrow\infty$ (i.e. the convergence holds in the weak-* topology on measures with respect to continuous test functions with compact supports).
In other words, we have $\Pi_i\in S_{o_i(1),\mu_{Leb}}(M)$.
\end{Prop}

\begin{proof}
As argued in Section \ref{section_upper}, the surface constructed in the previous theorem has a pleated representative. Each good pair of pants admits a pleated structure composed of two ideal triangles (see \cite[8.10]{thurston} or \cite[Lemma 3.6]{Calegari_scl}). The result arises from the equidistribution of the barycenters of these ideal triangles.

To see this, we begin by analyzing the feet of good pants. Let $\{R_j\}$ be a sequence, to be specified later, such that $R_j\rightarrow\infty$ as $j\rightarrow\infty$. The cutoff height $h_{C_j}$ defined in the previous theorem tends to infinity as $j\to\infty$. Since we only consider good pants (regardless of whether they are removed) that have at least one cuff with height bounded by $h_{C_j}$. We denote by $\Pi_{\frac 1j,R_j}^0$ the set of such pants, that is, \begin{equation*}
    \Pi_{\frac 1j,R_j}^0=\bigcup_{h(\gamma)\leq h_{C_j}}\Pi_{\frac1j,R_j}(\gamma)\subset \Pi_{\frac1j,R_j}.
\end{equation*} 
Each foot of pants corresponds to a point in $\mathcal{F}rM$. In the following discussion, we will refer to this point simply as a \emph{foot}. By Theorem 3.3 of \cite{kahn-wright2020}, for each good curve $\gamma\in\Gamma_{\frac1j,R_j}^{\leq h_{C_j}}$ with height of $\gamma$ satisfying $h(\gamma)\leq h_{C_j}$, the feet of all pants in $\Pi_{\frac1j,R_j}(\gamma)$ (containing those in $\Pi_{\frac1j,R_j}^{>h_{C_j}}(\gamma)$ whose other cuffs have large heights) become equidistributed along $\gamma$ as $j\to\infty$. 
Moreover, as argued in Section 5.3 of \cite{AlAssal}, these good curves in $\Gamma_{\frac1j,R_j}^{\leq h_{C_j}}$ are asymptotically almost surely well-distributed in the unit tangent bundle of the compact set of $M$ bounded in height by $h_{C_j}$. This produces a weighted uniform probability measure $f_j$ on $\mathcal{F}rM$, supported on the feet of pants in $\Pi_{\frac1j,R_j}^0$.
After passing to a subsequence, $f_j$ converges to $\mu_{Leb}$ on compact sets with height bounded by $h_{C_j}$. Since $h_{C_j}\to\infty$, any compactly supported continuous test function $\mathbf{g}$ will eventually be supported entirely within the region of height less than $h_{C_j}$ for sufficiently large $j$. Therefore, in a subsequence we have $f_j(\mathbf{g})\to\mu_{Leb}(\mathbf{g})$ as $j\to\infty$. This shows that the convergence is vague. 

Next, we evaluate the number of feet removed from the support of $f_j$, which is equivalent to counting the number of removed good pants or, equivalently, the number of added umbrellas. Let $\mathcal{U}_j$ be the collection of all umbrellas added to the surface $\Sigma_j$. For each good curve $\gamma$ with height $h(\gamma)\leq h_{C_j}$, the number of pants with height at least $h_{C_j}$ and with $\gamma$ as a cuff satisfies the following bound (\cite{kahn-wright2020}, Theorem 5.9):
\begin{equation}\label{equ_pants_large_height}
    \#\Pi^{>h_{C_j}}_{\frac1j,R_j}(\gamma)\leq c_jR_je^{-2\left(h_{C_j}-\max(h(\gamma),0)\right)}\frac{\#\Pi_{\frac1j,R_j}^0}{\#\Gamma_{\frac1j,R_j}^{\leq h_{C_j}}},
\end{equation}
where the constant $c_j$ is independent of $R_j$. When $h\geq 0$, $\#\Gamma^{> h}_{\frac1j,R_j}$ is at most $c_j'R_je^{-2h}\#\Gamma_{\frac1j,R_j}$ (\cite{kahn-wright2020}, Theorem 3.1), where $c_j'$ is independent of $R_j$. We can choose $R_j$ so that $c_j'R_j\leq\frac 12R_j^2$ for sufficiently large $j$. Since $h_{C_j}> \ln R_j$, we get $\#\Gamma^{>{h_{C_j}}}_{\frac1j,R_j}< \frac12\#\Gamma_{\frac1j,R_j}$, hence $\#\Gamma^{> h}_{\frac1j,R_j}< 2c_j'R_je^{-2h}\#\Gamma_{\frac1j,R_j}^{\leq h_{C_j}}$. Summing over all good curves $\gamma$ with heights $h(\gamma)\in [h,h+1)$ yields \begin{equation*}
   \sum_{h(\gamma)\in [h,h+1)} \#\Pi^{>h_{ C_j}}_{\frac1j,R_j}(\gamma)< 2c_jc_j'R_j^2e^{-2h_{C_j}+2}\#\Pi_{\frac1j,R_j}^0.
\end{equation*}
Summing over $h\in [0,h_{C_j})$ and using \eqref{equ_pants_large_height} to estimate the number of umbrellas glued along cuffs $\gamma$ with $h(\gamma)<0$, we obtain 
\begin{align*}
    \#\mathcal{U}_j=& \sum_{h(\gamma)\in [0,h_{C_j}]} \#\Pi^{>h_{ C_j}}_{\frac1j,R_j}(\gamma)+\sum_{h(\gamma)<0} \#\Pi^{>h_{ C_j}}_{\frac1j,R_j}(\gamma)\\
    <& 2c_jc_j'\ceil*{h_{C_j}}R_j^2e^{-2h_{C_j}+2}\#\Pi_{\frac1j,R_j}^0+c_jR_je^{-2h_{C_j}}\#\Pi_{\frac1j,R_j}^0.
\end{align*}
Following the choice of parameters in the proof of the main theorem of \cite{kahn-wright2020}, we set
\begin{equation}\label{equ_hcht}
    h_{C_j}=50\ln R_j,\quad h_{T_j}=6\ln R_j,
\end{equation}
which implies that for sufficient large $j$, the number of umbrellas satisfies
\begin{equation}\label{equ_number_umbrella}
    \#\mathcal{U}_j=\#\left(\Pi_{\frac1j,R_j}^0\setminus\Pi_{\frac1j,R_j}^{\leq h_{C_j}}\right)<c_j''R_j^3e^{-2h_{C_j}}\#\Pi_{\frac1j,R_j}^0=c_j''R_j^{-97}\#\Pi_{\frac1j,R_j}^0.
\end{equation}
Since the constant $c_j''$ depends only on the first index $\frac1j$ and not on $R_j$, we can choose appropriate $R_j$ such that $c_j''R_j^{-97}\to 0$ as $j\to\infty$.

Consequently, as $j\to\infty$, the measure of the feet removed from the support of $f_j$ tends to zero in the limit. We can then modify $f_j$ to a new measure $f_j'$ supported only on feet of pants in $\Pi_{\frac1j,R_j}^{\leq h_{C_j}}\subset \Pi_{\frac1j,R_j}^0$, and we have $f_j'\to\mu_{Leb}$ vaguely.

Furthermore, for each pair of pants $P\in \Pi_{\frac1j,R_j}^{\leq h_{C_j}}$, consider a geodesic triangle $\tau$ with vertices lying on the three cuffs. By spinning the vertices repeatedly around the cuffs and letting the edges of $\tau$ accumulate on them, the Hausdorff limit of $\partial\tau$ becomes a geodesic lamination consisting of three infinite leaves spiraling around the cuffs. The complement of this lamination consists of two ideal triangles. For each ideal triangle $\Tilde{\Delta}\subset \mathbb{H}^3$ with vertices $x_1,x_2,x_3\in \partial_\infty\mathbb{H}^3$, the horocycle based at $x_i$ intersects the opposite side of $\Tilde{\Delta}$ tangentially at a point called the \emph{midpoint}. The geodesic rays drawn from each midpoint toward the opposite vertex intersect at the \emph{barycenter} of the triangle, denoted by $b(\Tilde{\Delta})$.
Let $\{e_1(\Tilde{\Delta}),e_2(\Tilde{\Delta}),e_3(\Tilde{\Delta})\}$ be a positive frame at $b(\Tilde{\Delta})$ so that $e_1$ points away from a side of $\Tilde{\Delta}$.
Then each ideal triangle thus determines three possible \emph{framed barycenters} $\mathbf{b}(\Tilde{\Delta})$ in $\mathcal{F}r\mathbb{H}^3$, the frame bundle of $\mathbb{H}^3$. The \emph{framed barycenters} of each ideal triangle $\Delta\subset P$ in $M$ is then defined as the projection of the elements of $\mathbf{b}(\Tilde{\Delta})$ to $\mathcal{F}rM$, denoted by $\mathbf{b}(\Delta)$.

According to Sections 5.4-5.5 of \cite{AlAssal}, there is a right action $R_{\alpha_j}$ on $\mathcal{F}rM$ by an element \begin{equation*}\label{equ_a_j}
   \alpha_j=a_{\frac{R_j}{2}}ra_{\ln\sqrt{3}} \in \PSLR,
\end{equation*} 
where $a_t=\text{diag}(e^{\frac t2},e^{-\frac t2})$, and $r\in SO(2)$ denotes the right-angle rotation that sends the first basis vector to the second. For each $P\in\Pi_{\frac1j,R_j}^{\leq h_{C_j}}$, the map $R_{\alpha_j}$ moves each foot along a cuff $\gamma$ by a distance of $\frac{R_j}{2}$, followed by a right-angle rotation approximately aligned with the inward normal, and then proceeds along a geodesic arc orthogonal to $\gamma$ in $P$, of length $\ln\sqrt{3}$. Then $R_{\alpha_j}$ maps the feet of $P$ into an $O(\frac1j)$-neighborhood of the barycenters of its ideal triangles.
Therefore, the measure $(R_{\alpha_j})_*f_j'$ can be approximated by a weighted uniform probability measure $\beta_j$, supported on the barycenters of pants in $\Pi_{\frac1j,R_j}^{\leq h_{C_j}}$. 
Let $\textbf{g}$ be an arbitrary compactly supported continuous function on $\mathcal{F}rM$, for sufficiently large $j$, $\textbf{g}$ has support in $\mathcal{F}r(M(h_{C_j}))$, the set of frames on the compact set of $M$ with height bounded by $h_{C_j}$. In particular we have

\begin{equation*}
\left|\int_{\mathcal{F}rM}\textbf{g} \,\mu_{Leb}-\int_{\mathcal{F}r(M(h_{C_j}))} \textbf{g}\circ R_{\alpha_j}\mu_{Leb}\right|
\leq \Vert \textbf{g}\Vert_{C^0(\mathcal{F}rM)}\mu_{Leb}\left(\mathcal{F}rM\setminus \mathcal{F}r(M(h_{C_j}))\right),
\end{equation*}
which tends to zero as $j\to\infty$. Hence we get that $(R_{\alpha_j})_*\mu_{Leb}|_{\mathcal{F}r(M(h_{C_j}))}$ converges vaguely to $\mu_{Leb}$ as $j\to\infty$.
As a result, after passing to subsequence, we have $\beta_j\to\mu_{Leb}$ vaguely.

Let $b:\mathcal{F}rM\to M$ be the canonical basepoint projection.
As the arguments in Theorem 5.7 of \cite{Labourie} or Theorem 6.1 of \cite{AlAssal} are done over the considered ideal triangles, we can apply them to see that the vague convergence of $\beta_j$ to $\mu_{Leb}$ implies that for any compactly supported continuous function $g$ on $M$, letting $\textbf{g}=g\circ b$, we have
\begin{equation}\label{equ_int_pants}
    \lim_{j\rightarrow\infty}\frac{1}{\pi\# \lbrace\Delta \in\Pi_{\frac1j,R_j}^{\leq h_{C_j}}\rbrace}\sum_{\Delta\in \Pi_{\frac1j,R_j}^{\leq h_{C_j}}}\int_\Delta g\,dA_{h_0}=\int_{\mathcal{F}rM}\textbf{g}\,d\mu_{Leb},
\end{equation}
where $\Delta\in \Pi_{\frac1j,R_j}^{\leq h_{C_j}}$ denotes the ideal triangles obtained from good pants in $\Pi_{\frac1j,R_j}^{\leq h_{C_j}}$. 

Furthermore, for each umbrella $U_\gamma\in\mathcal{U}_j$ glued along $\gamma\in\Gamma_{\frac1j,R_j}^{\leq h_{C_j}}$, let $\mathcal{H}_j(U_\gamma)$ be the set of good hamster wheels contained in $U_\gamma$. According to Theorem 4.3 of \cite{kahn-wright2020}, their number satisfies 
\begin{equation}\label{equ_HinU}
    \#\mathcal{H}_j(U_\gamma)\leq R_je^{(1+\frac1j)\max(0,h(\gamma)-h_{T_j})}.
\end{equation}
Let $\mathcal{H}_j$ be the set of all good hamster wheels. Combining \eqref{equ_number_umbrella} and \eqref{equ_HinU}, we obtain
\begin{equation}\label{equ_Hnumber}
    \#\mathcal{H}_j<c_j''R_j^4e^{-2h_{C_j}+(1+\frac1j)(h_{C_j}-h_{T_j})}\#\Pi_{\frac1j,R_j}^0<2c_j''R_j^4e^{-2h_{C_j}+(1+\frac1j)(h_{C_j}-h_{T_j})}\#\Pi_{\frac1j,R_j}^{\leq h_{C_j}}.
\end{equation}

By Theorem \ref{thm_Kahn-Wright}, the good components, which contain good pants in $\Pi^{\leq h_{C_j}}_{\frac1j,R_j}$ and good hamster wheels in $\mathcal{H}_j$, can be glued together to form a collection of closed, connected, essential surface representations $\Sigma_j=(\Sigma_j^1,\cdots,\Sigma_j^{m_j})$ in $M$. The well-matching property \eqref{equ_bijection_matching} ensures that each pair of pants $P$ with a cuff $\gamma$ appears exactly once in the gluing, meaning that they all carry equal weights, denoted by $w_P$. The average weight $w_H$ of each hamster wheel $H$ given to each rim or inner cuff $\gamma$ is at most $c_j'''\dfrac{R^{14}e^{2h_{T_j}}}{\#\Gamma_{\frac1j,R_j}^{\leq h_{C_j}}}$ (\cite{kahn-wright2020}, Theorem 4.13), where $c_j'''$ is independent of $R_j$. Thus, $$w_H\leq c_j'''\frac{R^{14}e^{2h_{T_j}}}{\#\Gamma_{\frac1j,R_j}^{\leq h_{C_j}}}\frac{\#\Gamma_{\frac1j,R_j}^{\leq h_{C_j}}}{3\#\Pi_{\frac1j,R_j}^{\leq h_{C_j}}}w_P=c_j'''\frac{R^{14}e^{2h_{T_j}}}{3\#\Pi_{\frac1j,R_j}^{\leq h_{C_j}}}w_P.$$ 
For large $j$, we have $\frac1j<\frac12$, then by \eqref{equ_hcht} and \eqref{equ_Hnumber}, 
\begin{equation*}
    \frac{w_H}{w_P} \#\mathcal{H}_j< \frac{2}{3}c_j''c_j'''R_j^{18}e^{-(1-\frac1j)(h_{C_j}-h_{T_j})}<\frac{2}{3}c_j''c_j'''R_j^{-4}.
\end{equation*}
By adjusting $R_j$ if necessary, we can ensure that $c_j''c_j'''R_j^{-4}\to 0$ as $j\to\infty$.

Note that $\Sigma_j$ is totally geodesic except along the pleating locus, and the totally geodesic part occupies full measure in $\Sigma_j$. Therefore, for sufficiently large $j$, we have $\area_{h_0}(P)\approx -2\pi\chi(P)=2\pi$, $\area_{h_0}(H)\approx -2\pi\chi(H)=2\pi R_j$, and $\area_{h_0}(\Sigma_j)\approx -2\pi\chi(\Sigma_j)$.
Therefore, in \eqref{equ_int_pants} assume $g\geq 0$,
\begin{align*}
&\lim_{j\to\infty}\frac{1}{-2\pi\chi(\Sigma_j)}\int_{\Sigma_j} g\,dA_{h_0}\\
=& \lim_{j\rightarrow\infty}\frac{1}{2\pi\#\Pi_{\frac1j,R_j}^{\leq h_{C_j}}+2\pi R_j\frac{w_H}{w_P}\#\mathcal{H}_j}\left(\sum_{P\in \Pi_{\frac1j,R_j}^{\leq h_{C_j}}}\int_P g\,dA_{h_0}+\sum_{H\in\mathcal{H}_j}\frac{w_H}{w_P}\int_H g\,dA_{h_0}\right)\\
=& \lim_{j\rightarrow\infty}\frac{1}{2\pi\#\Pi_{\frac1j,R_j}^{\leq h_{C_j}}}\left(\sum_{P\in \Pi_{\frac1j,R_j}^{\leq h_{C_j}}}\int_P g\,dA_{h_0}+\sum_{H\in\mathcal{H}_j}\frac{w_H}{w_P}\int_H g\,dA_{h_0}\right)\\
\geq&
    \lim_{j\rightarrow\infty}\frac{1}{2\pi\#\Pi_{\frac1j,R_j}^{\leq h_{C_j}}}\sum_{P\in \Pi_{\frac1j,R_j}^{\leq h_{C_j}}}\int_P g\,dA_{h_0}=\int_{\mathcal{F}rM}\mathbf{g}\,d\mu_{Leb}.
\end{align*}
It indicates that, after passing to a subsequence, the Radon measures $\mu_{\Sigma_j}$ on $\mathcal{F}rM$ obtained by averaging integrals over $\Sigma_j$, converge vaguely to $\mu_{Leb}$.

The limit set of $\Pi_j^k:=\pi_1(\Sigma_j^k)$ is a $(1+o_j(1))$-quasicircle. For sufficiently large $j$, as discussed at the beginning of this subsection, there exists a unique minimal surface homotopic to $\Sigma_j^k$, which we denote by $S_j^k$. Let $S_j=S_j^1\cup\cdots\cup S_j^{m_j}$, and let $\mu_{S_i}$ be the Radon measures on $\mathcal{F}rM$ obtained by averaging integrals over $S_j$. Following \cite[Section 2]{AlAssal}, we can show that, after passing to a subsequence, both $\mu_{\Sigma_j}$ and $\mu_{S_j}$ converge to the same vague limit. 

To see this, consider the top boundary component $\partial_+C_j^k$ of the convex hull of $\Pi_j^k$, and let $C_j^k$ be the pleated surface representative of $\Pi_j^k$ such that its lift $\Tilde{C}_j^k\subset\mathbb{H}^3$ lies in $\partial_+C_j^k$. Let $\mu_{C_j}$ be the weighted measure associated with $C_j^1\cup\cdots\cup C_j^{m_j}$. 
On the one hand, if we flow the lift $\Tilde{\Sigma}_j^k$ of $\Sigma_j^k$ normally in $\mathbb{H}^3$ until it reaches $\partial_+C_j^k$, then---using the fact that this map has uniformly small derivatives on most of its domain---\cite[Theorem 2.2]{AlAssal} proves that, as $j\to \infty$ (after taking a subsequence), $\mu_{\Sigma_j}$ and $\mu_{C_j}$ have the same vague limit. On the other hand, the lift $\Tilde{S}_j^k$ of the minimal surface $S_j^k$ lies inside the convex hull of $\Pi_j^k$. By \cite[Lemma 3]{Jiang}, the convex hulls converge in Hausdorff distance in $\overline{\mathbb{H}^3}$ as $j\to\infty$, and the limit is contained in a totally geodesic disk, by taking lifts intersecting a compact fundamental domain of a thick region. In particular, the Hausdorff distance between $\Tilde{S}_j^k$ and $\Tilde{C}_j^k$ tends to zero. Then, by \cite[Theorem 2.3]{AlAssal}, it follows that $\mu_{S_j}$ also has the same vague limit as $\mu_{C_j}$. 

We conclude that a subsequence of $\mu_{S_j}$ converges vaguely to $\mu_{Leb}$. Moreover, recall that each $S_j^k$ is obtained by a map $\phi_j^k\in\mathcal{F}(M,o_j(1))$ and is associated with the laminar measure $\delta_{\phi_j^k}$ as defined in \eqref{def_laminar_measure}. This proves the following lemma.

\begin{Lemma}\label{lemma_convergence_measure_sum}
For any $j\in\mathbb{N}$, there exist a finite sequence $\phi_j^1,\cdots,\phi_j^{m_j}$ in $\mathcal{F}(M,o_j(1))$, and $\theta_j^1,\cdots,\theta_j^{m_j}\in (0,1)$ with $\theta_j^1+\cdots +\theta_j^{m_j}=1$, such that each $\phi_j^k$ is equivariant with respect to a surface subgroup in $\Gamma$. Moreover, the laminar measure \begin{equation*}
    \mu_{S_j}=\sum_{k=1}^{m_j}\theta_j^k\delta_{\phi_j^k}
\end{equation*} 
satisfies that, after passing to a subsequence, $\Omega_*\mu_{S_j}$ converges vaguely to $\mu_{Leb}$ on $\mathcal{F}rM$ as $j\rightarrow\infty$. 
\end{Lemma}

Next, in order to find a connected component $S_j^k$ such that the associated laminar measure $\delta_{\phi_j^k}$ converges vaguely to $\mu_{Leb}$, we need the following lemma. 
\begin{Lemma}\label{lemma_weak_compactness}
Let $\phi_i\in\mathcal{F}(M,\epsilon_i)$, where $\epsilon_i\rightarrow 0$ as $i\rightarrow\infty$. Then after passing to a subsequence, $\Omega_*\delta_{\phi_i}$ converges vaguely to a probability measure $\nu$ on $\mathcal{F}rM$, and $\nu$ is invariant under the homogeneous action of $\PSLR$.
\end{Lemma}

\begin{proof}
   Consider the space of continuous functions on $\mathcal{F}rM$ vanishing at infinity, denoted by $C_0(\mathcal{F}rM)$. The dual space $C_0(\mathcal{F}rM)^*$ is isometrically isomorphic to the space of finite Radon measures $\mathcal{M}(\mathcal{F}rM)$.
    Based on Urysohn's metrization theorem, the one-point compactification of $\mathcal{F}rM$, denoted by $\mathcal{F}rM^*$, is a compact metrizable Hausdorff space, then the space of continuous functions $C(\mathcal{F}rM^*)$ is separable.
    Since $C_0(\mathcal{F}rM)$ is a subspace of the metric space $C(\mathcal{F}rM^*)$, it is also separable. Therefore, due to the sequential Banach-Alaoglu theorem, the closed unit ball in $C_0(\mathcal{F}rM)^*$ or $\mathcal{M}(\mathcal{F}rM)$ is sequentially compact in the weak-$*$ topology. 
As a result, after passing to a subsequence, the probability measure $\Omega_*\delta_{\phi_i}$ converges vaguely to a probability measure $\nu$ on $\mathcal{F}rM$. 

Furthermore, we show that the limit $\nu$ is invariant under the homogeneous action of $\PSLR$. 
Consider the projection of $\mathcal{F}(M,\epsilon)$ onto $\mathcal{F}(M,0)$ given by $P:=\Omega_0^{-1}\circ \Omega$. For any $f\in C_0(\mathcal{F}rM)$, let $\eta:=f\circ \Omega\circ \mathcal{R}_\tau$, where $\mathcal{R}_\tau$ is defined in \eqref{psl2r}. We also have $f\circ R_\tau\circ \Omega=\eta\circ P$, where $R_\tau$ is defined in \eqref{psl2r_2}.
As proved in Lemma 3.2 of \cite{Lowe-Neves}, it suffices to show that
\begin{equation*}
    \lim_{i\rightarrow\infty} |\delta_{\phi_i}(\eta\circ P)-\delta_{\phi_i}(\eta)|=0.
\end{equation*}
Therefore, by the definition of $\delta_{\phi_i}$ in \eqref{def_laminar_measure}, we only need to check that for any $\gamma\in \PSLR$, we have \begin{equation*}
    \lim_{i\rightarrow\infty}|\eta\circ P(\phi_i\circ \gamma)-\eta(\phi_i\circ\gamma)|=0.
\end{equation*}
If the equation does not hold, then there exist $\alpha>0$ and $\gamma\in \PSLR$, such that $|\eta\circ P(\phi_i\circ \gamma)-\eta(\phi_i\circ \gamma)|\geq \alpha$ for an infinite subsequence of $i$. Let $\overline{\phi_i}$ be a lift of $\phi_i\circ \gamma$ in $\mathcal{F}(\mathbb{H}^3,\frac{1}{i})$. After passing to a subsequence, $\im(\overline{\phi_i})$ converges smoothly on compact sets to a totally geodesic disk of $\mathbb{H}^3$ as $i\rightarrow\infty$. Consequently, after rearrangement, all $\overline{\phi_i}(i)$ are contained in a compact subset of $\mathbb{H}^3$. 
Note that the evaluation map, which sends $\overline{\phi_i}$ to $\overline{\phi_i}(i) \in \mathbb{H}^3$, is proper (see \cite{Labourie}, Theorem 5.2). Therefore, after passing to a subsequence, $\phi_i\circ \gamma$ converges to some $\phi_\infty\in \mathcal{F}(M,0)$ with $|\eta\circ P(\phi_\infty)-\eta(\phi_\infty)|\geq \alpha>0$. However, it violates the fact that $P(\phi_\infty)=\phi_\infty$. 
\end{proof}

We now proceed with the proof of Proposition \ref{prop_measure_convergence}. 
Let $\mathcal{T}$ be the set of finite-volume totally geodesic surfaces in $M$, it contains at most countably many candidates.
We can find a decreasing sequence of open subsets $\{B_k\}\subset \mathcal{F}rM$, so that for any $k\in\mathbb{N}$, $B_k$ covers $\cup_{T\in\mathcal{T}}\mathcal{F}rT$, and it satisfies $\mu_{Leb}(B_k)<2^{-2k-1}$ and $\mu_{Leb}(\partial B_k)=0$. 
In consequence of Lemma \ref{lemma_convergence_measure_sum}, we have $\Omega_*\mu_{S_j}(B_k)<2^{-2k}$ for sufficiently large $j$ in a subsequence. Additionally, as argued in Lemma 6.2 of \cite{Lowe-Neves}, we can find a subsequence $\{j_i\}_{i\in\mathbb{N}}$, and component $\phi_i\in\{\phi_{j_i}^1,\cdots,\phi_{j_i}^{m_{j_i}}\}$, 
such that $\Omega_*(\delta_{\phi_i})(B_k)<2^{-k}$.
By Lemma \ref{lemma_weak_compactness}, after passing to a subsequence, $\Omega_*\delta_{\phi_i}$ converges vaguely to a probability measure $\nu$ on $\mathcal{F}rM$. $\nu$ is invariant under the homogeneous action of $\PSLR$, and for any compact set $M(s):=M\setminus(\cup_i T_i\times (s,\infty))$ with $s\geq 0$, it satisfies \begin{equation}\label{nu_ineq}
    \nu(B_k\cap M(s))\leq 2^{-k}.
\end{equation}

It remains to show $\nu=\mu_{Leb}$.
According to the ergodic decomposition theorem (\cite{Hasselblatt-Katok}), $\nu$ can be expressed by a linear combination of ergodic measures for the $\PSLR$-action on $\mathcal{F}rM$. Moreover, Ratner's measure classification theorem (see \cite{Ratner} or \cite{Shah}) says that any ergodic $\PSLR$-invariant measure on $\mathcal{F}rM$ is either an invariant probability measure supported on a union of $\mathcal{F}rT_i$, where $T_i\in \mathcal{T}$, or it is identical to $\mu_{Leb}$. Thus, we can write $\nu$ as \begin{equation*}
     \nu=a\mu_{\mathcal{T}}+(1-a)\mu_{Leb},
\end{equation*}
where $\mu_{\mathcal{T}}$ represents a probability measure supported on $\cup_{T\in\mathcal{T}}\mathcal{F}rT$, and its mass does not accumulate at infinity, since for any totally geodesic surface larger and larger portions of the area are contained in bigger and bigger thick regions. 
By \eqref{nu_ineq}, for all $k\in\mathbb{N}$, \begin{equation*}
    a=a\mu_{\mathcal{T}}(B_k)\leq \lim_{s\rightarrow\infty}\nu(B_k\cap M(s))\leq 2^{-k}.
\end{equation*}
So \begin{equation*}
    1-a\geq 1-2^{-k},\quad\forall k\in\mathbb{N}.
\end{equation*}
We must have $a=0$, and therefore $\nu=\mu_{Leb}$. 
\end{proof}

\begin{Rem}
    Observe that Proposition~\ref{prop_measure_convergence} is mainly applied for the case when $M$ contains at least one totally geodesic surface, but at most finitely many of them. Indeed, if $M$ does not contain any totally geodesic surface, then $\mu_{Leb}$ is the only possible $\PSLR$ invariant limit. If $M$ contains infinitely many totally geodesic surfaces, by \cite{bfms2021arithmeticity} $M$ must be arithmetic of type I, and hence contains infinitely many compact totally geodesic surfaces. A sequence of these surfaces already equidistributes by \cite{mozes_shah_1995}.
\end{Rem}

\subsubsection{The lower bound in Proposition \ref{prop_s}}
Let $\Sigma_\epsilon$ denote a closed, connected, essential surface associated with a surface subgroup in $S_{\epsilon,\mu_{Leb}}(M)$, and let $g_0$ represent the genus of $\Sigma_\epsilon$.
If $\Sigma_k$ is a degree $k$ cover over $\Sigma_\epsilon$, then the genus $g_k$ of $\Sigma_k$ is computed by the following relation of Euler characteristic: \begin{equation}\label{equ_genus of cover}
    2-2g_k=\chi(\Sigma_k)=k\,\chi(\Sigma_\epsilon)=k(2-2g_0)\quad \Longrightarrow \quad g_k=k(g_0-1)+1,
\end{equation}
Additionally, according to the M{\"u}ller-Puchta's formula (see \cite{Muller-Puchta}), the number of index $k$ subgroups of $\pi_1(\Sigma_\epsilon)$ grows like $2k(k!)^{2g_0-2}(1+o(1))$, we denoted it by $\#S_{\epsilon,\mu_{Leb}}^k(M)$. By utilizing Stirling's approximation and \eqref{equ_genus of cover}, we observe that for sufficiently large $k$, \begin{equation*}
    \#S_{\epsilon,\mu_{Leb}}^k(M)\sim 2k(2\pi k)^{g_0-1}\left(\frac{k}{e}\right)^{2k(g_0-1)}(1+o(1))\sim \left(\frac{1}{e(g_0-1)}g_k\right)^{2g_k}.
\end{equation*}
Let $c_1=\frac{1}{e(g_0-1)}$, it depends only on $g_0$, hence only on $M$ and $\epsilon$.
Therefore, we obtain the following lower bound when $g$ is large. \begin{equation*}
    \#S(M,g,\epsilon)\geq \#S_{\mu_{Leb}}(M,g,\epsilon)\geq (c_1g)^{2g}.
\end{equation*}

\subsection{Proof of Theorem \ref{thm_hyp}}

Given $\Pi\in S_\epsilon(M)$, let $S\in\Pi$ be the essential minimal surface with $\area_{h_0}(S)=\area_{h_0}(\Pi)$, and denote its genus by $g$. By Gauss-Bonnet formula and the second fundamental form estimate \begin{equation}\label{equ_seppi}
    |A|^2_{L^\infty(S,h_0)}=O(\epsilon)
\end{equation}
in \cite{Seppi}, when $\epsilon$ is small enough, 
\begin{equation}\label{equ_Gauss_h_0}
    \area_{h_0}(S)=4\pi(g-1)-\frac{1}{2}\int_{S}|A|^2dA_{h_0}=4\pi(g-1)(1+O(\epsilon)).
\end{equation}

As a result, given $0<\eta<1$, for all sufficiently small $\epsilon$, and 
sufficiently large $L$ which only depend on $\eta$, we conclude the following statements. \begin{enumerate}[(i)]
    \item For $\Pi\in S_\epsilon(M)$, if it has $\area_{h_0}(\Pi)\leq 4\pi(L-1)$, then by \eqref{equ_Gauss_h_0} the genus satisfies $g\leq \floor*{(1+\eta)L}$, and thus $\Pi\in S(M,\floor*{(1+\eta)L},\epsilon)$. \label{3.3_item_i}

    \item If $\Pi\in S(M,\floor*{(1-\eta)L},\epsilon)$, then we have $$\area_{h_0}(\Pi)\leq 4\pi (\floor*{(1-\eta)L}-1)\leq 4\pi(L-1).$$ \label{3.3_item_ii}

    \item By Proposition \ref{prop_s}, there are positive constants $c_1^\pm=c_1^\pm(M,\epsilon), c_2^\pm=c_2^\pm(M)$ such that \begin{align*}
        \big(c_1^\pm ((1\pm\eta)L-1)\big)^{2((1\pm\eta)L-1)}\leq &\#S_{\mu_{Leb}}(M,\floor*{(1\pm\eta)L},\epsilon) \leq \#S(M,\floor*{(1\pm\eta)L},\epsilon)\\
        \leq &\big(c_2^\pm (1\pm\eta)L\big)^{2(1\pm\eta)L}.
    \end{align*}
    It follows from the squeeze theorem that, for all sufficiently small $\epsilon>0$, \begin{equation*}
        \underset{L\rightarrow\infty}{\lim}\frac{\ln{\#S_{\mu_{Leb}}(M,\floor*{(1\pm\eta)L},\epsilon)}}{L\ln{L}}=\underset{L\rightarrow\infty}{\lim}\frac{\ln{\#S(M,\floor*{(1\pm\eta)L},\epsilon)}}{L\ln{L}}=2(1\pm\eta).
    \end{equation*}\label{3.3_item_iii}
\end{enumerate}
Consequently, \begin{align*}
    2(1-\eta)
    &=\underset{L\rightarrow\infty}{\lim}\frac{\ln{\#S_{\mu_{Leb}}(M,\floor*{(1-\eta)L},\epsilon)}}{L\ln{L}} \quad \text{by \ref{3.3_item_iii}}\\
    &\leq\underset{L\rightarrow\infty}{\lim}\dfrac{\ln\#\{\area_h(\Pi)\leq 4\pi(L-1):\Pi\in S_{\epsilon,\mu_{Leb}}(M)\}}{L\ln L} \quad \text{by \ref{3.3_item_ii}}\\
    &\leq\underset{L\rightarrow\infty}{\lim}\dfrac{\ln\#\{\area_h(\Pi)\leq 4\pi(L-1):\Pi\in S_\epsilon(M)\}}{L\ln L} \\
    &\leq\underset{L\rightarrow\infty}{\lim}\frac{\ln{\#S(M,\floor*{(1+\eta)L},\epsilon)}}{L\ln{L}} \quad \text{by \ref{3.3_item_i}}\\
    &= 2(1+\eta) \quad \text{by \ref{3.3_item_iii}}.
\end{align*}
As we can choose $\eta$ to be an arbitrarily small positive number, we conclude that $E_{\mu_{Leb}}(h_0)=E(h_0)=2$.

\section{Area minimizers for general metrics}\label{section_existence_minimizer}
In this section, we investigate the conditions on a general metric $h$ on $M$ that guarantee the existence of essential area-minimizing surfaces with respect to $h$. Furthermore, under certain assumptions, we show that most of the area of these surfaces lies within the thick part of $M$.

The minimal surface entropy depends only on the set of surface subgroups $S_\epsilon(M)$ for sufficiently small $\epsilon$. As shown in equation \eqref{equ_seppi}, the closed minimizer $S$ in $(M,h_0)$ corresponding to such a surface subgroup has uniformly small squared norm of the second fundamental form $|A|^2_{L^\infty(S,h_0)}$.
According to \cite[Theorem 4.1]{Rubinstein} (using \cite{Uhlenbeck} and \cite{Epstein86}), any closed immersed surface $S$ with $|A|^2_{L^\infty(S,h_0)}<2$ cannot have accidental parabolics. Therefore, in the discussion that follows, we restrict our attention to closed surface subgroups without accidental parabolics.

\subsection{Existence of closed minimal surfaces}
Consider a weakly cusped metric $h$ on $M$, that is, there exists a compact set $K$ such that $sec(h)\leq 0$ in $M\setminus K$. 
Let $S$ be a closed immersed essential minimal surface of $(M,h_0)$, in this section we find a closed immersed minimal surface $\Sigma$ of $(M,h)$ homotopic to $S$.

\begin{Th}[Controlled existence of area minimizers]\label{lem:surfaceexistencemonotonicity}
    Let $M$ be a finite-volume hyperbolic $3$-manifold, and let $h$ be a weakly cusped metric on $M$. Then for any $A>0$ and any closed surface subgroup $\Pi$ without accidental parabolics that satisfies $\area_h(\Pi)\leq A$, there exist a constant $s=s(M,h,A,\Pi)> 0$ and an area minimizer $\Sigma$ for $\Pi$ with respect to $h$, so that $\Sigma\subseteq M(s)$. Moreover, any area minimizer of $\Pi$ with respect to $h$ is contained in $M(s)$.
\end{Th}
\begin{proof}
We first consider the embedded surfaces. 
Let $s_0=s_0(M,h)\geq 0$ be the smallest number such that $M\setminus M(s_0)=\cup_iT_i\times (s_0,\infty)$ consists of disjoint cusp neighborhoods.
    Since $h$ is weakly cusped, we can choose a sufficiently large constant $s'=s'(M,h)> s_0$, so that $M\setminus M(s')$ is a union of disjoint cusp neighborhoods where $sec(h)\leq 0$. Let $s=s(M,h,A)>s'$ be a constant so that 
    \begin{equation}\label{equ_3.1_1}
        d_h(s,s')= 2\left(\sqrt{\frac{A}{\pi}}+1\right).
    \end{equation} 
    We claim this choice of $s$ fulfills the conclusion of the theorem. 
    
    We now show the existence of an area minimizer in $M(s)$. Let $\delta>0$ be a fixed number, and let $s_n>s+\delta$ be a sequence of real numbers going to $+\infty$. Assume that $\Sigma_0$ is an embedded representative of $\Pi$ with $\area_h(\Sigma_0)\leq A$. It is contained in the interior of $M(s_n-\delta)$ for sufficiently large $n$. In the $\delta$-tubular neighborhood of $\partial M(s_n)$, modify the metric $h$ to obtain a new metric $h_n$ on $M(s_n)=\cup_iT_i\times[s_n,\infty)$, such that $\partial M(s_n)$ is totally geodesic with respect to $h_n$.
   Using the result of \cite{Hass-Scott}, we can find an embedded area-minimizing surface $\Sigma_n$ in $(M(s_n),h_n)$ homotopic to $\Sigma_0$.
    
    We will start by showing that $\Sigma_n\subset M(s)$. Suppose otherwise $\Sigma_n\cap (M\setminus M(s))\neq\emptyset$. 
    Since $\Sigma_n$ is a $\pi_1$-injective immersion without accidental parabolics,
    it cannot be entirely contained in a union of cusps $\cup_iT_i\times[r,\infty)$ for some $r\geq 0$, we have that there exists a point $p\in\Sigma_n\cap (M(s)\setminus M(s'))$ which is at distance at least $\sqrt{\frac{A}{\pi}}+1$ from both $\partial M(s')$ and $\partial M(s_n-\delta)$, assuming $n$ is sufficiently large. Let $H$ be a horoball in $\mathbb{H}^3$ that is a lift of a component of $M\setminus M(s')$, and let $\Tilde{\Sigma}_n$ be the lift of $\Sigma_n$ to the universal cover such that $\Tilde{\Sigma}_n\cap H\neq\emptyset$. 
    As $\Pi$ is a parabolic-free surface subgroup, for any $n$, we see that in the universal cover the intersection of $\Tilde{\Sigma}_n$ with $H$ embeds in $\Sigma_n$. Moreover, let $\Tilde{p}$ be the lift of $p$ in $\Tilde{\Sigma}_n\cap H$, we have $B_h\left(\Tilde{p},\sqrt{\frac{A}{\pi}}+1\right)\subset H$.
    By assumption, $(H,h)$ has non-positive sectional curvature, this allows us to apply monotonicity formula \cite[Theorem 1]{Anderson82}. 
Hence, we have that $$\area_h\left(\Tilde{\Sigma}_n\cap B_h\left(\Tilde{p},\sqrt{\frac{A}{\pi}}+1\right)\right)\geq \pi\left(\sqrt{\frac{A}{\pi}}+1\right)^2>A.$$
But this is impossible, since for large enough $n$, both $\Sigma_0$ and $B_h\left(p,\sqrt{\frac{A}{\pi}}+1\right)$ are contained in $M(s_n-\delta)$, where the metric $h_n$ is identical to $h$. This would imply that
$$\area_{h_n}(\Sigma_0)=\area_h(\Sigma_0)\leq A< \area_{h_n}\left(\Sigma_n\cap B_{h_n}\left(p,\sqrt{\frac{A}{\pi}}+1\right)\right)\leq \area_{h_n}(\Sigma_n).$$
We would have that $\Sigma_0$ is homotopic to $\Sigma_n$ in $M(s_n)$ while having less area with respect to $h_n$, which contradicts the minimality of $\Sigma_n$. Hence, it follows that $\Sigma_n\subset M(s)$ for all sufficiently large $n$. 

As $s_n>s+\delta$, we get $h_n|_{M(s)} = h$. It follows that $\area_{h}(\Sigma_n) = \area_{h_n}(\Sigma_n)$ and that any surface $\Sigma_n$ is an area minimizer in $(M,h)$ for the homotopy class of $\Sigma_0$. Indeed, for any surface $\Sigma\subset M$ homotopic to $\Sigma_0$, there is $n$ large enough so that $\area_h(\Sigma) = \area_{h_n}(\Sigma) \geq \area_{h_n}(\Sigma_n)$. Hence, for the embedded case, the existence of area minimizers in $M(s)$ follows.

To see that any embedded area minimizer with respect to $h$ has such property, observe that otherwise the monotonicity formula argument of the previous paragraph shows that the minimizer has area at least $A$, which is not possible.

For the immersed case, it follows from the work of Wise \cite{Wise} that surface groups in hyperbolic cusped manifolds are separable. By a result of Scott \cite{Scott}, this implies that each immersed essential surface lifts to an embedded surface in a finite cover $\Tilde{M}$ of $M$. We carry out the argument in this $k$-sheeted cover $\Tilde{M}$, and obtain a constant $s'(M,h)$. For $s>s'$, we need to adjust \eqref{equ_3.1_1} by 
     \begin{equation}\label{equ_3.1_2}
        d_h(s,s')= 2\left(\sqrt{\frac{kA}{\pi}}+1\right),
    \end{equation} 
    as the embedded surface $\Tilde{\Sigma}_0\subset \Tilde{M}$ lifted by $\Sigma_0$ satisfies $\area_h(\Tilde{\Sigma}_0)\leq kA$. Therefore, in this immersed case, the choice of $s=s(M,h,A,\Tilde{M})$ also depends on $\Tilde{M}$, and hence ultimately on the surface subgroup $\Pi$.
\end{proof}

Using a similar approach, we prove the following existence result for metrics that are $C^1$-close to the hyperbolic metric in compact sets. This lemma will be used in Section \ref{sec:proofS>-6}.

\begin{Lemma}[Controlled existence of area minimizers for small perturbations of the hyperbolic metric] \label{lem:minimizer_close_to_hyp_in_compact}
    Let $(M,h_0)$ be a finite-volume hyperbolic $3$-manifold. For any constant $a>0$ and any closed surface subgroup $\Pi$ without accidental parabolics, there exists a compact subset $K=K(M,h_0,a,\Pi)$ of $M$ and a constant $\epsilon=\epsilon(M,h_0,a,\Pi)>0$, such that if $h$ is a metric in $M$ satisfying $\Vert (h-h_0)|_K\Vert_{C^1} <\epsilon$ and $sec(h|_K)\leq-a^2<0$, then there exists an area minimizer of $\Pi$ contained in $K$. Moreover, any area minimizer is contained in $K$.
\end{Lemma}

\begin{proof}
    Let $\Sigma_0$ be the area minimizer for $\Pi$ in $(M,h_0)$.
    Fix $s_0=s_0(M,h_0)\geq 0$ be so that the thin region $M\setminus M(s_0)=\cup_iT_i\times (s_0,\infty)$ consists of disjoint cusp neighborhoods.
    Consider $K=M(s)$ for a constant $s$ satisfying \begin{equation}\label{equ_3.2_1}
        \frac{2\pi}{a^2}\left(\cosh\left(\frac{a(s-s_0)}{4}\right)-1\right) \geq 2\area_{h_0}(\Sigma_0).
    \end{equation}
    Then $K$ depends on $M,h_0,a,\Pi$. Take sufficiently small $\epsilon$ (depending only on $M,h_0,a,\Pi$) so that for any metric $h$ in $M$ satisfying $\Vert (h-h_0)|_K\Vert_{C^1} <\epsilon$, we have
    \begin{equation}\label{equ_3.2_2}
        \area_h(\Sigma_0) < 2\area_{h_0}(\Sigma_0),
    \end{equation}
    and 
    \begin{equation}\label{equ_3.2_3}
        \dist_h(\partial M(s_0), \partial M(s))>\frac{s-s_0}{2}.
    \end{equation} 
    We claim that these choices of $K$ and $\epsilon$ fulfill the theorem.
    
    The strategy is as in Theorem~\ref{lem:surfaceexistencemonotonicity}. Let us start with existence. Let $\delta>0$ be a fixed constant, and let $h$ be a metric as in the statement and let $s_n>s+\delta$ be a sequence of real numbers going to $+\infty$. We have $\Sigma_0\subset M(s_n-\delta)$ for sufficiently large $n$. In a small $\delta$-neighborhood of $\partial M(s_n)$, modify the metric $h$ to obtain a metric $h_n$ on $M(s_n)$, so that $\partial M(s_n)$ is totally geodesic. As proved in Theorem \ref{lem:surfaceexistencemonotonicity} there exists an area minimizer $\Sigma_n$ of the homotopy class of $\Sigma_0$ in $M(s_n)$ with respect to the metric $h_n$.

    We will start by showing that $\Sigma_n\subset K$. Assuming otherwise, by \eqref{equ_3.2_3} we have $\Sigma_n \cap M(\frac{s+s_0}{2}) \neq \emptyset$, as $\Sigma_n$ is a $\pi_1$-injective immersion without accidental parabolics. The universal cover $\Tilde{\Sigma}_n$ embeds in $\mathbb{H}^3$. As $\Pi$ is a parabolic-free surface subgroup, then we have that in the universal cover every intersection of $\Tilde{\Sigma}_n$ with a horoball $H_0$ covering $M\setminus M(s_0)$ embeds in $\Sigma_n$. Then consider a geodesic ball $B$ of radius $\frac{s-s_0}{4}$, centered at a point in the lift of $\Sigma_n\cap \partial M(\frac{s+s_0}{2})$. When $n$ is large enough, the metric $h_n$ coincides with $h$ in $B$ and the sectional curvature satisfies $sec(h_n|_B)\leq -a^2$.
    Applying the monotonicity formula \cite[Theorem 1]{Anderson82} and using \eqref{equ_3.2_1} and \eqref{equ_3.2_2}, we obtain 
    \begin{align*}
        \area_{h_n}(\Tilde{\Sigma}_n\cap B)\geq &\frac{2\pi}{a^2}\left(\cosh\left(\frac{a(s-s_0)}{4}\right)-1\right) \geq  2\area_{h_0}(\Sigma_0) \\>& \area_h(\Sigma_0)=\area_{h_n}(\Sigma_0).
    \end{align*}
    This is not possible as $\Tilde{\Sigma}_n\cap B$ embeds in $\Sigma_n$, which is an area minimizer for the homotopy class of $\Sigma_0$ in $(M(s_n), h_n)$. Hence, it follows that $\Sigma_n\subset K$ for all sufficiently large $n$. 
    
    As $s_n>s+\delta$, we get $h_n|_K = h$. It follows that $\area_{h}(\Sigma_n) = \area_{h_n}(\Sigma_n)$ and that any surface $\Sigma_n$ is an area minimizer in $(M,h)$ for the homotopy class of $\Sigma_0$. Indeed, for any surface $\Sigma\subset M$ homotopic to $\Sigma_0$, there is $n$ large enough, so that $\area_h(\Sigma) = \area_{h_n}(\Sigma) \geq \area_{h_n}(\Sigma_n)$. Hence, the existence of area minimizers in $K$ follows.
    
    To see that any area minimizer has such property, observe that otherwise the monotonicity formula argument of the previous paragraph shows that the minimizer has area at least as big as $2\area_{h_0}(\Sigma)>\area_h(\Sigma_0)$, which is not possible.

    The immersed case follows as in the previous Lemma by taking finite covers.
\end{proof}

\subsection{Most area in thick regions}
In this section, we discuss two types of metrics such that most of the area of the minimizers lies within the thick part of $M$. The following two lemmas will be used to prove the main theorems in Sections \ref{section_thm_b}, \ref{sec:proofS>-6} and \ref{section_proof_thmD}.

\begin{Lemma}[Most area in the thick regions, sectional version]\label{lem:areaincusps}
    Let $h$ be a metric on $M$, and assume that there exists $a>0$ so that $sec(h)\leq -a^2<0$. Then for any $0<\kappa<1$, there exists a compact subset $K=K(M,h,a,\kappa)$ of $M$, so that if $\Pi$ is a closed surface subgroup without accidental parabolics, then any area minimizer $\Sigma$ for $\Pi$ in $h$ satisfies $$\area_{h}(\Sigma\cap K)\geq \kappa\left( \area_{h}(\Sigma)\right).$$ 
\end{Lemma}
\begin{proof}
    Let $s_0=s_0(M,h)$ be the constant defined in Theorem~\ref{lem:surfaceexistencemonotonicity}, and let $H_0$ be a horoball covering a component of $M\setminus M(s_0)$. Denote by $$N(t) : = \lbrace x\in M\setminus M(s_0)\,|\, \dist_h(x,M(s_0))>t \rbrace$$ with lift $\Tilde{N}(t)$ in $H_0$. Let $t_0\geq \frac{\kappa}{a(1-\kappa)}$. We claim that $K=M\setminus N(t_0)$ fulfills the conclusion.
    
    Let $\Sigma$ be an area minimizer for $\Pi$ with respect to $h$. Since $h$ is weakly cusped and $\Pi$ is a surface subgroup without accidental parabolics, the existence of such surface follows from Theorem~\ref{lem:surfaceexistencemonotonicity}. Moreover, we can see that $\Sigma$ has curvature bounded above by $-a^2$. Let $\Tilde{\Sigma}$ be a lift of $\Sigma$ to $\mathbb{H}^3$ and $D_t : = \Tilde{\Sigma}\cap \Tilde{N}(t)$. By isoperimetric inequality (see for instance \cite[34.2.6]{BuragoZalgaller}), we have $$a\cdot \area_h(D_t)\leq \ell_h (\partial D_t)\quad \forall t>0,$$ 
    where $\ell_h(\cdot)$ represents the length of a curve with respect to the metric induced by $h$. In particular, for any $t\leq t_0$, $$\ell_h(\partial D_t)\geq \ell_h(\partial D_{t_0}) \geq a\cdot \area_h(D_{t_0}).$$
    Applying coarea formula we get
    \[\area_h(D_{0}\setminus D_{t_0}) \geq \int_{0}^{t_0} a\cdot \area_h(D_{t_0}) dt = a\cdot t_0\cdot \area_h(D_{t_0})\geq \frac{\kappa}{1-\kappa}\area_h(D_{t_0}).
    \]
    Therefore, we obtain $\area_h(D_0\setminus D_{t_0})\geq \kappa(\area_h(D_0))$.
    As the lift of $\Sigma\cap M(s_0)$ is given by a disjoint collection of regions as $D_0$, the desired inequality follows by addition.
\end{proof}

Next, we consider an asymptotically cusped metric $h$ of order one. 
Note that $h$ is only assumed to converge to $h_{cusp}$ in $C^1$ near the ends, so it may not satisfy the sectional curvature assumptions in Theorem \ref{lem:surfaceexistencemonotonicity} or Lemma \ref{lem:minimizer_close_to_hyp_in_compact}. Therefore, the existence of area-minimizing surfaces with respect to $h$ is not guaranteed. In the lemma below, we assume that the existence is given and then discuss the area of the minimizers.

\begin{Lemma}[Most area in thick regions, asymptotic version]\label{lem:areaincusp_asymphyp}
    Let $h$ be an asymptotically cusped metric on $M$ of order one. Then for any $0<\kappa<1$, there exists a compact set $K=K(M,h_0,h,\kappa)$ of $M$, so that if $\Pi$ is a closed surface subgroup and $\Sigma$ is an area minimizer for $\Pi$ in $h$, then $$\area_{h}(\Sigma\cap K)\geq \kappa\left( \area_{h}(\Sigma)\right).$$
\end{Lemma}
\begin{proof}
By scaling the metric we can assume without loss of generality that in Definition \ref{def_asymptotically_cusp} we have $\lambda=1$. This means that in coordinates on $\mathcal{C}=\cup_iT_i\times [0,\infty)$, there is a hyperbolic metric $h_{cusp}$, such that both $| h_{ij} - (h_{cusp})_{ij}|$ and $ | h_{ij;k} - (h_{cusp})_{ij;k}|$ tend to zero as one moves toward infinity along the end.
As a result, there exists a sufficiently large constant $s_1>0$ so that for $M\setminus M(s_1)$, we have
    \begin{enumerate}
    \item The coordinate vector field $\partial_t$ induced by the $t\in [s_1,\infty)$ factor satisfies $$\frac{1}{2}\leq \Vert \partial_t\Vert_{h}\leq 2.$$
    \item  For any vector $v$ in $T(M\setminus M(s_1))$ we have
    \[
    \frac12 h_0(v,v) \leq h(v,v) \leq 2h_0(v,v). 
    \]
    \end{enumerate}
    Let $s_2>s_1+ \frac{4\kappa}{1-\kappa}$. We will show that $K=M(s_2)$ satisfies the conclusion of the theorem.

    Let $\Sigma$ be an area minimizer of a closed surface subgroup $\Pi$. Let $\Tilde{\Sigma}$ be a lift of $\Sigma$ to the universal cover of $M$. As $\Sigma$ is essential, $\Tilde{\Sigma}$ must be a properly embedded disk. Let $H(s_1)$ be a lift of a component of $M\setminus M(s_1)$, and let $H(t)\subseteq H(s_1)$ ($t\geq s_1$) be the lift of the corresponding contained component of $M\setminus M(t)$. As by the same reasoning of the previous lemma, we see that $\Tilde{\Sigma}\cap H(s_1)$ embeds in $\Sigma$, it is sufficient to show
    \begin{equation}\label{equ_3.4_1}
        \area_h\big(\Tilde{\Sigma}\cap (H(s_1)\setminus H(s_2))\big) \geq \frac{\kappa}{1-\kappa} \big(\area_h(\Tilde{\Sigma}\cap H(s_2))\big).
    \end{equation}
    
    For $s_1\leq t\leq s_2$ let $\ell(t) : = \text{length}_h(\Tilde{\Sigma}\cap \partial H(t))$, and let $\hat{\ell} = \inf_{s_1\leq t\leq s_2} \ell(t)$. For $s_1\leq t\leq s_2$ so that $\Tilde{\Sigma}$ is transverse to $\partial H(t)$, each component of $\Tilde{\Sigma} \cap H(t)$ is a planar domain with finitely many Jordan curves in its boundary. Let $D(t)$ be a given component of $\Tilde{\Sigma} \cap H(t)$, and let $c(t)$ the outermost boundary component of $D(t)$. Hence $\area_h(D(t)) \leq 2c(t)$, as the area on the interior of $c(t)$ in $\Tilde{\Sigma}$ (which is larger than $\area_h(D(t))$) cannot be larger of the area of any disk filling $c(t)$, which we can get competitor as close to $2c(t)$ by filling in higher and higher vertical cylinders over $c(t)$. By taking addition over all possible components of $\Tilde{\Sigma} \cap H(t)$, we obtain
    \[
    \area_h(\Tilde{\Sigma}\cap H(s_2)) \leq 2\ell(t)
    \]
    for any $s_1\leq t\leq s_2$ so that $\Tilde{\Sigma}$ is transverse to $\partial H(t)$. In particular
    \[
    \area_h(\Tilde{\Sigma}\cap H(s_2)) \leq 2\hat{\ell},
    \]
    so without loss of generality we will assume $\hat{\ell}>0$.
    
    By coarea formula,
    \[
    \area_h\big(\Tilde{\Sigma}\cap (H(s_1)\setminus H(s_2))\big)\geq \int_{s_1}^{s_2}\ell(t)\frac{dt}{2}\geq  \frac{(s_2-s_1)}{2}\hat{\ell}.
    \]
    As $s_2 > s_1+ \frac{4\kappa}{1-\kappa}$ and $2\hat{\ell} \geq \area_h(\Tilde{\Sigma}\cap H(s_2))$, the desired inequality \eqref{equ_3.4_1} follows.
\end{proof}

\section{Proof of Theorem \ref{thm_finite_vol}}\label{section_thm_b}
In this section, we consider a metric $h$ on $M$ with sectional curvature $sec(h)\leq -1$, and we prove Theorem \ref{thm_finite_vol}. The key lemma used to derive the result is the following, which is analogous to Theorem 5.1 in \cite{Calegari-Marques-Neves} for closed hyperbolic $3$-manifolds.

\begin{Prop}\label{prop_area_comparison}
Suppose that $(M,h_0)$ is a hyperbolic $3$-manifold of finite volume, and let $h$ be a Riemannian metric on $M$ with $sec(h)\leq -1$.
Then given a sequence of surface subgroups $\Pi_i \in S_{\frac1i,\mu_{Leb}}(M)$, we have 
\begin{equation}\label{area}
    \underset{i\rightarrow\infty}{\lim\sup}\dfrac{\area_h(\Pi_i)}{\area_{h_0}(\Pi_i)}\leq 1.
\end{equation}
Furthermore, assume that $h$ is bilipschitz equivalent to $h_0$, and that there is a constant $k>1$ such that $sec(h)\geq -k^2$. Then the equality holds if and only if $h$ is hyperbolic and isometric to $h_0$.
\end{Prop}

In the following discussion, we assume that $S_i$ and $\Sigma_i$ are closed essential surfaces immersed in $M$ that minimize the area in the homotopy class corresponding to $\Pi_i$ with respect to the metrics $h_0$ and $h$, respectively. As argued in Section \ref{subsection_equidistribution}, for sufficiently large $i$, $S_i$ is the unique minimizer for $\Pi_i$.

\subsection{Proof of Proposition \ref{prop_area_comparison}}
The inequality follows immediately from the Gauss-Bonnet formula:
\begin{equation}\label{equ_Gauss}
    \area_h(\Sigma_i)=4\pi(g_i-1)+\int_{\Sigma_i} (sec(h)+1)\,dA_h-\frac{1}{2}\int_{\Sigma_i} |A|^2dA_h\leq 4\pi(g_i-1),
\end{equation}
where $g_i$ denotes the genus of the surface subgroup $\Pi_i$. 
On the other hand, by the second fundamental form estimate \eqref{equ_seppi}, we have $|A|^2_{L^\infty(S_i,h_0)}\rightarrow 0$ as $i\rightarrow \infty$, which implies \begin{equation}\label{equ_gauss_ratio}
    \lim_{i\rightarrow\infty} \dfrac{\area_{h_0}(S_i)}{4\pi(g_i-1)}=1.
\end{equation} 
The inequality follows suit.

If the equality of \eqref{area} holds, it yields that \begin{equation*}
    \lim_{i\rightarrow\infty}\frac{1}{\area_h(\Sigma_i)}\int_{\Sigma_i}\left(-(sec(h)+1)+\frac{1}{2}|A|^2\right)dA_h=0.
\end{equation*}
Let $\mathscr{C}$ be the set of all round circles in $S^2_\infty$, and define \begin{align*}
\mathscr{L}=&\{\gamma\in\mathscr{C}:\exists\phi_i\in F_i(\epsilon_i,R_i), \epsilon_i\rightarrow 0, R_i\rightarrow\infty,\text{ such that after passing to}\\
& \text{a subsequence, the limit set }\Lambda(\phi_i\Pi_i\phi_i^{-1})\text{ of $\phi_i\Pi_i\phi_i^{-1}$ converges to }\gamma\},
\end{align*} in which \begin{equation}\label{F_i}
    F_i(\epsilon,R)=\left\{\phi\in\Gamma:\int_{\phi(\Tilde{\Sigma}_i)\cap B_R(0)}\left(-(sec(h)+1)+\frac{1}{2}|A|^2\right)dA_h\leq\epsilon\right\}.
\end{equation}
It is not hard to see that $\mathscr{L}$ is closed and $\Gamma$-invariant. Due to Lemma 5.2 in \cite{Shah2} or Corollary A of \cite{Ratner}, almost every element in $\mathscr{C}$ has a dense $\Gamma$-orbit. 

Assume that $\mathscr{L}$ contains no element whose $\Gamma$-orbit is dense. Then, as shown in Theorem D in \cite{Shah} or Theorem B of \cite{Ratner}, for each $\gamma\in\mathscr{L}$, the unique totally geodesic disk $D(\gamma)$ in $\mathbb{H}^3$ with $\partial_\infty D(\gamma)=\gamma$ projects to an immersed surface in $M$, consisting of a finite union of connected components with finite area. Suppose that there are infinitely many such immersed totally geodesic surfaces corresponding to elements of $\mathscr{L}$. 
According to Corollary 1.5 of \cite{mozes_shah_1995}, any infinite sequence of immersed totally geodesic surfaces becomes dense in $M$. Therefore, we could choose an infinite sequence $\{\gamma_i\}$ in $\mathscr{L}$, and the limit of the orbit closures $\overline{\Gamma\gamma_i}$ would then be dense in $\mathscr{C}$. Since $\mathscr{L}$ is closed and $\Gamma$-invariant, it would follow that $\mathscr{L}=\mathscr{C}$. This means that almost every element in $\mathscr{L}$ has a dense $\Gamma$-orbit in $\mathscr{C}$, contradicting our assumption.
Let $\Delta\subset\mathbb{H}^3$ be a fundamental domain of $M$ whose boundary is transverse to both $\phi(\Tilde{S}_i)$ and $\phi(\Tilde{\Sigma}_i)$ for any $\phi\in \Gamma$.
Hence, only finitely many $\gamma\in\mathscr{L}$ have the property that the associated totally geodesic disk $D(\gamma)$ with $\partial_\infty D(\gamma)=\gamma$ intersects $\Delta$. We denote the union of these intersections by $\Delta_{\mathscr{L}}$.

Furthermore, building on the discussion in Theorem 6.1 of \cite{Calegari-Marques-Neves}, we establish the following result.
 \begin{align}\label{star}\tag{$\star$}
     &\text{For any compact subset $K\subset\mathbb{H}^3$ with non-empty interior, there exists $\gamma\in\mathscr{L}$,}\\
     &\text{such that the unique totally geodesic disk $D(\gamma)$ bounded by $\gamma$ intersects $K$. }\nonumber
 \end{align}

Indeed, let $\Gamma^{S_i}$ and $\Gamma^{S_i}(K)$ be the sets of $\phi\in\Gamma$ such that $\phi(\Tilde{S}_i)$ intersect $\Delta$ and $K$, respectively. Their projections, along with the projection of $\Gamma^{S_i}\cap F_i(\epsilon_i,R_i)$ in the set $\Gamma/\Pi_i:=\{\phi \Pi_i: \phi\in \Gamma\}$, are denoted by $\underline{\Gamma}^{S_i}$, $\underline{\Gamma}^{S_i}(K)$, and $\underline{\Gamma}^{S_i}(\epsilon_i,R_i)$, respectively. Consider the projection of $K$ in $M$, and denote its lift to $\mathcal{F}rM$ by $\mathcal{F}rK$.
Suppose that $f$ is a continuous function on $\mathcal{F}rM$ satisfying $0\leq f\leq 1$, with support in $\mathcal{F}rK$. Proposition 6.4 of \cite{Calegari-Marques-Neves} estimates $\#\underline{\Gamma}^{S_i}$ using the area of $S_i$. It provides a constant $c>0$, such that for sufficiently large $i$, 
$\#\underline{\Gamma}^{S_i}(K)/\#\underline{\Gamma}^{S_i}$
has a lower bound of $c\,\Omega_*\delta_{\phi_i}(f)$, where $\delta_{\phi_i}$ is the laminar measure associated with $S_i$. By Proposition \ref{prop_measure_convergence}, after passing to a subsequence, $\Omega_*\delta_{\phi_i}(f)$ converges to $\mu_{Leb}(f)$. As a result, $\#\underline{\Gamma}^{S_i}(K)/\#\underline{\Gamma}^{S_i}$ is bounded below away from zero. Furthermore, Proposition 6.5 of \cite{Calegari-Marques-Neves} indicates that $\#\underline{\Gamma}^{S_i}(\epsilon_i,R_i)/\#\underline{\Gamma}^{S_i}$ is close to one. Therefore, combining both results, we observe that $\underline{\Gamma}^{S_i}(K)\cap \underline{\Gamma}^{S_i}(\epsilon_i,R_i)$ is non-empty. We obtain $\phi_i\in \Gamma^{S_i}(K)\cap \Gamma^{S_i}(\epsilon_i,R_i)$.

Moreover, by Lemma \ref{lem:areaincusps}, the area of $\Sigma_i$ does not accumulate entirely in the cusp region. This implies that the limit set $\Lambda(\phi_i \Pi_i \phi_i^{-1})$ does not concentrate at a single point. Therefore, after passing to a subsequence, $\Lambda(\phi_i \Pi_i \phi_i^{-1})$ converges to a round circle $\gamma \in \mathscr{C}$. Consequently, we have $\gamma \in \mathscr{L}$. This implies \eqref{star}.

We choose a compact set $K$ within $\Delta\setminus\Delta_{\mathscr{L}}$ that has a non-empty interior. This ensures that $K$ does not intersect any such totally geodesic disks. This contradicts \eqref{star}, and thereby showing that $\mathscr{L}$ must contain at least one element whose $\Gamma$-orbit is dense. We summarize this conclusion in the following lemma.

\begin{Lemma}\label{lem_L=C}
There exists a round circle $\gamma\in\mathscr{L}$ such that $\Gamma\gamma$ is dense in $\mathscr{C}$. Moreover, the fact that $\mathscr{L}$ is closed and $\Gamma$-invariant implies the stronger conclusion that $\mathscr{L}=\mathscr{C}$. Therefore, by applying the results of \cite{Shah2} or \cite{Ratner} again, we conclude that almost every round circle in $\mathscr{L}$ has a dense $\Gamma$-orbit. 
\end{Lemma}

After proving Lemma \ref{lem_L=C}, we can choose an arbitrary round circle $\gamma\in\mathscr{L}$ that has a dense $\Gamma$-orbit, and we find $\phi_i\in F_i(\epsilon_i,R_i)$, $\epsilon_i\rightarrow 0$ and $R_i\rightarrow \infty$ as $i\rightarrow \infty$, such that the limit set $\Lambda(\phi_i\Pi_i\phi_i^{-1})$ converges to $\gamma$. Denote by $D_i, \Omega_i$ the lifts of $S_i, \Sigma_i$ to the universal cover $B^3$ of $M$ that are preserved by $\phi_i\Pi_i\phi_i^{-1}$. Due to the estimates of the second fundamental form \eqref{equ_seppi}, after passing to a subsequence, $D_i$ converges on compact sets to a totally geodesic disk $D\subset \mathbb{H}^3$.

We argue that $\partial_\infty D=\gamma$. Let $y$ be an arbitrary point in $\gamma$. 
Take a sequence $x_i\in D_i$ that converges to $x\in D$. Due to the convergence $\Lambda(\phi_i\Pi_i\phi_i^{-1})\to\gamma$, we can take a sequence $y_i\in \Lambda(\phi_i\Pi_i\phi_i^{-1})$ that converges to $y$. Let $\alpha_i$ be the geodesic arc in $\mathbb{H}^3$ connecting $x_i$ to $y_i$, and let $\beta_i$ be the geodesic arc in $D_i$ connecting $x_i$ to $y_i$. 
Because the geodesic curvature of $\beta_i$ in $\mathbb{H}^3$ is uniformly bounded by a small constant for sufficiently large $i$, there exists a uniform constant $r>0$ such that $\beta_i$ is contained in the $r$-tubular neighborhood of $\alpha_i$. 
Furthermore, since $D$ is totally geodesic,  both $\alpha_i$ and $\beta_i$ converge to the same geodesic arc contained in $D$, which connects $x$ to a point in $\partial_\infty D$. 
This shows that the limit $y$ of the sequence $y_i$ is in $\partial_\infty D$.
As a consequence, $\gamma\subset \partial_\infty D$. Therefore, we must have $\partial_\infty D=\gamma$ since $\partial_\infty D$ is a round circle.

We observe from \eqref{F_i} that \begin{equation}\label{integral}
    \lim_{i\rightarrow\infty} \int_{\Omega_i\cap B_{R_i}(0)} \left(-(sec(h)+1)+\frac{1}{2}|A|^2\right)dA_h=0.
\end{equation}

Next, we prove the following result. 

\begin{Lemma}\label{lem_convergence of Omega}
There exists a connected component $\Omega^0_i\subset \Omega_i\cap B_{R_i}(0)$, such that $\Omega^0_i$ is a disk, and after passing to a subsequence, $\Omega^0_i$ converges smoothly to a totally geodesic hyperbolic disk $\Omega$ with asymptotic boundary $\partial_\infty\Omega=\gamma$.
\end{Lemma}  

\begin{proof}
We explore the convex hulls in the same way as in Section 3 of \cite{Calegari-Marques-Neves}. In what follows, the convex hull of a closed curved $\alpha\subset S^2_\infty$ stands for the smallest (geodesically) convex set in $\overline{B^3}$ that contains $\alpha$. Note that by assumption, there exists $k>1$ such that the sectional curvature of $(M,h)$ satisfies $-k^2\leq sec(h)\leq -1$, while $h$ is bilipschitz to $h_0$. This ensures that Proposition 2.5.4 in \cite{Bowditch} and the proof of \cite[Proposition 3.2]{Calegari-Marques-Neves} apply to our setup. We state the version for our case below.

\begin{Lemma}\label{lemma_convex_hull}
    There is a constant $r_0=r_0(h)>0$, such that the Hausdorff distance between $C_h\left(\Lambda(\phi_i \Pi_i\phi_i^{-1})\right)$ and $C_{h_0}\left(\Lambda(\phi_i \Pi_i\phi_i^{-1})\right)$, which are the convex hulls of $\Lambda(\phi_i \Pi_i\phi_i^{-1})$ with respect to metrics $h$ and $h_0$, respectively, satisfies the following inequality. \begin{equation}\label{equ_4.4}
    d_{H,h}\left(C_{h_0}\left(\Lambda(\phi_i \Pi_i\phi_i^{-1})\right),C_{h}\left(\Lambda(\phi_i \Pi_i\phi_i^{-1})\right)\right)\leq r_0.
\end{equation} 
\end{Lemma}

We proceed with the proof of Lemma \ref{lem_convergence of Omega}. 

Consider the covering space $\Tilde{M}_i=\mathbb{H}^3/\phi_i \Pi_i\phi_i^{-1}$ of $M$. With respect to the induced metric of $h_0$, ($\Tilde{M}_i,h_0)$ is a quasi-Fuchsian manifold with $\pi_1(\Tilde{M}_i)\cong\pi_1(\Sigma_i)$. $\Sigma_i$ can be considered as a closed minimal surface in the complete manifold $(\Tilde{M}_i,h)$ with negative curvature. The convex core of $(\Tilde{M}_i,h)$ is defined as $C_h\left(\Lambda(\phi_i \Pi_i\phi_i^{-1})\right)/\phi_i \Pi_i\phi_i^{-1}$. In this setting, \cite[Proposition 3.3]{Calegari-Marques-Neves} shows that each $\Sigma_i$ is contained in the convex core of $(\Tilde{M}_i,h)$. This implies that
\begin{equation}\label{convexhull}
    \Omega_i\subset C_h\left(\Lambda(\phi_i \Pi_i\phi_i^{-1})\right).
\end{equation}
Let $D^r_i$ be the disk in $\mathbb{H}^3$ with the fixed signed distance $r$ to $D_i$.
By the computation in \cite{Uhlenbeck}, when $r>\tanh^{-1}\dfrac{|A|^2_{L^\infty(D_i,h_0)}}{2}$, the closed set enclosed by $D^r_i\cup D^{-r}_i\cup\Lambda(\phi_i \Pi_i\phi_i^{-1})$ in $\overline{\mathbb{H}^3}$ is strictly convex and it bounds inside the convex hull of $\Lambda(\phi_i \Pi_i\phi_i^{-1})$,
Therefore, \begin{equation*}
    d_{H,h_0}\left(C_{h_0}(\Lambda(\phi_i \Pi_i\phi_i^{-1}),D_i\right)\leq \tanh^{-1}\frac{|A|^2_{L^\infty(D_i,h_0)}}{2}.
\end{equation*}  
We can convert the Hausdorff distance with respect to $h_0$ into that of $h$ by adding a constant that depends only on $h$. Thus,
combining these estimates, we conclude that the Hausdorff distance between $D_i$ and $\Omega_i$ is uniformly bounded by some constant $R_0=R_0(h)>0$: \begin{equation*}
    d_{H,h}(\Omega_i,D_i)\leq R_0, \quad \forall i>>1.
\end{equation*}
Then, because of the convergence of $D_i$, there exists $R>0$, such that for sufficiently large $i$ and generic $r\geq R$, $\Omega_i$ intersects $B_{r}(0)$ by a union of circles. We can slightly perturb $R_i$ so that $\Omega_i\cap B_{R_i}(0)$ is a union of circles.

Let $\Omega_i^0$ be a component of $\Omega_i\cap B_{R_i}(0)$ intersecting $B_{R}(0)$.
We claim that it is a disk. Otherwise, if $\Omega_i^0$ were a planar region other than the disk, then we could find a larger ball $B_{R_i'}(0)$ with some $R_i'>R_i$ whose boundary met tangentially with $\Omega_i^0$ at some point. However, the convexity of $\partial B_{R_i'}(0)$ and the minimality of $\Omega_i^0$ contradict the maximum principle. Therefore, $\Omega_i^0$ is a disk provided that $i$ is large enough. 
Furthermore, the total curvature estimates based on \eqref{integral} imply that \begin{equation}\label{equ_omega}
\underset{i\rightarrow\infty}{\lim\sup}\left\{|sec(h(x))+1|+\frac{1}{2}|A(x)|^2:x\in\Omega_i^0\right\}=0.
\end{equation}
From the standard compactness theory for minimal surfaces with a uniform bound on the second fundamental form, after passing to a subsequence, $\Omega_i^0$ converges smoothly on compact sets to a minimal disk $\Omega$ in $(B^3,h)$. Moreover, by \eqref{equ_omega}, $\Omega$ is totally geodesic and has sectional curvature equal to $-1$. 

It remains to show that $\partial_\infty\Omega=\gamma$.  We will use a similar approach to the one previously used in proving that $\partial_\infty D=\gamma$. Let $y$ be an arbitrary point in $\gamma$. 
Take a sequence $x_i\in \Omega_i^0$ that converges to $x\in \Omega$, and a sequence $y_i\in \Lambda(\phi_i\Pi_i\phi_i^{-1})$ that converges to $y$. Let $\alpha_i$ be the geodesic arc in $(B^3,h)$ connecting $x_i$ to $y_i$, and let $\beta_i$ be the geodesic arc in $\Omega_i$ connecting $x_i$ to $y_i$. 
By \eqref{convexhull} and Proposition 2.5.4 in \cite{Bowditch}, there exists a constant $r=r(h)>0$, independent of $i$, such that $\beta_i$ is contained in the $r$-neighborhood of $\alpha_i$. 
Furthermore, since $\Omega$ is totally geodesic,  both $\alpha_i$ and $\beta_i$ converge to the same geodesic arc contained in $\Omega$, which connects $x$ to a point in $\partial_\infty \Omega$. 
This shows that $y$ is contained in $\partial_\infty \Omega$, and thus $\gamma\subset \partial_\infty \Omega$. Since $\partial_\infty\Omega$ is homeomorphic to a circle, it is identical to $\gamma$. 

\end{proof}

To complete the proof of the rigidity in Proposition \ref{prop_area_comparison}, we consider
\begin{align*}
    &\mathcal{F}r_2^D:=\{(x;e_1,e_2): x\in D, (e_1, e_2) \text{ orthonormal base of } \mathcal{F}r_2 D_x\},\\
    &\mathcal{F}r_2^\Omega:=\{(x;e_1,e_2): x\in D, (e_1, e_2) \text{ orthonormal base of } \mathcal{F}r_2 \Omega_x\}.
\end{align*}
Let $\mathcal{F}r_2^D(M)$ and $\mathcal{F}r_2^\Omega(M)$ be the projections of $\mathcal{F}r_2^D$ and $\mathcal{F}r_2^\Omega$ to the $2$-frames bundles of $M$ with respect to $h_0$ and $h$, denoted by $\mathcal{F}r_2M(h_0)$ and $\mathcal{F}r_2M(h)$, respectively. 

We define the Cheeger homeomorphism from $\mathcal{F}r_2M(h_0)$ to $\mathcal{F}r_2M(h)$ as described in \cite{Gromov2000}: we first define equivariant homeomorphisms between $\mathcal{F}r_2\mathbb{H}^3(h_0)$, $\mathcal{F}r_2\mathbb{H}^3(h)$ and $(S^2_\infty)_3$, the set ordered triples of pairwise distinct elements of $S^2_\infty$. Each point $(x;e_1,e_2)$ in $\mathcal{F}r_2\mathbb{H}^3(h_0)$ is uniquely and continuously determined by an ordered triple $(y_1,y_2,y_3)$ on $S^2_\infty$ of distinct elements, where $y_1, y_2$ are the backward and forward asymptotic endpoints of the geodesic (with respect to $h_0$) passing through $x$ with tangent $e_1$, while $y_3$ is the forward asymptotic endpoint of the geodesic (with respect to $h_0$) passing through $x$ with tangent $y_3$. Conversely, given $(y_1, y_2, y_3)$ an ordered triple in $S^2_\infty$, $x$ is (with respect to $h_0$) the orthogonal projection of $y_3$ to the geodesic going from $y_1$ to $y_3$, while $e_1, e_2$ are the unit tangent vectors at $x$ whose corresponding forward points at infinity are $y_2, y_3$ (with respect to $h_0$). As $h$ is complete and strictly negatively curved, we have the analogous correspondence between $\mathcal{F}r_2\mathbb{H}^3(h)$ and $(S^2_\infty)_3$.  Hence the map $\mathcal{F}r_2\mathbb{H}^3(h_0) \rightarrow \mathcal{F}r_2\mathbb{H}^3(h)$ is defined as the composition of homeomorphisms $$\mathcal{F}r_2\mathbb{H}^3(h_0)\rightarrow (S^2_\infty)_3 \rightarrow \mathcal{F}r_2\mathbb{H}^3(h).$$ As the correspondences with $(S^2)^3_\infty$ is equivariant by the geometric action of $\Gamma$ in $\mathbb{H}^3$ with respect to $h_0$ and $h$, it follows then that the homeomorphism is $\Gamma$ equivariant. Then we can pass to the quotient and define a homeomorphism between $\mathcal{F}r_2M(h_0)$ and $\mathcal{F}r_2M(h)$.
Although Gromov initially stated this construction for two closed manifolds $M$ and $N$ with isomorphic fundamental groups in \cite{Gromov2000}, as argued above, the Cheeger homeomorphism also extends naturally to finite volume manifolds $(M,h_0)$ and $(M,h)$.

In particular, since $D$ and $\Omega$ are totally geodesic disks with the same asymptotic boundary with respect to $h_0$ and $h$, respectively, the Cheeger homeomorphism maps $\mathcal{F}r_2^D(M)$ to $\mathcal{F}r_2^\Omega(M)$. 
By the results of Shah \cite{Shah} and Ratner \cite{Ratner}, $\mathcal{F}r_2^D(M)$ is dense in $\mathcal{F}r_2M(h_0)$. 
Therefore, $\mathcal{F}r_2^\Omega(M)$ is dense in $\mathcal{F}r_2M(h)$. 
It follows that for any $(x;e_1,e_2)\in \mathcal{F}r_2M(h)$, there exists a sequence $\{\psi_i\}\subset\Gamma$, such that the images $\psi_i(\Omega)$ converge to a totally geodesic hyperbolic disk in $(B^3,h)$, whose projection to $M$
has orthonormal basis $\{e_1,e_2\}$ at $x$.
Consequently, the metric $h$ on $M$ must have constant sectional curvature equal to $-1$ and thus it is isometric to $h_0$ by Mostow rigidity.

\subsection{Proof of Theorem \ref{thm_finite_vol}}
First, if a metric $h$ on $M$ has sectional curvature less than or equal to $-1$, then $\Pi\in S_{\mu_{Leb}}(M,\floor*{L},\epsilon)$ implies that $\area_h(\Pi)\leq 4\pi(L-1)$ because of the Gauss equation \eqref{equ_Gauss}. Thus, we have $\underline{E}(h)\geq 2= E_{\mu_{Leb}}(h_0)$.

Next, suppose $\underline{E}(h)=2$. Assume that there exists $\eta>0$, such that for any $L>0$ and any increasing sequence $\{k_i\}\subset\mathbb{N}$, the condition $\Pi\in S_{\mu_{Leb}}\big(M,\floor*{(1+\eta)L},\frac{1}{k_i}\big)$ must produce that $\area_h(\Pi)\leq 4\pi(L-1)$.
As a result,\begin{equation*}
    \underline{E}(h)\geq\underset{L\rightarrow\infty}{\lim\inf}\frac{\ln{\#S_{\mu_{Leb}}\big(M,\floor*{(1+\eta)L},\frac{1}{k_i}\big)}}{L\ln{L}}\geq 2(1+\eta),
\end{equation*}
which violates the assumption. Therefore, there exists an increasing sequence $\{k_i\}\subset\mathbb{N}$, a sequence of integers $\{g_i\}$ and $\Pi_i\in S_{\mu_{Leb}}\big(M,g_i,\frac{1}{k_i}\big)$, so that \begin{equation*}
    \area_h(\Pi_i)> 4\pi\Big(\Big(1-\frac{1}{i}\Big)g_i-1\Big).
\end{equation*}
From the above inequality and Proposition \ref{prop_area_comparison},\begin{equation*}
     1\geq\underset{i\rightarrow\infty}{\lim\sup}\dfrac{\area_h(\Pi_i)}{\area_{h_0}(\Pi_i)}\geq \underset{i\rightarrow\infty}{\lim\inf}\dfrac{\area_h(\Pi_i)}{\area_{h_0}(\Pi_i)}\geq \underset{i\rightarrow\infty}{\lim\inf}\frac{4\pi\left(\left(1-\frac{1}{i}\right)g_i-1\right)}{4\pi(g_i-1)}=1.
\end{equation*}
The equality holds if and only if the metric $h$ 
is isometric to $h_0$.

\section{Background of Ricci flow}\label{section_bkg_rf}
In this section, we will briefly review the tools used to prove Theorems \ref{Thm_sar>-6} and \ref{Thm_sar>-6_c0perturbation}.
\subsection{Normalized Ricci flow and Ricci-DeTurck flow}
The \emph{normalized Ricci flow} on $M$ is defined as \begin{equation}\label{RF}
    \frac{\partial h}{\partial t}=-2Ric(h)-4h.
\end{equation}
However, this evolution equation is only weakly parabolic. To achieve strict parabolicity, one considers the following DeTurck-modified version. Let $Sym^2(T^*M)$ be the space of smooth symmetric covariant $(0,2)$-tensors on $M$, and let $Sym^2_+(T^*M)$ be the subset of positive-definite tensors. Moreover, we denote by $\Omega^1(M):=\Gamma(T^*M)$ the space of differential 1-forms. 
Given a Riemannian metric $h$ on $M$, we use $\delta_h: Sym^2(T^*M)\rightarrow \Omega^1(M)$ to denote the map $\delta_hl=-h^{ij}\nabla_il_{jk}dx^k$. The formal adjoint for the $L^2$ product is denoted by $\delta_h^*:\Omega^1(M)\rightarrow Sym^2(T^*M)$.
Define a map $G: Sym^2_+(T^*M)\times Sym^2(T^*M)\rightarrow Sym^2(T^*M)$ by \begin{equation*}
    G(h,u)=\Big(u_{ij}-\frac{1}{2}h^{km}u_{km}h_{ij}\Big)dx^i\otimes dx^j.
\end{equation*}
And $P:Sym^2_+(T^*M)\times Sym^2_+(T^*M)\rightarrow Sym^2(T^*M)$ is defined by \begin{equation*}
    P_u(h)=-2\delta_h^*\left(u^{-1}\delta_h(G(h,u))\right).
\end{equation*}
Finally, the \emph{normalized Ricci-DeTurck flow} for \eqref{RF} is given by \begin{equation}\label{DRF}
    \frac{\partial h}{\partial t}=-2Ric(h)-4h-P_{h_0}(h),
\end{equation}
where we set the background metric $u$ to be the hyperbolic metric $h_0$ so that $h_0$ is a fixed point of \eqref{DRF}.
Notice that the right hand side is a strictly elliptic operator known as the \emph{DeTurck operator}, we denote it by $\mathcal{A}(h)$.

\subsection{Largest spectrum estimate}
In the subsequent section, we consider the linearization of \eqref{DRF} at $h_0$: \begin{equation*}
    \frac{\partial l}{\partial t}=\Delta_Ll-4l,
\end{equation*}
where $\Delta_L$ is the Lichnerowicz Laplacian, and in local coordinates, we have \begin{equation*}
    (\Delta_Ll)_{ij}=(\Delta l)_{ij}+2R_{iklj}l^{kl}-R^k_il_{kj}-R_j^kl_{ki}.
\end{equation*}
Denote by $A_{h_0}:Sym^2(T^*M)\rightarrow Sym^2(T^*M)$ the linear operator 
\begin{equation}\label{equ_def_A}
    A_{h_0}(l):=D\mathcal{A}(h)|_{h=h_0}(l)=\Delta_L l-4l.
\end{equation}
It is a self-adjoint operator, and strictly elliptic when acting on $l\in 
 Sym^2_c(T^*M)$, the space of symmetric covariant 2-tensors with compact support.
 
Next, we will see that the $L^2$-spectra of $A_{h_0}$ are negative and then proceed to estimate the largest spectrum.
Denote by $(\cdot,\cdot)$ the $L^2$-product on $Sym^2_c(T^*M)$.
 Since $R_i^j=-2\delta_i^j$, we have \begin{align}\label{equ_Al,l}
     (A_{h_0}(l),l) &= \int_M \langle \Delta l,l\rangle \dvol +2\int_M R_{iklj}l^{kl}l^{ij}\dvol\\\nonumber
     &=-\int_M\langle \nabla l,\nabla l\rangle \dvol +2\int_M \langle \mathcal{R}(l),l\rangle \dvol,
 \end{align}
using integration by parts, where $\mathcal{R}: Sym^2(T^*M)\rightarrow Sym^2(T^*M)$ is defined by $\langle \mathcal{R}(h),l\rangle =R_{iklj}h^{ij}l^{kl}$.
Moreover, define a covariant 3-tensor by $T_{ijk}:=\nabla_kl_{ij}-\nabla_il_{jk}$, then \begin{align}\label{equ_T}
    \Vert T\Vert ^2 &= \int_M \langle \nabla_k l_{ij}-\nabla_i l_{jk},\nabla^k l^{ij}-\nabla^i l^{jk}\rangle \dvol\\\nonumber
    &= 2\Vert \nabla l\Vert ^2-2\int_M \nabla_kl_{ij}\nabla^il^{jk}\dvol. 
\end{align}
For the second term, we integrate by parts and obtain that \begin{align*}
    -2\int_M \nabla_kl_{ij}\nabla^il^{jk}\dvol &= 2\int_M l_j^i\nabla_k\nabla_il^{jk}\dvol\\
    &= 2\int_M l_j^i(\nabla_i\nabla_kl^{jk}+R^j_{kip}l^{pk}+R^k_{kip}l^{jp})\dvol\\
    &= -2\Vert \delta l\Vert ^2+2\int_M l^i_j R_{ip}l^{jp}\dvol +2 \int_M l_j^i R^j_{kip}l^{pk}\dvol\\
    &= -2\Vert \delta l\Vert ^2-4\Vert l\Vert ^2-2\int_M\langle \mathcal{R}(l),l\rangle \dvol.
\end{align*}
Substituting this into \eqref{equ_T}, we obtain \begin{equation*}
    \Vert T\Vert ^2= 2\Vert \nabla l\Vert ^2-2\Vert \delta l\Vert ^2-4\Vert l\Vert ^2-2\int_M\langle \mathcal{R}(l),l\rangle \dvol.
\end{equation*}
Furthermore, when combined with \eqref{equ_Al,l}, it implies \begin{equation*}
    (A_{h_0}(l),l)=-\frac{1}{2}\Vert T\Vert ^2-\Vert \delta l\Vert ^2-2\Vert l\Vert ^2+\int_M\langle \mathcal{R}(l),l\rangle \dvol,
\end{equation*}
where \begin{equation*}
    \int_M\langle \mathcal{R}(l),l\rangle \dvol=\int_M -\left((h_0)_{ij}(h_0)_{kl}-(h_0)_{ik}(h_0)_{jl}\right)l^{ij}l^{kl}\dvol=-\Vert \tr_{h_0}(l)\Vert ^2+\Vert l\Vert ^2.
\end{equation*}
Thus we have \begin{equation}\label{equ_spectrum}
    (A_{h_0}(l),l)=-\frac{1}{2}\Vert T\Vert ^2-\Vert \delta l\Vert ^2-\Vert l\Vert ^2-\Vert \tr_{h_0}(l)\Vert ^2\leq -\Vert l\Vert ^2.
\end{equation}
Moreover, by \eqref{equ_Al,l} inequality \eqref{equ_spectrum} extends for the closure of $Sym^2_c(T^*(M))$ in the Sobolev space $W^{1,2}(T^*(M))$.

\subsection{Ricci flow with bubbling-off}\label{subsection_bubbling}
In this section, we review the definitions and notations related to Ricci flow with bubbling-off that will be useful in Section \ref{sec:proofS>-6}. For more details, readers are encouraged to consult the book by Bessi{\`e}res, Besson, Boileau, Maillot, and Porti \cite{BBB+10}. However, note that while \cite{BBB+10} and some of other works below discuss Ricci flow, their results can be applied to the normalized flow, as these flows related to one another by a time reparametrization and rescaling.

The construction of Ricci flow with this specific version of surgery on the cusped manifold $M$ was established by Bessi{\`e}res, Besson, and Maillot in \cite{Bessieres-Besson-Maillot}, under the assumption that the initial metric $h$ admits a \emph{cusp-like structure}. This means that the restriction of $h$ on each cusp satisfies the condition that $\lambda h-h_{cusp}$ approaches zero at infinity in the $C^k$-norm for each integer $k$, where $\lambda>0$ and $h_{cusp}=e^{-2s}h_{T_j}+ds^2$ is a hyperbolic metric on $T_j\times [0,\infty)$. Note that the hyperbolic metric $h_{cusp}$ is not unique, it varies based on different choices of flat metrics $h_{T_j}$ on $T_j$. The cusp-like structure ensures that the universal cover $(B^3,h)$ has bounded geometry, allowing the existence theorem of Ricci flow with surgery (Theorem 2.17, \cite{Bessieres-Besson-Maillot}) to apply, and thus making it possible to consider an equivalent version that passes to the quotient (Addendum 2.19, \cite{Bessieres-Besson-Maillot}).

Furthermore, their work examines the long-time behavior of the Ricci flow on $M$ starting from a metric $h(0)$ with a cusp-like structure. After a finite number of surgeries, as $t$ goes to infinity, the solution $h(t)$ converges smoothly to the hyperbolic metric $h_0$ on balls of radius $R$ for all $R>0$ (Theorem 1.2 of \cite{Bessieres-Besson-Maillot}). However, as indicated in the stability theorem (see Theorem 2.22 of \cite{Bessieres-Besson-Maillot}, and also Theorem \ref{thm_stability_cusp} below for a more general version), outside these balls, the cusp-like structure of $h(0)$ is preserved for all time. Therefore, if $h(0)$ is asymptotic to some $h_{cusp}$ different from the restriction of $h_0$ on the cusp, then the convergence cannot be global on $M$. 

It is worth noting that the proof of the stability theorem relies on a different construction of surgery. Since $M$ is both irreducible and lacks finite quotients of $S^3$ or $S^2\times S^1$, any surgery in $M$ splits off a 3-sphere and does not change the topology, the authors focused only on metric surgeries that change the metric on some 3-balls. This version of surgery is called \emph{Ricci flow with bubbling-off} (Definition \ref{Def_bubbling}). The main distinction from the usual Hamilton-Perelman surgery is that, the bubbling-off occurs before a singularity appears. Moreover, in addition to the surgery parameters $r$ and $\delta$, they introduced new \emph{associated cutoff parameters} $H$ and $\Theta$ to determine when the scalar curvature at one end of a neck is large enough to perform a bubbling-off. In particular, this construction of bubbling-off is essential in proving the stability of cusp-like structures at infinity.

The goal of this section is to extend the long-time existence and stability to asymptotically cusped metrics or order $\geq 2$. 
We will provide more details in Section \ref{subsubsection_Persistence and stability}.

\subsubsection{Definitions and notations}
\begin{Def}[Evolving metric (Definition 2.2.2, \cite{BBB+10})]
    Let $M$ be a $3$-manifold and $I\subset \mathbb{R}$ be an interval. An \emph{evolving metric} on $M$ is a map $t\mapsto h(t)$ from $I$ to the space of Riemannian metrics on $M$, then it is left continuous and has a right limit at each $t\in I$. We also define the following terms:
    \begin{itemize}
        \item If the map is $C^1$ in a neighborhood of $t\in I$, then $t$ is called a \emph{regular time}. Otherwise, it is \emph{singular}.

        \item 
        If, on a subset $M_0\times I_0\subset M\times I$, the map $t\mapsto h(t)|_{M_0}$ is $C^1$ at each $t\in I_0$, then $M_0\times I_0$ is \emph{unscathed}. Otherwise, it is \emph{scathed}. 
    \end{itemize}
\end{Def}

\begin{Def}[Ricci flow with bubbling-off (Definition 2.2.1, \cite{BBB+10})]
    Let $I\subset [0,\infty)$ be an interval, and let $h(t)$ be a piecewise $C^1$ evolving metric on $I$ that solves the normalized Ricci flow equation \eqref{RF} at all regular times. We say that $\{h(t)\}_{t\in I}$ is a \emph{Ricci flow with bubbling-off} if, for every singular time $t\in I$, the following conditions hold:
        \begin{equation*}
            \inf_M R(h_+(t))\geq \inf_M R(h(t))\quad \text{and}\quad h_+(t)\leq h(t),
        \end{equation*}
        where $h_+(t)$ denotes the right limit of $h(t)$.
\end{Def}

\begin{Def}[$\epsilon$-closeness, $\epsilon$-homothety (Definition 2.1.1, \cite{BBB+10})]\label{def_closeness}
    Let $U\subset M$ be an open subset, and let $h_0,h$ be Riemannian metrics on $U$. Assume $\epsilon>0$. \begin{itemize}
        \item We say that $h$ is \emph{$\epsilon$-close} to $h_0$ on $U$ if 
    \begin{equation*}
        \left\Vert h-h_0\right\Vert_{\floor*{\frac{1}{\epsilon}},U,h_0}:= \left(\sup_{x\in U} \sum_{k=0}^{\floor*{\frac{1}{\epsilon}}}|\nabla_{h_0}^k(h-h_0)(x)|_{h_0}^2\right)^{\frac{1}{2}}<\epsilon.
    \end{equation*}
    \item If there exists $\lambda>0$ such that $\lambda h$ is $\epsilon$-close to $h_0$ on $U$, we say that $h$ is \emph{$\epsilon$-homothetic} to $h_0$ on $U$. 

    \item Furthermore, a pointed manifold $(U,h,x)$ is said to be \emph{$\epsilon$-close} to $(U_0,h_0,x_0)$, if there exists a $C^{\ceil*{\frac{1}{\epsilon}}}$-diffeomorphism $\psi:(U_0,x_0)\rightarrow (U,x)$, such that the pullback metric $\psi^*(h)$ is $\epsilon$-close to $h_0$ on $U_0$. 
    \item If there exists $\lambda>0$ such that $(U,\lambda h,x)$ is $\epsilon$-close to $(U_0,h_0,x_0)$, we say that $(U,h,x)$ is \emph{$\epsilon$-homothetic} to $(U_0,h_0,x_0)$. 
    \end{itemize}
 \end{Def}

\begin{Def}[$\epsilon$-necks, $\epsilon$-caps (Definitions 3.1.1, 3.1.2, 4.2.6, 4.2.8, \cite{BBB+10})]
Let $\epsilon, C>0$.
\begin{itemize}
    \item An open subset $N\subset M$ is called an \emph{$\epsilon$-neck centered at $x$} if $(N,h,x)$ is $\epsilon$-homothetic to $\left(S^2\times \left(-\frac{1}{\epsilon},\frac{1}{\epsilon}\right),h_{cyl},(\ast,0)\right)$, where $h_{cyl}$ represents the standard metric with on $S^2\times \left(-\frac{1}{\epsilon},\frac{1}{\epsilon}\right)$ with constant scalar curvature 2. 

    \item An open subset $U\subset M$ is called an \emph{$\epsilon$-cap centered at $x$} if, $U$ can be written as $U=V\cup N$, where $V$ is a closed 3-ball, $N$ is an $\epsilon$-neck, and $\overline{N}\cap V=\partial V$, $x\in \text{Int} V$.

    \item An open subset $N\subset M$ is called a \emph{strong $\epsilon$-neck centered at $(x,t)$} if, there exists $Q>0$ such that $\left(N,\{h(t')\}_{t'\in [t-Q^{-1}, t]},x\right)$ is unscathed, and for the parabolic rescaling $\bar{h}(t'):=Qh(t+t'Q^{-1})$, $\left(N,\{\bar{h}(t')\}_{t'\in [-1,0]},x\right)$ is $\epsilon$-close to the cylindrical flow $\left(S^2\times (-\frac{1}{\epsilon},\frac{1}{\epsilon}),\{h_{cyl}(t')\}_{t'\in [-1,0]},(\ast,0)\right)$.

    \item An $\epsilon$-cap $U$ is called an \emph{$(\epsilon,C)$-cap centered at $x$} if $R(x)>0$ and there exists $r\in (C^{-1}R(x)^{-\frac{1}{2}}, CR(x)^{-\frac{1}{2}})$ so that the following properties hold on $U$. \begin{enumerate}[(i)]
        \item $\overline{B(x,r)}\subset U\subset B(x,2r)$.
        
        \item The scalar curvature function restricted on $U$ takes values in a compact subinterval of $(C^{-1}R(x),CR(x))$.
        
        \item \begin{equation*}
            \vol(U)>C^{-1}R(x)^{-\frac{3}{2}}.
        \end{equation*}
        Additionally, if on $B(y,s)\subset U$, one has $|Rm|\leq s^{-2}$, then \begin{equation*}
            \vol(B(y,s))>C^{-1}s^3.
        \end{equation*}

        \item \begin{equation*}
            |\nabla R|<CR^{\frac{3}{2}}.
        \end{equation*}
        \item \begin{equation*}
            |\Delta R+2|Ric|^2|<CR^2.
        \end{equation*}
        \item \begin{equation*}
            |\nabla Rm|<C|Rm|^{\frac{3}{2}}.
        \end{equation*}
    \end{enumerate}

\end{itemize}
\end{Def}
\begin{Rem}
    Given $\epsilon>0$, there exists $C=C(\epsilon)>0$, such that a strong $\epsilon$-neck satisfies properties (i)-(vi) for all time. 

    For (v), if $h(t)$ solves the normalized Ricci flow equation \eqref{RF}, by the evolution equation \begin{equation*}
        \frac{\partial R}{\partial t}=\Delta R+2|Ric|^2+4R,
    \end{equation*} 
    we have \begin{equation}\label{equ_rem_cap}
        \left|\frac{\partial R}{\partial t}\right|< CR^2+4|R|.
    \end{equation}
\end{Rem}

\begin{Def}[Canonical neighborhood (Definitions 4.2.10, 5.1.2, \cite{BBB+10})]\label{Def_CN}
\noindent
    \begin{itemize}
        \item A point $(x,t)$ admits an \emph{$(\epsilon,C)$-canonical neighborhood} if $x$ is the center of a strong $\epsilon$-neck or an $(\epsilon,C)$-cap that satisfies (i)-(vi) for all time.

        \item Let $r>0$, and let $\epsilon_0, C_0$ be the constants in Definition 3.2.1 and Definition 5.1.1 of \cite{BBB+10}.
        The evolving metric $h(t)$ on $M$ satisfies the \emph{Canonical Neighborhood Property $(CN)_r$} if, for each $(x,t)$, when $R(x,t)\geq r^{-2}$, the point $(x,t)$ is the center of an $(\epsilon_0,C_0)$-canonical neighborhood.
    \end{itemize}
\end{Def}

Consider the positive decreasing function $\phi_t(s)$ defined in Remark 4.4.2 of \cite{BBB+10}, such that $\frac{\phi_t(s)}{s}\rightarrow 0$ as $s\rightarrow\infty$.

\begin{Def}[Curvature pinched toward positive (Definition 4.4.3, \cite{BBB+10})]\label{Def_pinching}
    The evolving metric $h(t)$ is said to have \emph{curvature pinched toward positive} if \begin{equation*}
        R(x,t)\geq -6, \quad Rm(x,t)\geq -\phi_t(R(x,t)).
    \end{equation*}
\end{Def}

The definitions above enable us to define the parameters $r,\delta, H$ and $\Theta$ for bubbling-off, thereby introducing the concept of $(r,\delta)$-bubbling-off. 

\begin{Th}[Cutoff parameters (Theorem 5.2.4, Definition 5.2.5, \cite{BBB+10})]
    For any $r,\delta>0$, there exist $H\in (0,\delta r)$ and $D>10$ such that the following holds.
    If $\{h(t)\}_{t\in I}$ is a Ricci flow with bubbling-off on $M$ with curvature pinched toward positive and satisfies the Canonical Neighborhood Property $(CN)_r$, then: 

    Suppose $x,y,z\in M$ and $t\in I$ with 
    \begin{equation*}
        R(x,t)\leq 2r^{-2},\quad R(y,t)=H^{-2},\quad R(z,t)\geq DH^{-2},
    \end{equation*}
    and $y$ lies on the $h(t)$-geodesic segment connecting $x$ to $z$. Then $(y,t)$ is the center of a strong $\delta$-neck. 

    We refer to $r<10^{-3}$ and $\delta<\min (\epsilon_0,\delta_0)$, where $\delta_0$ is determined by Theorem 5.2.2 of \cite{BBB+10}, as the \emph{surgery parameters}. The quantities $H=H(r,\delta)$ and $\Theta=\Theta(r,\delta):=2D(r,\delta) H(r,\delta)^{-2}$ are called the \emph{associated cutoff parameters}. 
\end{Th}

\begin{Def}[$\delta$-almost standard cap (Definition 5.2.3, \cite{BBB+10})]
    Choose a constant $\delta\in (0, \min (\epsilon_0,\delta_0))$, and let $\delta'=\delta'(\delta)$ be the function determined by Theorem 5.2.2 of \cite{BBB+10}, which tends to zero as $\delta\rightarrow 0$. 
    Let $U$ be an open subset of $M$, $V\subset U$ be a compact subset, $p\in \text{Int} V$, $y\in \partial V$.
    The 4-tuple $(U,V,p,y)$ is called a \emph{$\delta$-almost standard cap} if there is a $\delta'$-isometry $\psi:B(p_0,5+\frac{1}{\delta})\rightarrow (U,R(y)h)$, which maps $p_0$ to $p$ and $B(p_0,5)$ to $\text{Int} V$.
\end{Def}

 Finally, we provide the core definition. 
\begin{Def}[Ricci flow with $(r,\delta)$-bubbling-off (Definition 5.2.8, \cite{BBB+10})]\label{Def_bubbling}
    Fix the surgery parameters $r,\delta$, and let $h,\Theta$ be the associated cutoff parameters. Consider an interval $I \subset [0,\infty)$, and let $\{h(t)\}_{t\in I}$ represent a Ricci flow with bubbling-off on $M$.
    
    We say that $\{h(t)\}_{t\in I}$ is a \emph{Ricci flow with $(r,\delta)$-bubbling-off} if it meets the following conditions.
    \begin{enumerate}
        \item $h(t)$ has curvature pinched toward positive and satisfies $R(x,t)\leq \Theta$ for all $(x,t)\in M\times I$.
        
        \item For every singular time $t\in I$, $h_+(t)$ is \emph{obtained from $h(t)$ by $(r,\delta)$-surgery at time $t$}. This means \begin{enumerate}
            \item for every $x\in M$ where $h_+(x,t)\neq h(x,t)$, there exists a $\delta$-almost standard cap $(U,V,p,y)$ with respect to $h_+(t)$ such that 
            \begin{enumerate}
                \item $x\in \text{Int} V$,
                \item $R(y,t)=H^{-2}$,
                \item $(y,t)$ is the center of a strong $\delta$-neck,
                \item $h_+(t)<h(t)$ on $\text{Int} V$.
            \end{enumerate}

            \item the following (in)equalities hold:
            \begin{equation*}
               \sup_M R(h(t))=\Theta \quad \text{and}\quad \sup_M R(h_+(t))\leq \frac{\Theta}{2}.
            \end{equation*}
        \end{enumerate}

        \item $h(t)$ satisfies the Canonical Neighborhood Property $(CN)_r$.   
    \end{enumerate}
\end{Def}

\begin{Rem}[Short-time existence of Ricci flow with bubbling-off] \label{Rem_parameters}
Let $M$ be a hyperbolic $3$-manifold of finite volume. 
In \cite{Bessieres-Besson-Maillot}, Bessi{\`e}res-Besson-Maillot investigated the existence of Ricci flow with bubbling-off starting from a cusp-like metric on $M$ (there exist a hyperbolic metric $h_{cusp}$ on the cusp and $\lambda>0$, such that $\lambda h-h_{cusp}$ approaches zero at infinity in the $C^k$-norm for each integer $k$). These metrics ensure that the universal cover of $M$ has a bounded geometry, allowing the existence of Ricci flow with bubbling-off on the universal cover to be transferred to the quotient manifold $M$.

We can generalize the setting to asymptotically cusped metrics of order $k\geq 2$, that is, metrics $h$ such that $\lambda h-h_{cusp}$ tends to zero at infinity in $C^2$. Under this assumption, there exists a compact set $K\subset M$ such that the sectional curvature is negative on the thin part $M\setminus K$. By the proof of the Hadamard theorem, the universal cover $\Tilde{M}$ of $M$, equipped with the lifted metric from $h$, has a uniform positive lower bound on the injectivity radius. Therefore, $\Tilde{M}$ has bounded geometry. Applying Addendum 2.19 from \cite{Bessieres-Besson-Maillot}, we obtain the existence of Ricci flow with bubbling-off on $M$, starting from $h$ and defined on a short time interval $[0,T]$.

Moreover, we can choose the parameters to be piecewise constant. In fact, there exist a partition $0=t_0<t_1<\cdots <t_{N-1}=T$ and decreasing sequences of positive numbers $r_j, \delta_j$, such that $r(t)=r_j$ and $\delta(t)=\delta_j$ on $(t_j,t_{j+1}]$. Given that $h(0)$ is an asymptotically cusped metric of order $k\geq 2$ on $M$, there exists a Ricci flow with $(r(t),\delta(t))$-bubbling-off for $t\in [0,T]$.
\end{Rem}

\subsubsection{Stability of asymptotically cusped metrics}\label{subsubsection_Persistence and stability}
Let $(M,h_0)$ be a finite-volume hyperbolic $3$-manifold, and let $\mathcal{C}:=\cup_j T_j\times (0,\infty)$ denote the cusp region. There are hyperbolic metrics on $\mathcal{C}$ that differ from the restriction of $h_0$ on $\mathcal{C}$, given by $h_{cusp}=e^{-s}h_{T_j}+ds^2$, where $h_{T_j}$ stands for a flat metric on the torus.

The following result generalizes Theorem 2.22 of \cite{Bessieres-Besson-Maillot}, which addresses the stability of cusp-like metrics on the cusp, to asymptotically cusped metrics of any order $k\geq 2$. The proof proceeds in a similar manner, and for completeness, we include it below.

\begin{Th}[Stability of asymptotically cusped metrics]\label{thm_stability_cusp}
Let $h(0)$ be an asymptotically cusped metric on $M$ of order $k\geq 2$. Then there exists a normalized Ricci flow with bubbling-off $h(t)$ on $M$, defined for all $t\in [0,\infty)$, starting at $h(0)$.

    Moreover, assume that $\Vert Rm(h(0))\Vert_{C^{k-1}(M)}<\infty$. Then there is a factor $\lambda(t)>0$, such that $\lambda(t)h(t)-h_{cusp}$ goes to zero at infinity in the cuspidal end in $C^k$ uniformly for $t\in [0,\infty)$. This means that $h(t)$ remains asymptotic to the same hyperbolic metric on the cusp for all time. 
\end{Th}

To prove the theorem, we need the following lemma, which is analogous to Theorem 8.1.3 in \cite{BBB+10}. The key difference is that their result measures the distance between two metrics using the notion of $\epsilon$-closeness defined in Definition \ref{def_closeness}, whereas we use the $C^k$ norm for a fixed integer $k$. Additionally, while their theorem addresses the persistence of a flow $h(t)$ relative to an arbitrary model flow $\bar{h}(t)$ with bounded curvature, we restrict our attention to the case where $\bar{h}(t)=h_{cusp}$. 

\begin{Lemma}\label{thm_persistence}
  Given an integer $k\in\mathbb{N}$ and $D, T,K>0$. There exists a constant $d=d(k, D, T,K)\leq D$ such that the following holds. Let $h(t)$ be a normalized Ricci flow defined on $M\times [0,T]$, which is unscathed on $\mathcal{C}\times [0,T]$. Consider a base point $x_0\in \mathcal{C}$ such that the ball $B(x_0,\frac1d)\subset \mathcal{C}$ is relatively compact.
  Suppose that
    \begin{enumerate}[(P1)]
    \item  \begin{equation*}
      \Vert Rm(h(0))\Vert_{C^{\max(k-1,0)}(\mathcal{C})}\leq K,
    \end{equation*}
    \item  \begin{equation*}
      \Vert Rm(h(t))\Vert_{C^{0}(\mathcal{C})}\leq K\quad \forall t\in [0,T], 
    \end{equation*}
    \item 
    \begin{equation*}
        \Vert h(0)-h_{husp}\Vert_{C^k(B(x_0,\frac1d))}\leq d.
    \end{equation*}
    \end{enumerate}
    Then 
    \begin{equation*}
        \Vert h(t)-h_{cusp}\Vert _{C^k(B(x_0,\frac1D))}< D\quad \forall t\in [0,T].
    \end{equation*}
\end{Lemma}

\begin{Rem}
   For a general model flow $\bar{h}(t)$, the persistence of $h(t)$ relative to $\bar{h}(t)$ may hold only on a finite time interval $[0,T]$. For example, an arbitrary large metric ball in the standard cylinder can be approximated by an almost cylindrical ball in the cigar soliton (Remark 8.1.4, \cite{BBB+10}). Consequently, in the proof of the stability theorem, we apply the lemma only over finite time intervals and proceed by induction.
\end{Rem}

\begin{proof}[Proof of Lemma \ref{thm_persistence}]
    Suppose by contradiction that there exist a sequence of normalized Ricci flows $g_n(t)$ defined on $M\times [0,T]$, a sequence $d_n\rightarrow 0$ as $n\rightarrow\infty$, and a sequence of points $x_n\in \mathcal{C}_{\frac{1}{d_n}}:=M\setminus M(\frac{1}{d_n})=\cup_jT_j\times (\frac{1}{d_n},\infty)$ such that $B(x_0,\frac1D)\subset B(x_n,\frac{1}{d_n})$ for each $n\in\mathbb{N}$, and the following conditions hold. 
\begin{enumerate}[(P1)]
         \item  \begin{equation*}
       \Vert Rm(g_n(0))\Vert_{C^{\max(k-1,0)}(\mathcal{C})}\leq K, 
    \end{equation*}
     \item  \begin{equation*}
      \Vert Rm(g_n(t))\Vert_{C^{0}(\mathcal{C})}\leq K \quad \forall t\in [0,T], 
    \end{equation*}
         \item 
         \begin{equation*}
            \Vert g_n(0)-h_{cusp}\Vert _{C^k(B(x_n,\frac{1}{d_n}))}\leq d_n.
        \end{equation*}
\end{enumerate}
    Moreover, there exists $t_n\in [0,T]$ such that 
    \begin{equation*}
        \Vert g_n(t_n)-h_{cusp}\Vert _{C^k(B(x_0,\frac1D))}\geq D.
    \end{equation*}
    We also assume that $t_n$ is the minimum time for this property.

Applying Shi's local derivative estimates \cite{Shi}---specifically, the stronger version stated in \cite[Theorem 3.29]{MorganTian2007}---and using (P1) and (P2) (in fact, only the bounds on $B(x_0,\frac2D)$ are required), 
we obtain a constant $K_m>0$ depending on $k,m,D,T,K$, such that
   \begin{equation}\label{equ_cond_local_persist}
      \Vert \nabla^mRm(g_n(t))\Vert _{C^0(B(x_0,\frac1D))}\leq K_mt^{-\frac{\max(m-k+1,0)}{2}} \quad \forall t\in [0,T].
   \end{equation}
   According to the proof of Lemma 8.2.1 of \cite{BBB+10}, for each integer $m\geq 1$, the pointwise norm $|\partial_t\nabla_{h_{cusp}}^mg_n(t)|_{h_{cusp}}$ is bounded by a constant depending on $\Vert \nabla^m Rm(g_n(t))\Vert_{C^0}$ for $0\leq m\leq k$, where the leading term is linear in $\nabla^k Ric(g_n(t))$.
   It then follows from \eqref{equ_cond_local_persist} that \begin{align}\label{equ_821}
       &\Vert g_n(t)-h_{cusp}\Vert _{C^{k}(B(x_0,\frac1D))}\\\nonumber
       \leq &\Vert g_n(0)-h_{cusp}\Vert _{C^{k}(B(x_0,\frac1D))}+\int_0^t \Vert \partial_s\nabla_{h_{cusp}}^mg_n(s)\Vert_{C^0(B(x_0,\frac1D))}ds\\\nonumber
       \leq & d_n+C_k\int_0^t\left(C(K,K_1,\cdots,K_{k-1})+K_ks^{-\frac12}\right)ds\\\nonumber
       =& d_n+C_k\left(C(K,K_1,\cdots,K_{k-1})t+2K_kt^{\frac12}\right).
   \end{align}
Therefore, there exist constants $d_{loc},T_{loc}>0$ depending on $k,D,K_i$ for $0\leq i\leq k$, and hence depending on $k,D,T,K$, such that if $n$ is sufficiently large so that $d_n\leq d_{loc}$, and if $T_{loc}$ is sufficiently small, we have
\begin{equation*}
        \Vert g_n(t)-h_{cusp}\Vert _{C^k(B(x_0,\frac1D))}<D\quad \forall t\in [0,T_{loc}].
    \end{equation*}
   Hence, we conclude that $t_n>T_{loc}$. This means that the explosion time $t_n$ cannot be too small.

    Define $t_\infty=\lim_{n\rightarrow\infty}t_n$, we have $t_\infty\in (T_{loc},T]$. Since $B(x_n,\frac{1}{d_n})$ shares a common marked point $x_0$, and the initial metrics $g_n(0)$ have a uniform positive lower bound on their injectivity radius at $x_0$, applying Hamilton's compactness theorem (Theorem 1.2, \cite{Hamilton}), we conclude that after passing to a subsequence, the normalized Ricci flows $g_n$ on $B(x_n,\frac{1}{d_n})\times [0,t_n)$ converge uniformly on compact sets in $C^{k}$ to a normalized Ricci flow $g_\infty$ defined on $ \mathcal{C}\times [0,t_\infty)$.
    
    Furthermore, by Chen-Zhu's uniqueness theorem \cite{Chen-Zhu}, the limit $g_\infty(t)$ is exactly $h_{cusp}$ for all $t\in [0,t_\infty)$. Therefore, after passing to a subsequence, $g_n(t_\infty-\frac{T_{loc}}{2})$ converges to $h_{cusp}$ on compact sets of $B(x_n,\frac{1}{d_n})$ in $C^k$. In particular, for sufficiently large $n$, we have \begin{equation*}
        \Big\Vert g_n\Big(t_\infty-\frac{T_{loc}}{2}\Big)-h_{cusp}\Big\Vert _{C^{k}(B(x_0,\frac{1}{D}))}\leq d_{loc}.
    \end{equation*} 
   Moreover, since the derivative estimate \eqref{equ_cond_local_persist} holds for $t\in[t_\infty-\frac{T_{loc}}{2},T]$, we can apply the estimate \eqref{equ_821} on $[t_\infty-\frac{T_{loc}}{2},\min(t_\infty+\frac{T_{loc}}{2},T)]$. Hence, using $t_\infty-\frac{T_{loc}}{2}>\frac{T_{loc}}{2}$,
   \begin{align*}
              & \left\Vert g_n(t)-h_{cusp}\right\Vert _{C^{k}(B(x_0,\frac{1}{D}))}\\
              \leq& \Bigg\{
        \Big\Vert g_n\Big(t_\infty-\frac{T_{loc}}{2}\Big)-h_{cusp}\Big\Vert _{C^{k}(B(x_0,\frac{1}{D}))}\\
        &+C_k\left(C(K,K_1,\cdots,K_{k-1})+K_k\left(t_\infty-\frac{T_{loc}}{2}\right)^{-\frac12}\right)\left(t-\left(t_\infty-\frac{T_{loc}}{2}\right)\right)\Bigg\}\\
              <& C_k\left(d_{loc}+C(K,K_1,\cdots,K_{k-1})T_{loc}+K_k(2T_{loc})^{\frac12}\right)\\
              <&D,\quad t\in \left[t_\infty-\frac{T_{loc}}{2},\min\left(t_\infty+\frac{T_{loc}}{2},T\right)\right].
   \end{align*}
    Therefore, the argument implies that, the first blow-up time $t_n$ can be extended to $\min(t_\infty+\frac{T_{loc}}{2},T)>t_n$, which contradicts the minimality of $t_n$.
\end{proof}

\begin{proof}[Proof of Theorem \ref{thm_stability_cusp}]
By assumption, $h(0)$ is an asymptotically cusped metric on $M$ of order $k$. Recall that this means that there is a constant $\lambda>0$ such that, on each cusp, the restriction of $h(0)$ satisfies the condition that $\lambda h(0)-h_{cusp}$ tends to zero at infinity in the $C^k$-norm. For simplicity, we assume $\lambda=1$ and show the theorem for $\lambda(t)=1$ for $t\in [0,\infty)$. Since $k\geq 2$, there exists $s\geq 0$ such that the scalar curvature satisfies $R(h(0))< 0$ on $\mathcal{C}_s=\cup_jT_j\times(s,\infty)$.

According to Remark \ref{Rem_parameters}, given an initial asymptotically cusped metric $h(0)$ of order $k\geq 2$, there exists a normalized Ricci flow with bubbling-off $h(t)$ on $M$, defined on a short time interval. Then by Perelman's proof of geometrization \cite{Perelman}, the flow $h(t)$ is defined for all time $t\in [0,\infty)$.
By Section 3 of \cite{Bessieres-Besson-Maillot}, each surgery reduces the volume of the manifold by at least a fixed amount, therefore only finitely many surgeries can occur. Let $T$ denote a time after all surgeries have taken place.
Moreover, using the surgery parameter $r(t)=r_j$ on  $(t_j,t_{j+1}]\subset [0,T]$, chosen in Remark \ref{Rem_parameters} and the constant $C_0$ in Definition \ref{Def_CN}, we define \begin{equation*}
    \sigma:=\frac{1}{2C_0r_{N-1}^{-2}+4}\leq \frac{1}{2C_0r_{j}^{-2}+4},\quad \forall j=0,\cdots, N-1.
\end{equation*}
This number is sufficiently small in this context, so that $h(t)$ cannot develop a singularity on a cusp within time $\sigma$. Indeed, if the scalar curvature explodes too fast, there are $t',t''\in (t_j,t_{j+1}]$ and $x\in \mathcal{C}_s$, where $s\geq 0$, such that \begin{equation*}
    0<t''-t'<\sigma,\quad R(x,t')\leq 0,\quad R(x,t'')= 2r_j^{-2}, \quad |R(x,t)|\leq 2r_j^{-2},\,\forall t\in (t',t'').
\end{equation*}
Then there exists $\tau\in (t',t'')$ with \begin{equation*}
   \left|\frac{\partial R(x,t)}{\partial t}\Big|_{t=\tau}\right|> \frac{2r_j^{-2}}{\sigma}\geq 4C_0r_j^{-4}+8r_j^{-2}\geq C_0R(x,\tau)^2+4|R(x,\tau)|.
\end{equation*}
However, it contradicts equation \eqref{equ_rem_cap} in the $(CN)_r$ condition.
Consequently, for any $t\in (t',t'')$ and any $x\in\mathcal{C}_s$, we have \begin{equation*}
    R(x,t)\leq 2r_j^{-2}<<H_j^{-2},
\end{equation*}
where $H_j^{-2}$ is the associated parameter determined by $r_j$ and $\delta_j$ on the interval $(t_j,t_{j+1}]$. 
According to Definition \ref{Def_bubbling}, the bubbling-off only occurs on a $\delta$-almost standard cap whose curvature is comparable to $H_j^{-2}$. Therefore, it is disjoint from the thin part $\mathcal{C}_s$. On $\mathcal{C}_s\times [t_j,t_{j+1}]$, the scalar curvature is uniformly bounded above. Due to the pinching assumption in Definition \ref{Def_pinching}, the curvature tensor $Rm$ is bounded below by a negative number. Moreover, $|Ric|$ cannot be too large. Otherwise, if $K_{12}+K_{13}$ were very large, the upper bound on $R$ would force $K_{23}$ to be very negative, contradicting the lower bound on $Rm$. This shows that $|Rm|$ must be uniformly bounded.

Since the cusp cannot be contained in a 3-ball where the surgery is performed, we conclude the following lemma. 

\begin{Lemma}\label{lemma_2.23}
Given $s\geq 0$. Suppose that $h(t)$ is unscathed on $\mathcal{C}_s\times [0,t]$ and has scalar curvature $R\leq 0$ there, then it is unscathed on $\mathcal{C}_s\times [0,t+\sigma]$ and on which $|Rm|$ is uniformly bounded. 
\end{Lemma}

Next, fix any $D>0$ and consider the time interval $[0,\sigma]$. Let $d_1$ be the constant arising from Lemma \ref{thm_persistence}, which depends on $k,D,\sigma$, $\Vert Rm(h(0))\Vert_{C^{k-1}(M)}$ and $\Vert Rm\Vert_{C^0(\mathcal{C}_s\times[0,\sigma])}$. Since $h(0)$ is asymptotically cusped of order $k$, we can find $s_0>0$ large enough so that $$\Vert h(0)-h_{cusp}\Vert_{C^k(B(x_0,\frac1d))}<d_1$$ 
for each $B(x_0,\frac1d)\subset \mathcal{C}_{s_0}$, where $x_0$ is a base point deep in the cusp. 
 Lemma \ref{thm_persistence} applies to the parabolic neighborhood $B(x_0,\frac1D)\times [0,\sigma]$ and implies that $$\Vert h(t)-h_{cusp}\Vert_{C^k\left(B(x_0,\frac1D)\times [0,\sigma]\right)}<D.$$
In particular, $h(t)$ is unscathed on $\mathcal{C}_{s_0+\frac{1}{d_1}-\frac1D}\times [0,\sigma]$ and has scalar curvature $R\leq 0$ (since $k\geq 2$). This allows us to apply Lemma \ref{lemma_2.23} once again with $s=s_0+\frac{1}{d_1}-\frac1D$ and $t=\sigma$, and then apply Lemma \ref{thm_persistence} to the time interval $[0,2\sigma]$.
By iterating the above process for $n:=\ceil*{\frac{T}{\sigma}}$ times, we obtain 
$$\Vert h(t)-h_{cusp}\Vert_{C^k\left(\mathcal{C}_{s_0+\frac{1}{d_1}+\cdots +\frac{1}{d_n}-\frac nD}\times [0,T]\right)}<D.$$

Furthermore, after the post-surgery time $T$, $h(t)$ remains unscathed on $M$ for all $t\geq T$. Then by \cite[Theorem 1.1]{Shi}, there exists a constant $K>0$ depending on $T$ and $\Vert Rm(T)\Vert_{C^0(M)}$, such that \begin{equation*}
       \Vert Rm(h(t))\Vert _{C^0(M)}\leq K \quad \forall t\in[T,2T].
   \end{equation*} 
   Additionally, $\Vert Rm(h(T))\Vert_{C^{k-1}(M)}$ uniformly bounded (Consider a covering of $M$ by a sequence of balls of fixed radius $r$. Then, by applying Shi's local derivative estimates on each ball---repeating the approach used in the proof of Lemma \ref{thm_persistence}---the uniform bound follows). Because $h(T)$ is asymptotically cusped, conditions (P1)-(P3) hold, we can apply the lemma to $[T,2T]$, and then repeatedly to $[nT,(n+1)T]$ for each $n\in\mathbb{N}$.

\end{proof}

\subsection{Stability for normalized Ricci-DeTurck flow}
In this section, we review the stability result associated with the normalized Ricci-DeTurck flow. It is shown in \cite{Bamler} that, under $C^0$ perturbations of the hyperbolic metric $h_0$, the corresponding flow exists for all time and remains close to $h_0$. The following result is deduced from \cite[Theorem 1.1]{Bamler} in \cite[Theorem 2.1]{Jiang-VargasPallete_RF}. 

\begin{Th}[Stability under $C^0$ perturbations, \cite{Bamler}, \cite{Jiang-VargasPallete_RF}]\label{thm_stability_C^0}
Let $(M,h_0)$ be a hyperbolic $3$-manifold of finite volume. There is a constant $d_0$, such that if a metric $h(0)$ satisfies \begin{equation*}
        \Vert h(0)-h_0\Vert _{C^0(M)}\leq d_0,
    \end{equation*}
then the normalized Ricci-DeTurck flow $h(t)$ starting from $h(0)$ exists for all time. 

Furthermore, given $k\in\mathbb{N}$. For any $D>0$, there exists $d=d(M, h_0,D,k)\leq \min\{d_0,D\}$ with the following property.
\begin{equation*}
        \Vert h(0)-h_0\Vert _{C^0(M)}\leq d.
    \end{equation*}
Then \begin{equation*}
        \Vert h(t)-h_0\Vert _{C^k(M)}\leq D\quad \forall t\in [1,\infty).
    \end{equation*}
\end{Th}

\section{Long time behavior of Ricci-DeTurck flow}\label{section_rf_conv}
In this section, we review the long time behavior of the normalized Ricci-DeTurck flow and  its convergence toward the hyperbolic metric. In particular, we present a quantitative exponential decay estimate, which plays an essential role in the proofs of Theorems \ref{Thm_sar>-6} and \ref{Thm_sar>-6_c0perturbation}.
These results were originally introduced by the authors in \cite{Jiang-VargasPallete_RF}.

\subsection{Weighted little H{\"o}lder spaces}
First, we introduce weighted little H{\"o}lder spaces, and apply the interpolation theory.
For closed hyperbolic $3$-manifolds, Knopf-Young \cite{Knopf-Young} studied the stability of the hyperbolic metric $h_0$ using Simonett's interpolation results \cite{Simonett}. They showed that starting from a metric in a little H{\"o}lder $\Vert \cdot\Vert _{2\alpha+\varrho}$ neighborhood of $h_0$, the normalized Ricci-DeTurck flow converges exponentially fast in the $\Vert \cdot\Vert _{2+\varrho}$ norm to $h_0$, where $\varrho\in (0,1)$ and $\alpha\in (\frac{1}{2},1)$.

However, as explained in Section 5 of \cite{Jiang-VargasPallete_RF}, for the cusped manifolds, it is necessary to introduce an additional exponential weight in the thin part of the cusps.

To start our discussion, let $s>0$. For each $x\in M$, let $\Tilde{B}(x)\subset \mathbb{H}^3$ be the unit ball centered at a lift of $x$. For each tensor $l$ on $M$, the lift of $l$ on $\mathbb{H}^3$ is still denoted by $l$.

\begin{Def}[weighted little H{\"o}lder spaces]\label{def_little_holder}

Given $s\geq 0$. 
    The \emph{weighted H{\"o}lder norm} $\Vert \cdot\Vert _{\mathfrak{h}^{k+\alpha}_{s}}$ is defined as
\begin{align}\label{equ_weighted norm}
    \Vert l\Vert _{\mathfrak{h}^{k+\alpha}_{s}}: &= \sup_{x\in M} \ww(x)\Vert l|_{\Tilde{B}(x)} \Vert_{\mathfrak{h}^{k+\alpha}}\\\nonumber
    &=\sup_{x\in M, 0\leq j\leq k} \left( \ww(x)|\nabla^j\,l(x)| + \sup_{y_1\neq y_2\in \Tilde{B}(x)} \ww(x)\frac{|\nabla^kl(y_1)-\nabla^kl(y_2)|}{d_{\Tilde{B}(x)}(y_1,y_2)^\alpha} \right)
\end{align}
where 
\begin{equation*}
\ww(x) =
    (r(x)+1)e^{-r(x)},
\end{equation*}
and
\begin{align*}
r(x) &=\begin{cases}
    0\quad\text{ if }x\in M(s),\\
    \dist(x, \partial M(s))=\min_k (\dist(x,T_k\times \{s\})\quad\text{otherwise}.
\end{cases}\\
\end{align*}
The $(r+1)$ multiplicative factor for the weight function $\ww(x)$ is so that
\[\Vert l \Vert_{L^2(M)} \leq C_{s} \Vert l\Vert _{\mathfrak{h}^{k+\alpha}_{s}},
\]
holds.

Moreover, $\ww(x)$ satisfies
\[|\nabla^j \ww(x)| \leq C_j \ww(x)
\]
we can easily check that the norm $\Vert l \Vert_{\mathfrak{h}^{k+\alpha}_{s}}$ is equivalent to
\[\sup_{x\in M, 0\leq j\leq k} \left( |\nabla^j(\ww(x)\,l(x))| + \sup_{y_1\neq y_2\in \Tilde{B}(x)} \ww(x)\frac{|\nabla^kl(y_1)-\nabla^kl(y_2)|}{d_{\Tilde{B}(x)}(y_1,y_2)^\alpha} \right)
\]

The \emph{little H{\"o}lder space} $\mathfrak{h}^{k+\alpha}_{s}$ is defined to be the closure of $C_c^\infty$ symmetric covariant 2-tensors compactly supported in $M$ with respect to the weighted H{\"o}lder norm $\Vert \cdot\Vert _{\mathfrak{h}^{k+\alpha}_{s}}$. 
\end{Def}

For fixed $0<\varrho<1$, we define \begin{equation*}
    \mathcal{X}_{0}=\mathcal{X}_{0}(M,\varrho,s)=:\mathfrak{h}^{0+\varrho}_s,\quad \mathcal{X}_{1}=\mathcal{X}_{1}(M, \varrho,s)=:\mathfrak{h}^{2+\varrho}_s.
\end{equation*}

\subsection{Exponential attractivity}

\begin{Th}[Theorem 1.1, \cite{Jiang-VargasPallete_RF}]\label{thm_ricci_flow}
Let $(M,h_0)$ be a hyperbolic $3$-manifold of finite volume, and let $\alpha\in (0,1)\setminus\{\frac{1-\varrho}{2},1-\frac{\varrho}{2}\}$. For every
    $\omega\in (0,1)$,
    there exist $\rho, c>0$, such that if $h$ is a metric on $M$ with 
    \begin{equation*}
  \Vert h-h_0\Vert _{C^0(M)}<\rho,
\end{equation*} then the solution $h(t)$ of the normalized Ricci-DeTurck flow \eqref{DRF} starting at $h(0)=h$ exists for all time. 
Moreover, we have \begin{equation*}
    \Vert h(t)-h_0\Vert _{\mathcal{X}_{1}}\leq \frac{c}{(t-1)^{1-\alpha}} e^{-\omega t}\Vert h-h_0\Vert _{C^0(M)}, \quad \forall t>1.
\end{equation*}
\end{Th}

Furthermore, we seek an estimate for the area of each closed essential minimal surface $\Sigma_i$ with respect to the given metric $h$ in Theorem \ref{Thm_sar>-6}. 
Given that we need a metric sufficiently close to $h_0$ to obtain global convergence of Ricci flow, in the cases that the metric $h$ was not as in Theorem~\ref{thm_ricci_flow} we will replace it with another metric $h_i$ that is hyperbolic outside of a thick part containing $\Sigma_i$, and then restart the flow. 
This guarantees all conditions of the theorem are met while taking case of not changing the respective area minimizer. Therefore, to define the weighted spaces $\mathcal{X}_j, j=0,1$, we must fix $i \in \mathbb{N}$ and derive a result for each $i$. In Section \ref{sec:proofS>-6}, we assume that the minimal surface is $\Sigma_i$ is contained within the thick part $M(s_i)$, where $M(s):=M\setminus (\cup_jT_j\times (s,\infty))$. We denote by $\mathfrak{h}_{s}^{k+\alpha}$ the weighted H{\"o}lder space where the weight is applied starting at a height $s$ we are fixing once and for all.

\section {Proof of Theorem \ref{Thm_sar>-6}}\label{sec:proofS>-6}
To give the proof of Theorem \ref{Thm_sar>-6}, the key observation is the following proposition. 

\begin{Prop}\label{prop_2}
Suppose that $(M,h_0)$ is a hyperbolic $3$-manifold of finite volume, and it is infinitesimally rigid.
Let $h$ be a weakly cusped metric on $M$ with $R(h)\geq -6$.
Then for any sequence $\Pi_i\in S_{\frac{1}{i}}(M)$,
we have
    \begin{equation*}
    \underset{i\rightarrow\infty}{\lim\inf}\dfrac{\area_h(\Pi_i)}{\area_{h_0}(\Pi_i)}\geq 1.
\end{equation*}
Furthermore, suppose that $h$ is asymptotically cusped of order at least two, and it satisfies $\Vert Rm(h)\Vert_{C^{1}(M)}<\infty$. 
Then the equality holds if and only if $h$ is isometric to $h_0$.
\end{Prop}

In the following discussion, we assume that $S_i$ and $\Sigma_i$ are closed essential surfaces immersed in $M$ that minimize the area in the homotopy class corresponding to $\Pi_i$ with respect to the metrics $h_0$ and $h$, respectively. As argued in \eqref{equ_gauss_ratio}, we have $\underset{{i\rightarrow\infty} }{\lim} \dfrac{\area_{h_0}(S_i)}{4\pi(g_i-1)}=1$, where $g_i$ represents the genus of $S_i$ and $\Sigma_i$.

Now, assuming for contradiction that there exists $\delta > 0$ and a subsequence of $\mathbb{N}$, each element still labeled by $i$, such that: \begin{equation}\label{contradiction_area2}
    \dfrac{\area_h(\Sigma_i)}{4\pi(g_i-1)}\leq 1-\delta.
\end{equation}
We will reveal the contradiction through subsequent steps.

\subsection{Modify the metric on thin part} We start by considering two special cases:
\begin{enumerate}[(I)]
    \item If $h$ is asymptotically cusped of order $k\geq 2$, then by \cite{Bessieres-Besson-Maillot} (which assumes $C^k$ asymptotics for all $k$, and generalizes to any given $k\geq 2$ as noted in Remark \ref{Rem_parameters} and Theorem \ref{thm_stability_cusp}), there exists a normalized Ricci flow with bubbling-off on $M$ starting from $h$, defined for all time.

    \item In a different setting, if $h$ satisfies $\Vert h-h_0\Vert _{C^0(M)}\leq \rho$, where $\rho$ is as in Theorem~\ref{thm_ricci_flow} (which already considers Theorem~\ref{thm_stability_C^0}), long-time existence of the normalized Ricci-DeTurck flow was established in \cite{Bamler}.
\end{enumerate}
We will examine these two cases in greater detail in the rigidity part of the proposition in Section \ref{subsection_7.1_rigidity} and in Section \ref{section_proof_thmD}.

If $h$ is a general weakly cusped metric, there may be neither long-time nor short-time existence of the flow in a sense where we still have existence and control over area minimizers as the flow evolves. Therefore we approximate $h$ by a sequence of asymptotically cusped metrics $\{h_i\}$, and run the normalized Ricci flow starting from each $h_i$.

Recall Theorem \ref{lem:surfaceexistencemonotonicity} and its proof. Given any weakly cusped metric $h$, there exists a constant $s'=s'(M,h)\geq 0$ such that $sec|_{M\setminus M(s')}(h)\leq 0$. Moreover, there exists a constant $\bar{s}_i=\bar{s}_i>s'$, depending on $M,h,\Pi_i$ and $\area_h(\Pi_i)$, such that any area-minimizing surface in the homotopy class $\Pi_i$ with respect to $h$ is contained in $M(\bar{s}_i)$.

We now choose a sequence $\{s_i\}$ with $s_i\rightarrow\infty$ as $i\rightarrow\infty$, such that for each $i$, the value $s_i$ satisfies the following properties:
\begin{align}\label{equ_s_i}
    \bullet & \,\,s_i\geq \bar{s}_i,\\\nonumber
    \bullet & \,\,s_i\text{ is large enough so that Lemma \ref{lem:minimizer_close_to_hyp_in_compact} applies to compact set $K=M(s_i)$: }\\\nonumber
   &\text{ Given $0<a<1$, there exists a constant $\epsilon_i=\epsilon_i(M,h_0,a,\Pi_i)>0$, such that,}\\\nonumber
   &\text{ if a metric $g$ on $M$ satisfies $\Vert (g-h_0)|_{M(s_i)}\Vert _{C^1}<\epsilon_i$ and $sec(g|_{M(s_i)})\leq -a^2< 0$,}\\\nonumber
    &\text{ then there exists an area minimizer of $\Pi_i$ with respect to $g$ contained in $M(s_i)$.}\\\nonumber
    &\text{ Moreover, all area minimizers of $\Pi_i$ with respect to $g$ are contained in $M(s_i)$.}
\end{align}
Then we define a new metric $h_i$ on $M$ using $s_i$, such that 
\begin{align}\label{equ_h_i}
    \bullet & \,\,h_i=h\text{ on the thick part }M(s_i),\\\nonumber
    \bullet & \,\,h_i=h_0\text{ on the thin part }M\setminus M(2s_i)=\cup_j T_j\times (2s_i,\infty),\\\nonumber
    \bullet &\,\, h_i\text{ is a smooth interpolation between the metrics }h\text{ and }h_0\text{ on }M(2s_i)\setminus M(s_i)\\\nonumber
    &=\cup_j T_j\times (s_i,2s_i], \text{ which satisfies }R(h_i)\geq -6 \text{ and } sec(h_i)\leq 0.
\end{align}

As $s_i\geq \bar{s}_i$, the surface $\Sigma_i$, which minimizes area in the homotopy class $\Pi_i$ with respect to $h$, lies within $M(s_i)$. 
We will show that $\Sigma_i$ is also an area minimizer in $\Pi_i$ with respect to the modified metric $h_i$.

Since $\Sigma_i$ lies in the region where $h_i$ agrees with $h$, we have $\area_{h}(\Sigma_i)=\area_{h_i}(\Sigma_i)\geq \area_{h_i}(\Pi_i)$.
Recall from Theorem \ref{lem:surfaceexistencemonotonicity} (equation \eqref{equ_3.1_2}) that the barrier $\bar{s}_i$ corresponding to the metric $h$ can be chosen by
\begin{align*}
    d_h(\bar{s}_i,s')= &2\left(\sqrt{\frac{k_i\area_h(\Pi_i)}{\pi}}+1\right)=2\left(\sqrt{\frac{k_i\area_h(\Sigma_i)}{\pi}}+1\right)\\
    \geq & 2\left(\sqrt{\frac{k_i\area_{h_i}(\Pi_i)}{\pi}}+1\right),
\end{align*}
where $s'<\bar{s}_i$ is a constant so that $M\setminus M(s')$ is a union of disjoint cusp neighborhoods and $sec(h|_{M\setminus M(s')})\leq 0$, and $k_i\in\mathbb{N}$ is the degree of the covering of $M$ so that the lift of $\Sigma_i$ is embedded. 
By $s'<\bar{s}_i\leq s_i$ and \eqref{equ_h_i}, for the new metric $h_i$, the thin region $M\setminus M(s')$ is still a union of disjoint cusp neighborhoods with $sec({h_i}|_{M\setminus M(s')})\leq 0$, and we have $d_{h_i}(\bar{s}_i,s')=d_h(\bar{s}_i,s')$. 
Hence, we can also use $\bar{s}_i$ as a barrier for the new metric $h_i$. This implies that any area minimizer for $\Pi_i$ with respect to $h_i$ must be contained in $M(\bar{s_i})\subset M(s_i)$, where $h_i=h$. In particular, $\Sigma_i$ also minimizes area among all homotopic surfaces with respect to $h_i$. In this case, changing the metric from $h$ to $h_i$ does not affect which surfaces minimize the area in $\Pi_i$

\subsection{Run Ricci and Ricci-DeTurck flows}
Suppose that $h_i(t)$ solves the normalized Ricci flow \eqref{RF}, starting with $h_i(0)=h_i$. 
Let $\Sigma_{i}(t)\subset M$ represent the surface with the minimum area with respect to $h_i(t)$ and it is homotopic to $\Sigma_i$.

\subsubsection{Metric surgeries}

Recall the notions of Ricci flow with bubbling-off in Section \ref{subsection_bubbling}. Since our initial metric $h_i$ is identical to $h_0$ on the thin part $M \setminus M(2s_i)$ as defined in \eqref{equ_h_i}, 
it possesses a cusp-like structure, which permits us to perform Ricci flow with $(r(t),\delta(t))$-bubbling-off on $M$ starting at $h_i$ using parameters defined in Remark \ref{Rem_parameters}.
According to the proof of stability in Theorem \ref{thm_stability_cusp}, $h_i(t)$ is asymptotic to $h_0$ at infinity in the cuspidal end in $C^k$ for all $k\in\mathbb{N}$, uniformly for all time $t\in [0,\infty)$, and the surgeries stay away from the cusp.

Furthermore, because of the reduction in volume through surgery, there can only be a finite number of surgeries. This finite number is represented as $m_i\in \mathbb{N}$. The only possible surgeries are pinching off inessential $\delta$-necks and attaching $\delta$-almost standard caps.

Let $t^1_i<t^2_i<\cdots <t^{m_i}_i$ be the times within $(0,\infty)$ at which some points of $M$ become singular, and let $I^{j}_i=[t^{j-1}_i,t^{j}_i)$ be a time interval, where $t_i^0=0$ and $1\leq j\leq m_i$. We consider the Ricci flow $(M^1_i\times I_i^1, h^1_i(t)),\cdots, (M^{m_i}_i\times I_i^{m_i}, h^{m_i}_i(t))$ on $3$-manifolds $M^1_i,\cdots, M^{m_i}_i$. 
Since all the surgeries are topologically trivial, we have $M^j_i=M$ for any $1\leq j\leq m_i$.
 Additionally, let $U^j_i\subset M$ be an open subset consisting of points where the curvature remains bounded as $t\rightarrow t^j_i$ on $I^j_i$, and let $\overline{h}_i^j$ be the limit of $h_i^j(t)$ as $t\rightarrow t^j_i$ on $I^j_i$. Then, there exists an isometry between $(U^j_i,\overline{h}_i^j)$ and $(M,h_i^{j+1}(t^j_i))$, representing the region where the surgery does not occur. $M\setminus U^j_i$ is diffeomorphic to a union of closed 3-balls, within which the surgeries occur. We can assume that the boundary of each 3-ball represents the centers of a $\delta$-neck. We then cut it off along its boundary sphere, remove the $\delta$-cap end, which, for instance, contains $\psi(S^2\times (0,\frac{1}{\delta}))$, and glue in an almost standard cap.

To proceed with the Ricci flow and use it to estimate the area of $\Sigma_i(t_i^j)$, we prove the following lemma.

\begin{Lemma}\label{Lemma_7.3}
    For each surgery time $t^j_i$, the area-minimizing surface $\Sigma_i(t^j_i)$ for $\Pi_i$ of the manifold $(M, \overline{h}^{j}_i)$ is contained in $U^j_i$.
\end{Lemma}
In other words, $\Sigma_i(t^j_i)$ does not intersect the surgery region. 
\begin{proof}
To see this, we assume by contradiction that $\Sigma_i(t^j_i)\cap (M\setminus U^j_i)\neq \emptyset$. Utilizing the diffeomorphism between the latter space and a collection of 3-balls, we can suppose that $\Sigma_i(t^j_i)$ intersects with the boundary $\partial B_i^j$ of some 3-ball $B^j_i$ within $M\setminus U^j_i$. This ball $B^j_i$ contains a $\delta$-neck $N$, characterized by a homothety constant $\lambda > 0$, and its boundary is given by $\partial B^j_i=\psi(S^2\times \{\frac{1}{\delta}\})$. By slightly perturbing $B^j_i$, we can assume that $\Sigma_i(t^j_i)$ intersects $\partial B_i^j$ transversely in a union of circles.

Let $D$ be a connected component of $\Sigma_i(t^j_i)\cap B^j_i$, then it intersects $\psi(S^2\times \{s\})$ for all $s\in (0,\frac{1}{\delta})$. Consequently, the monotonicity formula for minimal surfaces yields a constant $c>0$, depending only on $(M, \overline{h}_i^j)$, such that for any $s\in (\frac{1}{2}, \frac{1}{\delta}-\frac{1}{2})$, the following inequality holds:
\begin{equation*}
    \area_{\overline{h}_i^j}\Big(D\cap \psi\Big(S^2\times \Big(s-\frac{1}{2},s+\frac{1}{2}\Big)\Big)\Big)>c\lambda^2.
\end{equation*}
Choose $\delta< (\frac{8\pi}{c}+1)^{-1}$, the above estimate implies \begin{equation*}
\area_{\overline{h}_i^j}(D\cap B^j_i)\geq \area_{\overline{h}_i^j}\Big(D\cap \psi\Big(S^2\times \Big(0,\frac{1}{\delta}\Big)\Big)\Big)>c\Big(\frac{1}{\delta}-1\Big)\lambda^2>8\pi\lambda^2.
\end{equation*}

However, on the contrary, $\partial D$ bounds a disk $D'$ within the sphere $\partial B^j_i$ whose area is not greater than $4\pi\lambda^2$. By cutting off $D$ on $\Sigma_i(t^j_i)$ along $\partial D$ and replacing it with $D'$, we obtain a surface homotopic to $\Sigma_i(t^j_i)$ but with a smaller area with respect to the induced metric of $\overline{h}_i^j$, contradicting its minimality. Therefore, the surgeries in the Ricci flow do not impact $\Sigma_i(t)$ for all $t\in [0,\infty)$.
\end{proof}

\subsubsection{Mixed flows and Theorem \ref{thm_ricci_flow}}
We now verify the condition of Theorem \ref{thm_ricci_flow}.

Recall that $h_i(t)$ represents the normalized Ricci flow with $h_i(0)=h_i$, and $\Sigma_i(t)$ denotes the surface with the minimum area with respect to $h_i(t)$ that is homotopic to $\Sigma_i$, we have $\Sigma_i(0)=\Sigma_i$. Lemma \ref{Lemma_7.3} implies that, for each surgery time $t_i^j$, the surface $\Sigma_i(t_i^j)$ stays away from the surgery region. 

Due to the convergence of $h_i(t)$ toward $h_0$ on the thick part (\cite[Theorem 1.2]{Besson-Courtois-Gallot}), there exists a post-surgery time $t_i>t_i^{m_i}$ 
such that \begin{equation*}
   \Vert l_i(t_i)\Vert _{C^2(M(s_i))}<\rho.
\end{equation*}
If, on the thin part $M\setminus M(s_i)$, $h_i(t_i)$ is not in the $C^2$-neighborhood of $h_0$ of radius $\rho$, we replace $h_i(t_i)$ with $h_{i+}(t_i)$ on $M\setminus M(s_i)$ so that the new metric agrees with $h_0$ on a further thin part, and it satisfies \begin{equation}\label{equ_condition_r_i}
   \Vert h_{i+}(t_i)-h_0\Vert _{C^2(M)}<\rho.
\end{equation}
This verifies the condition of Theorem \ref{thm_ricci_flow}.

By the assumption of $s_i$ in \eqref{equ_s_i} and Lemma \ref{lem:minimizer_close_to_hyp_in_compact}, since the $C^2$-distance between $h_i(t_i)$ and $h_0$ on $M(s_i)$ is less than $\rho$, after replacing $\rho$ with a smaller constant $\rho_i$ if needed, we have $sec({h_i(t_i)}|_{M(s_i)})\leq -a^2<0$. Hence, the surface $\Sigma_i(t_i)$, along with any other area minimizers for $\Pi_i$ with respect to either $h_i(t_i)$ or $h_{i+}(t_i)$ (if they exist), is contained in $M(s_i)$. This implies that $\Sigma_i(t_i)$ is also an area minimizer in its homotopy class with respect to $h_{i+}(t_i)$, this modification does not affect the area-minimizing surfaces.

Now we redefine $h_i(t)$ as a mixed flow:
For $0\leq t< t_i$, $h_i(t)$ is still the normalized Ricci flow. And for $t\geq t_i$, it solves the normalized Ricci-DeTurck flow starting with $h_i(t_i):=h_{i+}(t_i)$.
We still use $\Sigma_i(t)$ to represent the surface with the minimum area with respect to $h_i(t)$ that is homotopic to $\Sigma_i$.

\subsubsection{Area ratio estimates}
Define $\mathcal{A}_i(t):=\area_{h_i(t)}(\Pi_i)$. According to Lemma 9 of \cite{Bray-Brendle-Eichmair-Neves}, $\mathcal{A}_i(t)$ is a Lipschitz function on both intervals $[0,t_i]$ and $[t_i,\infty)$. Therefore, it is differentiable almost everywhere. If $t$ is a point where $\mathcal{A}_i$ is differentiable, then we define $\mathcal{A}^t_i(s):=\area_{h_i(s)}(\Sigma_i(t))$. In this case, $\mathcal{A}_i(s)\leq \mathcal{A}_i^t(s)$ for all $s\in [0,\infty)$. 
Additionally, by applying Stokes theorem, we have \begin{align}\label{ode}
    (\mathcal{A}_i^t)'(t) &=\int_{\Sigma_i(t)}\frac{d}{ds}\Big|_{s=t}\sqrt{\text{det}_{h_i(t)}h_i(s)}\,dA_{h_i(t)}\\\nonumber
    &=\frac{1}{2}\int_{\Sigma_i(t)}\tr_{h_i(t)}\left(\frac{d}{ds}\Big|_{s=t}\,h_i(s)|_{\Sigma_i(t)}\right)\,dA_{h_i(t)}\\\nonumber
    &= -\int_{\Sigma_i(t)} \big(Ric(h_i(t))(e_1,e_1)+Ric(h_i(t))(e_2,e_2)+4\big)\,dA_{h_i(t)}\\\nonumber
    &=4\pi(g_i-1)-\area_{h_i(t)}(\Sigma_i(t))-\int_{\Sigma_i(t)}\left(\frac{R(h_i(t))+6}{2}+\frac{|A|^2}{2}\right)\,dA_{h_i(t)}\\\nonumber
    &\leq 4\pi(g_i-1)-
    \mathcal{A}_i^t(t).
\end{align}
We use $R(h_i(t))\geq -6$ in the last inequality because this lower bound of the scalar curvature is preserved by the normalized Ricci flow and DeTurck flow by maximum principle.
Consequently, we obtain \begin{equation*}
    \mathcal{A}_i'(t)\leq (\mathcal{A}_i^t)'(t)\leq 4\pi(g_i-1)-\mathcal{A}_i^t(t)= 4\pi(g_i-1)-\mathcal{A}_i(t).
\end{equation*}
Solving this ODE and applying the assumption \eqref{contradiction_area2} yield the following. \begin{equation}\label{sol_ode}
    \frac{\area_{h_i(t)}(\Sigma_i(t))}{4\pi(g_i-1)}\leq 1-e^{-t}\left(1-\frac{\area_{h_i}(\Sigma_i)}{4\pi(g_i-1)}\right) = 1-e^{-t}\left(1-\frac{\area_{h}(\Sigma_i)}{4\pi(g_i-1)}\right) \leq 1-\delta e^{-t}.
\end{equation}

\subsection{Apply exponential decay estimate}

Denote by $l_i(t)$ the difference between $h_i(t)$ and $h_0$. The condition of Theorem \ref{thm_ricci_flow} is verified in \eqref{equ_condition_r_i}, and then we will prove the following result.

\begin{Lemma}\label{lemma_A.2}
    Let $\omega\in (0,1)$ be the constant in Theorem \ref{thm_ricci_flow}, then we can find a sequence $\{T_i\}_{i\in \mathbb{N}}$ with $T_i>  t_i+1$ and $T_i\rightarrow \infty$, such that the following statements hold. 
    \begin{enumerate}
        \item For each $k\in\mathbb{N}$, there exists a constant $C_k>0$ independent of $i$ so that  \begin{equation*}
        \Vert l_i(T_i)\Vert _{\mathfrak{h}^{k+\varrho}(M(s_i))}\leq C_ke^{-\omega T_i}.
    \end{equation*}

    \item  
    As $i\to\infty$,
    $$\Vert e^{T_i}l_i(T_i)\Vert_{C^2(M(s_i))}\to 0.$$ 
    \end{enumerate}
\end{Lemma}

To derive (2), we need the next lemma.

\begin{Lemma}\label{lemma_covergence_to_l}
   Let $h(t)$ be a normalized Ricci-DeTurck flow satisfying the assumptions of Theorem \ref{thm_ricci_flow}, where the little H{\"o}lder spaces are defined with spatial parameter $s>0$. Define $l(t)=h(t)-h_0$. Then for each integer $k$, there exists a sequence on the $t$ variable going to infinity so that along this sequence, $e^{t} l(t)$ converges in $C^2$ on compact sets to a tensor $\ovl{l}$ as $t\rightarrow\infty$, where $\ovl{l}\in C^2_{loc}(Sym^2(T^*M))\cap H^k(M)$, and it satisfies $A_{h_0}(\ovl{l}) = -\ovl{l}$.
\end{Lemma}
\begin{proof}
Let $\mathcal{A}(h(t))$ be the DeTurck operator of $h(t)$, given by the expression on the right-hand side of \eqref{DRF}.
    By equation (7.1) in \cite{Jiang-VargasPallete_RF}, we have that $l(t)$ is of the form
    \begin{equation*}
    l(t)=e^{tA_{h_0}} l(0)+\int_0^t e^{(t-s)A_{h_0}}\big(\mathcal{A}(h(s))(h(s))-\mathcal{A}(h_0)(h_0)-A_{h_0}(l(s))
    \big)\,ds.
\end{equation*}
   By Theorem \ref{thm_ricci_flow}, it satisfies $\Vert l(t)\Vert _{\mathfrak{h}^{2+\varrho}_s(M)}= O(e^{-\omega t})$ for $t\geq 2$.
   Consider any $k\in\mathbb{N}$.
   The derivative estimates in \cite[Corollary 2.7]{Bamler} applies to any ball $B(\Tilde{x},r)\subset \mathbb{H}^3$ with radius $r$, 
   and it provides a constant $c(k,\varrho,r)>0$, such that for any $x\in M$ with a lift $\Tilde{x}$, we have $\Vert l(t+1)\Vert_{\mathfrak{h}^{k+\varrho}(B(\Tilde{x},r))}\leq c(k,\varrho,r)\Vert l(t)\Vert_{C^0(B(\Tilde{x},2r))}$. This implies that
   \begin{equation}\label{equ_schauder_weighted}
       \Vert l(t)\Vert _{\mathfrak{h}^{k+\varrho}_s(M)}= O(e^{-\omega t})\quad \forall k\in\mathbb{N}, \,t\geq 3.
   \end{equation}
   
   Consider the term  
    \begin{equation*}
        l(t) - e^{tA_{h_0}}l(0) =  \int_0^t e^{(t-s)A_{h_0}}\big(\mathcal{A}(h(s))(h(s))-\mathcal{A}(h_0)(h_0)-A_{h_0}(l(s))
    \big)\,ds.
    \end{equation*}
Since $h(0)$ is sufficiently close to $h_0$ in $C^0$, Theorem \ref{thm_stability_C^0} implies that the metric $h(t)$ stays in a small neighborhood of $h_0$ in $C^2$ for all $t\geq 1$.
Therefore, for any $\tau\in [0,1]$ and $t\geq 1$, $h_\tau(t):=\tau h(t)+(1-\tau)h_0$ remains close to $h_0$ in $C^2$. Denote the linearization of $\mathcal{A}(\cdot)(\cdot)$ at $h_\tau(t)$ by $A_{h_\tau(t)}$, then it has the following form (see, for example, \cite[Proposition 2.3.7]{Topping}).
\begin{equation*}
    A_{h_\tau(t)}(l(t))=\Delta_L^{h_\tau(t)}l(t)+\mathcal{L}_{(\delta G(l(t)))^\#}h_\tau(t)-4l(t),
\end{equation*}
where $G(l(t))=l(t)-\frac12\left(\tr_{h_\tau(t)}(l(t))\right)h_\tau(t)$. We observe that \begin{align*}
    &\left\Vert A_{h_\tau(t)}(l(t))-A_{h_0}(l(t))\right\Vert _{\mathfrak{h}^{0+\varrho}_s(M)}\\\nonumber
    \leq& \left\Vert \left(\Delta_L^{h_\tau(t)}-\Delta_L^{h_0}\right)l(t)\right\Vert_{\mathfrak{h}^{0+\varrho}_s(M)}+\left\Vert\mathcal{L}_{(\delta G(l(t)))^\#}h_\tau(t)\right\Vert _{\mathfrak{h}^{0+\varrho}_s(M)}\\\nonumber
    \lesssim& \Vert h_\tau(t)-h_0\Vert_{\mathfrak{h}^{2+\varrho}_s(M)}\Vert l(t)\Vert_{\mathfrak{h}^{2+\varrho}_s(M)}+\Vert h_\tau(t)-h_0\Vert_{\mathfrak{h}^{1+\varrho}_s(M)}\Vert l(t)\Vert_{\mathfrak{h}^{1+\varrho}_s(M)}\\\nonumber
    \lesssim & \Vert l(t)\Vert_{\mathfrak{h}^{2+\varrho}_s(M)}^2.
\end{align*}
Since the map $\tau\mapsto \mathcal{A}(h_\tau(t))(h_\tau(t))$ defined on $[0,1]$ is $C^1$, by applying the mean value theorem to this map, we have  \begin{align*}
    &\left\Vert \mathcal{A}(h(t))(h(t))-\mathcal{A}(h_0)(h_0)-A_{h_0}(l(t))
    \right\Vert _{\mathfrak{h}^{0+\varrho}_s(M)}\\
    \leq & \max_{0\leq \tau \leq 1}\Vert A_{h_\tau(t)}(l(t))-A_{h_0}(l(t))\Vert _{\mathfrak{h}^{0+\varrho}_s(M)}\\
    \lesssim &\Vert l(t)\Vert_{\mathfrak{h}^{2+\varrho}_s(M)}^2=O(e^{-2\omega t}),
\end{align*}
where $t\geq 2$.
Let \begin{equation*}
    Q_t:=\mathcal{A}(h(t))(h(t))-\mathcal{A}(h_0)(h_0)-A_{h_0}(l(t)).
\end{equation*}
    We follow a similar argument as in \eqref{equ_schauder_weighted}, from which we get
    \begin{equation*}
    \Vert Q_t\Vert _{\mathfrak{h}^{k+\varrho}_s(M)}= O(e^{-2\omega t})\quad \forall k\in\mathbb{N},\,t\geq 3.
    \end{equation*}
    Taking $\omega\in (\frac{1}{2},1)$, since $\mathfrak{h}^{k+\varrho}_s(M)\subset H^k(M)$, we obtain that $e^tQ_t\to 0$ in $H^k$ on $M$ as $t\to \infty$. 

    Defining $v=e^{t}l$ we have that $v$ is the solution of
    \[
    \frac{dv}{dt} = A_{h_0}v + v + e^tQ_t.
    \]
    By the $L^2$ estimate of $A_{h_0}$ in \eqref{equ_spectrum}, we have $\langle A_{h_0}v,v\rangle_{L^2} \leq -\Vert v\Vert^2_{L^2}$. Therefore, $f=\Vert v\Vert_{L^2(M)}$ satisfies the following differential inequality
    \[
    \frac{df}{dt} \leq \Vert e^tQ_t\Vert_{L^2} f \leq e^{-(2\omega-1)t}f, \quad \omega\in \left(\frac12,1\right),\,t\geq 3.
    \]
    Hence it follows that
    \[
    f(t) \leq f(3)e^{\int_3^te^{-(2\omega-1)s}ds},\quad \omega\in \left(\frac12,1\right),\,t\geq 3.
    \]
    Or equivalently, $\Vert v \Vert_{L^2}$ is uniformly bounded along the flow for $t\geq 3$. As the terms in $A_{h_0}$ are parallel with respect to the Levi-Civita connection of $h_0$, we can take derivatives to the equation satisfied by $v$ and proceed to analogy to obtain $\Vert v\Vert_{H^k} \leq c_k$ for some constants $c_k$ and for all $t\geq 3$. This in turn implies uniform bounds for times derivatives on given spatial compact sets. In particular, by applying Sobolev embedding and Rellich-Kondrachov compactness theorem, there is a sequence of times $t_i\rightarrow+\infty$ so that the flowlines starting at $v(t_i)$ converge in $H^k$ for any given time $t\in [0,+\infty[$ and $C^2$ in compact sets of $M\times [0,\infty[$ to a time dependent tensor $V_t$ on $M$ ($t\geq 0$).

    As $e^tQ_t$ goes to zero in $\mathfrak{h}^{k+\varrho}$ norm we have then that $V_t$ is the solution of the differential equation
    \[
    \frac{dV_t}{dt} = A_{h_0}V_t + V_t.
    \]
    Reasoning as before, we not only obtain in this case that $\Vert V_t\Vert_{H^k}$ is uniformly bounded (by a constant depending on $k$) but monotone. Since $V_t$ converges to a tensor $\ovl{l}$ in $H^k$ and $C^2$ in compact sets, we have $\ovl{l}\in C^2_{loc}(Sym^2(T^*M))\cap H^k(M)$.
    We claim that $\ovl{l}$ must be a $-1$ eigentensor.

    Assume the contrary. Then by starting at $\overline{l}$ and flowing by the equation $\ovl{l} = A_{h_0}\ovl{l} + \ovl{l}$ we will strictly decrease the $L^2$ norm in a neighbourhood of $\overline{l}$, making impossible the $L^2$ convergence of $V_t$ to $\overline{l}$.

    The argument is finished by doing a diagonal argument and taking a subsequence of times so that $l(t)$ approaches an accumulation tensor of $V_t$.

\end{proof}

We now provide the exponential decay estimate for $l_i(t)$, and use the above lemma to deduce the convergence of $e^{T_i}l_i(T_i)$.

\begin{proof}[Proof of Lemma \ref{lemma_A.2}]
Using Theorem \ref{thm_ricci_flow}, we get
\begin{equation*}
    \Vert l_i(t)\Vert _{\mathfrak{h}^{2+\varrho}_{s}(M)}\leq \frac{c\rho_i}{(t-t_i-1)^{1-\alpha}}e^{-\omega (t-t_i)},\quad \forall t> t_i+1.
\end{equation*}
    Furthermore, since $h(t)$ stays close to $h_0$ for $t\geq t_i+1$, the spatial parameter $r(x)$ on $\partial M(s')$ in Definition \ref{def_little_holder} is approximately $s'-s$, where $s'\geq s$. Therefore, $r(x)$ is bounded by $2(s_i-s)$ on $M(s_i)$. We have 
    \begin{equation*}
    \Vert l_i(t)\Vert _{\mathfrak{h}^{2+\varrho}(M(s_i))}\leq \frac{c\rho_i}{(t-t_i-1)^{1-\alpha}}e^{-\omega (t-t_i)+2(s_i-s)},\quad \forall t> t_i+1.
\end{equation*}
   Observe that we can take a constant $C_2>0$ independent of $i$, so that for sufficiently large $t$ (where sufficiently large depends on $i$), we have $\frac{c\rho_i}{(t-t_i-1)^{1-\alpha}}e^{\omega t_i+2(s_i-s)}\leq C_2$. It implies $\Vert l_i(t)\Vert _{\mathfrak{h}^{2+\varrho}(M(s_i))}\leq C_2e^{-\omega t}$.
   A similar argument as in \eqref{equ_schauder_weighted} applies to $t\geq t_i'$ for some $t_i'>t_i+1$ and deduces the estimates for the $\mathfrak{h}^{k+\varrho}(M(s_i))$ norms, this proves (1).
   
  Next, we prove (2). Fix an arbitrary $\epsilon>0$. 
By Lemma \ref{lemma_covergence_to_l}, there exists $T_i\geq t_i'$ such that $T_i\to\infty$, and 
\begin{equation*}
    \Vert e^{T_i}l_i(T_i)-\ovl{l}_i\Vert_{C^2(M(s_i))}<\epsilon, 
\end{equation*}
where the tensor $\ovl{l}_i$ satisfies $\ovl{l}_i\in C^2_{loc}(Sym^2(T^*M))\cap H^1(M)$ and $A_{h_0}(\ovl{l}_i)=-\ovl{l}_i$. 

By \eqref{equ_spectrum}, $\ovl{l}_i$ must be a traceless Codazzi tensor, which correspond to infinitesimal conformally flat deformations of the hyperbolic metric. 
Since $(M,h_0)$ is infinitesimally rigid, $\ovl{l}_i$ must be $0$. Hence, Lemma \ref{lemma_A.2} (2) follows.

\end{proof}

\subsection{Proof of inequality in Proposition \ref{prop_2}}\label{subsection_inequality_prop_2}

To obtain the inequality of the proposition, we will use the exponential decay of $l_i(t)$ on $M(s_i)$ to analyze the area ratio inequality \eqref{sol_ode}, associated with the surface $\Sigma_i(t)$. So we first need to argue that \begin{equation}\label{equ_surface_in_M_s}
    \Sigma_i(t)\subset M(s_i)\quad \forall t\geq t_i+1.
\end{equation}
According to Theorem \ref{thm_stability_C^0}, given $a\in (0,1)$ and $\epsilon_i>0$ in \eqref{equ_s_i}, after possibly replacing $\rho$ in Theorem \ref{thm_ricci_flow} with a smaller constant $\rho_i$ as done before, the metric $h_i(t)$ remains in the $\epsilon_i$-neighborhood of $h_0$ in $C^2$ for all $t\geq t_i+1$, and its sectional curvature satisfies $sec(h_i(t))\leq -a^2< 0$. This allows us to apply \eqref{equ_s_i} (which uses Lemma \ref{lem:minimizer_close_to_hyp_in_compact}). Consequently, we deduce that $\Sigma_i(t)$ lies inside $M(s_i)$ for all $t\geq t_i+1$. 

Let $D_i$ ($\Omega_i(t)$) be the lifts of $S_i$ ($\Sigma_i(t)$, respectively) to the universal cover of $M$. These discs $D_i$ and $\Omega_i(t)$ are asymptotic and at a uniformly bounded Hausdorff distance from each other for sufficiently large $t$. Additionally, as $h_i(t)\rightarrow h_0$ on $M(s_i)$, $\Omega_i(t)$ converges uniformly on compact sets to $D_i$ in $\mathfrak{h}^{2+\varrho}$. Hence, there exists a smooth map $f_i(t)$ on $D_i$ with $|f_i(t)|_{\mathfrak{h}^{2+\varrho}}<1$, such that $\Omega_i(t)$ can be expressed as the graph of $f_i(t)$ over $D_i$. More precisely, let $n_i$ be the unit normal vector field of $D_i$, then we have the following diffeomorphism $F_i(t)$ from the Minkowski model of $\mathbb{H}^3$.
\begin{equation*}
    F_i(t):D_i\rightarrow\Omega_i(t),\quad F_i(x,t)=\cosh (f_i(x,t))x+\sinh (f_i(x,t))n_i(x).
\end{equation*}
In particular, we have a diffeomorphism at $t=T_i$.

Recall the laminar measure associated with $\phi_i$ defined in \eqref{def_laminar_measure}.
By equation (13) of \cite{Lowe-Neves}, using the Gauss-Bonnet formula, we get \begin{align*}
     \frac{\area_{h_i(T_i)}(\Sigma_i(T_i))}{4\pi(g_i-1)}=&1-\delta_{\phi_i}\left(\Big(Ric(h_i(T_i))(e_3,e_3)-\frac{R(h_i(T_i))}{3}+|A|^2_{h_i(T_i)}\Big)\Lambda_{h_i(T_i)}\right)\\
     &+\delta_{\phi_i} \Big(\frac{R(h_i(T_i))+6}{6}\Lambda_{h_i(T_i)}\Big),
\end{align*}
where $\Lambda_{h_i(T_i)}(\phi_i)$ is the Jacobian of $F_i(T_i)\circ \phi_i$.
When this is combined with \eqref{sol_ode}, it yields the following inequality.
\begin{equation}\label{16}
    \delta e^{-T_i}\leq \delta_{\phi_i}\left(\Big(Ric(h_i(T_i))(e_3,e_3)-\frac{R(h_i(T_i))}{3}+|A|^2_{h_i(T_i)}\Big)\Lambda_{h_i(T_i)}\right).
\end{equation}

Next, we follow the approach of \cite[Lemma 4.2]{Lowe-Neves} to estimate the right-hand side of \eqref{16}.
Let $\theta(l):\mathcal{F}rM\rightarrow \mathbb{R}$ be the continuous function defined by \begin{equation*}
    \theta(l)(x,\{e_1,e_2,e_3\}):=-\frac{1}{2}A_{h_0}(l)_x(e_3,e_3),
\end{equation*}
where $\{e_1,e_2,e_3\}$ is an orthonormal basis of $M$ at $x$. 
Therefore, \begin{equation*}
    \theta(\ovl{l})(x,\{e_1,e_2,e_3\})=\frac{1}{2}(\ovl{l})_x(e_3,e_3).
\end{equation*}

As \eqref{equ_surface_in_M_s} holds for $t=T_i$, according to Lemma \ref{lemma_A.2} (1), $l_i(T_i)$ in $\Sigma_i(T_i)$ has the $\mathfrak{h}^{4+\varrho}$-norm bounded by $C_4e^{-\omega T_i}$. Consequently, utilizing the estimates in Lemma 4.2 of \cite{Lowe-Neves}, we obtain a constant $C>0$ independent of $i$, so that 
\begin{align*}
    \delta e^{-T_i}\leq &\delta_{\phi_i}\left(\Big(Ric(h_i(T_i))(e_3,e_3)-\frac{R(h_i(T_i))}{3}+|A|^2_{h_i(T_i)}\Big)\Lambda_{h_i(T_i)}\right)\\
    \leq &\Omega_*\delta_{\phi_i}\big(\theta(l_i(T_i))\big)+ Ce^{-2\omega T_i}.
\end{align*}
We multiply both sides by $e^{T_i}$:
\begin{equation}\label{equ_delta_Omega_*}
    \delta\leq \Omega_*\delta_{\phi_i}\big(\theta(e^{T_i}l_i(T_i))\big)+Ce^{-(2\omega-1) T_i}.
\end{equation}
    Since $\theta(\cdot)$ involves derivatives up to the second order, Lemma \ref{lemma_A.2} (2) implies that as $i\rightarrow\infty$, \begin{equation*}
    \Vert \theta(e^{T_i}l_i(T_i)))\Vert_{C^0(M(s_i))}=\Vert e^{T_i}\theta(l_i(T_i)))\Vert_{C^0(M(s_i))}\to 0.
\end{equation*}
Since $\Omega_*\delta_{\phi_i}$ has support in $\mathcal{F}r(M(s_i))$,
\begin{align}\label{equ_theta_conv_0}
    \left|\Omega_*\delta_{\phi_i}\big(e^{T_i}\theta(l_i(T_i))\big)\right|\leq & \Vert e^{T_i}\theta(l_i(T_i)))\Vert_{C^0(M(s_i))}\cdot\Omega_*\delta_{\phi_i}\left(\mathcal{F}r(M(s_i))\right)\\\nonumber
    \leq& \Vert e^{T_i}\theta(l_i(T_i)))\Vert_{C^0(M(s_i))}\to 0, \quad i\to \infty.
\end{align}
Choosing $\omega\in (\frac{1}{2},1)$, then it follows from \eqref{equ_delta_Omega_*} and \eqref{equ_theta_conv_0} that \begin{equation*}
    0<\delta\leq 0,
\end{equation*}
leading to a contradiction.
This means that the assumption \eqref{contradiction_area2} is false, therefore the inequality stated in Proposition \ref{prop_2} must hold.

\subsection{Proof of rigidity in Proposition \ref{prop_2}}\label{subsection_7.1_rigidity}
Suppose that a weakly cusped metric $h$ on $M$ is asymptotically cusped of order $k\geq 2$ with $\Vert Rm(h)\Vert_{C^{1}(M)}<\infty$. 
Moreover, it satisfies $R(h)\geq -6$, and $\underset{i\rightarrow\infty}{\lim\inf}\,\dfrac{\area_{h}(\Sigma_i)}{4\pi(g_i-1)}= 1$.

    By Theorem \ref{thm_stability_cusp}, there exists a normalized Ricci flow $h(t)$ with bubbling-off, starting from $h$ and defined for all time. Therefore, it is not necessary to modify the initial metric as in \eqref{equ_h_i} or to run the Ricci flow starting from different modified initial data. Furthermore, the stability in Theorem \ref{thm_stability_cusp} ensures that $h(t)$ remains asymptotically cusped of order $2$ for all time. This allows us to apply Lemma \ref{lem:areaincusp_asymphyp} to a compact set $t\in [t_0,2t_0]$, which guarantees the existence of a constant $\kappa>0$ and a thick region $K=M(s)$, so that for any sequence $\Pi_i\in S_{\frac{1}{i}}(M)$ and $\Sigma_i(t)$ minimal area representative of $\Pi_i$, we have \begin{equation}\label{equ_rigidity7.1_1}
        \area_{h(t)}(\Sigma_i(t)\cap M(s))\geq \kappa\left( \area_{h(t)}(\Sigma_i(t))\right)\quad \forall t\in [t_0,2t_0].
    \end{equation}

Let $$a(t)=\underset{i\rightarrow\infty}{\lim\inf}\,\dfrac{\area_{h(t)}(\Sigma_i(t))}{4\pi(g_i-1)}.$$ In particular, we have $a(0) = \underset{i\rightarrow\infty}{\lim\inf}\,\dfrac{\area_{h}(\Sigma_i)}{4\pi(g_i-1)}=1$. 
The area ratio estimates in \eqref{ode}-\eqref{sol_ode} implies that \begin{align}\label{ode_rigidity}
    \frac{d}{dt}\area_{h(t)}(\Sigma_i(t))\leq & 4\pi(g_i-1)-\area_{h(t)}(\Sigma_i(t))-\frac{1}{2}\int_{\Sigma_i(t)}\left(R(h(t))+6\right)\,dA_{h(t)}\\\nonumber
    \leq & 4\pi(g_i-1)-\area_{h(t)}(\Sigma_i(t)).
\end{align}
Solving the ODE and letting $i\to\infty$, we obtain
$$a(t)\leq 1-e^{-t}\left(1-a(0)\right)= 1.$$ As $h(t)$ is always weakly cusped with $R(h(t))\geq -6$, the inequality in Proposition \ref{prop_2} applies to it and implies that $a(t)\geq 1$. 
Therefore, we must have $$a(t)\equiv 1\quad\forall t\in [0,2t_0].$$

Assume that $h$ is not hyperbolic. By the maximum principle, we have $R(h(t))\geq -6$ for $t\geq0$. Moreover, by the strong maximum principle, we see that if for $t>0$, $R(h(t))$ is equal to $-6$ at an interior point, then $R(h(t))\equiv-6$ and $\overset{\circ}{Ric}\equiv 0$, which in turn implies that $h(t)$ would be hyperbolic. Since this contradicts $h$ not being hyperbolic, for the compact set $K=M(s)$ above, there exists $\delta>0$ so that \begin{equation}\label{equ_rigidity7.1_3}
    R(h(t)|_{M(s)})\geq -6+2\delta\quad \forall t\in [t_0,2t_0].
\end{equation}

\eqref{equ_rigidity7.1_1} and \eqref{equ_rigidity7.1_3} imply that
\begin{equation*}
    \int_{\Sigma_i(t)} R(h(t))\,dA_{h(t)}\geq (-6+2\kappa\delta)\area_{h(t)}(\Sigma_i(t)),\quad \forall t\in [t_0,2t_0].
\end{equation*}
Hence, the inequality in \eqref{ode_rigidity} is of the following form. \begin{equation*}
    \frac{d}{dt}\area_{h(t)}(\Sigma_i(t))\leq 4\pi(g_i-1)-(1+\kappa\delta)\area_{h(t)}(\Sigma_i(t)).
\end{equation*}
We conclude by solving this ODE that \begin{equation*}
    a(2t_0)\leq a(t_0)e^{-(1+\kappa\delta)t_0}+\frac{1-e^{-(1+\kappa\delta)t_0}}{1+\kappa\delta}=e^{-(1+\kappa\delta)t_0}+\frac{1-e^{-(1+\kappa\delta)t_0}}{1+\kappa\delta}<1,
\end{equation*}
which contradicts $a(2t_0)=1$. Consequently, if the equality of Proposition \ref{prop_2} holds, then the metric $h$ is Einstein, and thus it is hyperbolic and isometric to $h_0$.

\subsection{Proof of Theorem \ref{Thm_sar>-6}}
Let $h$ be a weakly cusped metric on $M$ with $R(h)\geq -6$, we first prove that $\overline{E}(h)\leq 2$.
For any $\eta>0$, Proposition \ref{prop_2} gives rise to a constant $\epsilon_0>0$ such that for any $\Pi\in S_{\epsilon_0}(M)$, \begin{equation*}
    \area_{h_0}(\Pi)\leq (1+\eta)\area_h(\Pi).
\end{equation*}
Thus, for any positive number $\epsilon<\epsilon_0$, \begin{align*}
    &\ln\#\{\area_{h}(\Pi)\leq 4\pi(L-1):\Pi\in S_{\epsilon}(M)\}\\
    \leq &\ln\#\{\area_{h_0}(\Pi)\leq 4\pi(1+\eta)(L-1):\Pi\in S_\epsilon(M)\}.
\end{align*}
By the definition of minimal surface entropy, it implies that \begin{equation*}
   \overline{E}(h)\leq (1+\eta)E(h_0)=2(1+\eta).
\end{equation*}
Therefore, the inequality of Theorem \ref{Thm_sar>-6} follows by taking $\eta\rightarrow 0$. 

Next, we prove the rigidity of Theorem \ref{Thm_sar>-6}. Suppose that $h$ is asymptotically cusped of order $k\geq 2$ with $\Vert Rm(h)\Vert_{C^{1}(M)}<\infty$. 
Additionally, suppose $R(h)\geq -6$ and $\overline{E}(h)=2$. Assume by contradiction that there are $\eta>0$ and $\epsilon_0>0$ such that for all $\Pi\in S_{\epsilon_0}(M)$, we have \begin{equation*}
    \area_{h_0}(\Pi)\leq (1-\eta)\area_h(\Pi).
\end{equation*}
Then, as we discussed before, \begin{equation*}
    \overline{E}(h)\leq (1-\eta)E(h_0)=2(1-\eta),
\end{equation*}
which is a contradiction. 

Therefore, we can find a sequence $\Pi_i\in S_{\frac{1}{i}}(M)$ such that \begin{equation*}
    \area_{h_0}(\Pi_i)> \Big(1-\frac{1}{i}\Big)\area_h(\Pi_i)\quad \Longrightarrow\quad  \underset{i\rightarrow\infty}{\lim\inf}\dfrac{\area_h(\Pi_i)}{\area_{h_0}(\Pi_i)}\leq 1.
\end{equation*}
It then follows from Proposition \ref{prop_2} that $h$ is isometric to $h_0$.

\section{Proof of Theorem \ref{Thm_sar>-6_c0perturbation}}\label{section_proof_thmD}
Similarly to the previous section, it suffices to prove the following proposition. 

\begin{Prop}\label{prop_2.2}
Suppose that $(M,h_0)$ is a hyperbolic $3$-manifold of finite volume. Let $h$ be a weakly cusped metric on $M$ that satisfies the following conditions.
\begin{enumerate}
    \item $\Vert h-h_0\Vert_{C^0(M)}\leq\epsilon$ for a given constant $\epsilon>0$,
    \item $h$ is asymptotically cusped of order at least two with $\Vert Rm(h)\Vert_{C^{1}(M)}<\infty$.
\end{enumerate}
If $R(h)\geq -6$, then for any sequence $\Pi_i\in S_{\frac{1}{i},\mu_{Leb}}(M)$,
we have
    \begin{equation*}
    \underset{i\rightarrow\infty}{\lim\inf}\dfrac{\area_h(\Pi_i)}{\area_{h_0}(\Pi_i)}\geq 1.
\end{equation*}
Furthermore, the equality holds if and only if $h$ is isometric to $h_0$.
\end{Prop}

In the following discussion, we assume that $\Sigma_i$ is a closed essential surface immersed in $M$ that minimize the area in the homotopy class corresponding to $\Pi_i$ with respect to the metric $h$, and denote the genus of $\Sigma_i$ by $g_i$.
Furthermore, assume for contradiction that there exists $\delta > 0$ and a subsequence of $\mathbb{N}$, each element still labeled by $i$, such that: \begin{equation}\label{contradiction_area2_2}
    \dfrac{\area_h(\Sigma_i)}{4\pi(g_i-1)}\leq 1-\delta.
\end{equation}

Suppose that $\epsilon\leq \rho$, where $\rho$ is the constant in Theorem \ref{thm_ricci_flow}.
Then the normalized Ricci-DeTurck flow $h(t)$ starting from $h$ exists for all time. 
In this case, there is only one flow without any modification on metric. Moreover, by condition (2) and Theorem \ref{thm_stability_cusp}, $h(t)$ remains asymptotically cusped of order two for all time. This implies that $h(t)$ is always weakly cusped, and therefore there exists an area-minimizing surface homotopic to $\Sigma_i$ at time $t$, which we denote by $\Sigma_i(t)$.
Following the same approach, we obtain \begin{equation}\label{sol_ode_2}
    \frac{\area_{h(t)}(\Sigma_i(t))}{4\pi(g_i-1)}\leq 1-e^{-t}\left(1-\frac{\area_{h}(\Sigma_i)}{4\pi(g_i-1)}\right) \leq 1-\delta e^{-t}.
\end{equation}

Let $l(t)=h(t)-h_0$. Using Lemma \ref{lemma_covergence_to_l} again, we obtain the following result.   
\begin{Lemma}\label{lemma_A.2_2}
    Let $\omega\in (0,1)$ be the constant in Theorem \ref{thm_ricci_flow}, 
    then we can find a sequence $\{T_i\}_{i\in \mathbb{N}}$ with $T_i>  t_i+1$ and $T_i\rightarrow \infty$, such that the following statements hold. 
    \begin{enumerate}
        \item For each $k\in\mathbb{N}$, there exists a constant $C_k>0$ independent of $i$ so that  \begin{equation*}
        \Vert l(T_i)\Vert _{\mathfrak{h}^{k+\varrho}(M(s_i))}\leq C_ke^{-\omega T_i}.
    \end{equation*}

    \item  As $i\to\infty$,
    $$\Vert e^{T_i}l(T_i)-\ovl{l}\Vert_{C^2(M(s_i))}\to 0,$$ 
    where the tensor $\ovl{l}\in C^2_{loc}(Sym^2(T^*M))\cap H^1(M)$, and it satisfies $A_{h_0}(\ovl{l})=-\ovl{l}$ in \eqref{equ_def_A}, i.e., it is an eigentensor corresponding to the largest spectrum $-1$ of $A_{h_0}$.
    \end{enumerate}
\end{Lemma}

As argued in Section \ref{subsection_inequality_prop_2}, from Lemma \ref{lemma_A.2_2} (1) we get
\begin{equation}\label{equ_delta_Omega_*_2}
    \delta\leq \Omega_*\delta_{\phi_i}\big(\theta(e^{T_i}l(T_i))\big)+Ce^{-(2\omega-1) T_i}.
\end{equation}
    Since $\theta(\cdot)$ involves derivatives up to the second order, Lemma \ref{lemma_A.2} (2) implies that as $i\rightarrow\infty$, \begin{equation}\label{equ_theta_conv_L2_2}
    \Vert \theta(e^{T_i}l(T_i))-\theta(\ovl{l})\Vert_{C^0(M(s_i))}=\Vert e^{T_i}\theta(l(T_i))-\theta(\ovl{l})\Vert_{C^0(M(s_i))}\to 0.
\end{equation}

To discuss the limit of $\Omega_*\delta_{\phi_i}\big(\theta(e^{T_i}l(T_i))\big)$, we need the following two lemmas. 
First, for those $-1$ eigentensors of $A_{h_0}$, we estimate their $L^1$ decay in the cusps and their $C^0$ norms. 
\begin{Lemma}\label{lem:L2-1eigendecay}
Let $\ovl{l}\in C^2_{loc}(Sym^2(T^*M))\cap H^1(M)$ be a tensor satisfying $A_{h_0}(\ovl{l}) = -\ovl{l}$.
Then \begin{enumerate}
    \item $\Vert\ovl{l}\Vert_{L^1(\cup_iT_i\times[r,\infty))}\lesssim \Vert\ovl{l}\Vert_{L^2(M)}e^{-\alpha r}$, where $\alpha>0$.
    \item $\Vert \ovl{l}\Vert_{C^0(M)}\lesssim \Vert\ovl{l}\Vert_{L^2(M)}<\infty$.
\end{enumerate}
\end{Lemma}

\begin{proof}
    Taking coordinates $e^{-2r}g_{\rm flat} + dr^2$ in each cusp $T\times [0,\infty)$, let $\hat{l}$ be the tensor defined in the cusp as average on each horotorus $T(r):=T\times \{r\}$ of $\ovl{l}$, that is,  \begin{equation*}
    \hat{l}_{ij}(x)=\frac{1}{\vol(T(r))}\int_{T(r)} \ovl{l}_{ij}(y)\dvol(y),
\end{equation*} 
where $x\in T(r)$.
Then we have $A_{h_0}(\hat{l}) = -\hat{l}$. As $\hat{l}$ only depends on $r$, the eigentensor equation can be expressed as follows. 
    \begin{align}\label{equ_odes}
        &(e^{2r}\hat{l}_{ij})''-2(e^{2r}\hat{l}_{ij})'+e^{2r}\hat{l}_{ij} = 2\delta_{ij}\big( \tr_{h_0}(\hat{l}) -\hat{l}_{33}\big),\quad i,j=1,2,\\\nonumber
        &(e^r\hat{l}_{i3})''-2(e^r\hat{l}_{i3})'-2e^r\hat{l}_{i3} = 0,\quad i=1,2,\\\nonumber
        &(\hat{l}_{33})''-2(\hat{l}_{33})'-3\hat{l}_{33} = 0,\\\nonumber
        &(\tr_{h_0}(\hat{l}))''-2(\tr_{h_0}(\hat{l}))'-3\tr_{h_0}(\hat{l}) = 0.
    \end{align}
    The ODEs are derived from equations (6.4)-(6.5) of \cite{Jiang-VargasPallete_RF}, where we take $\omega=0$ and $\hat{f}=\hat{l}$. The roots of the characteristic polynomials of $e^{2r}l_{12}$, $e^r\hat{l}_{i3}$, $\hat{l}_{33}$, and $\tr_{h_0}(\hat{l})$ are $1$, $1\pm\sqrt{3}$, $1\pm2$, and $1\pm 2$, respectively.
    Then the solutions to the system \eqref{equ_odes} are as follows.
\begin{align}\label{eq:ODEsolutions}
    e^{2r}\hat{l}_{12} &= a_1 e^{r}+a_2 re^{r},\\\nonumber
    e^r\hat{l}_{i3} &= b_1^i e^{(1+\sqrt{3})r}+ b_2^i e^{(1-\sqrt{3})r},\quad i=1,2,\\\nonumber
    \hat{l}_{33} &= c_1e^{3r}+c_2e^{-r},\\\nonumber
    \tr_{h_0}(\hat{l}) &= d_1e^{3r}+d_2e^{-r},
\end{align}
where the coefficients are real numbers. 

Observe that $\hat{l}$ is $L^2$-integrable, as by applying Cauchy-Schwartz we have that for $x\in T(r)$ 
\[\left( \hat{l}_{ij}(x) \right)^2 \leq \frac{\int_{T(r)} \left(\bar{l}_{ij}(y)\right)^2 \dvol(y)}{\int_{T(r)}\dvol(y)},
\]
it follows that
\[
\int_0^{\infty} e^{-2r}|\hat{l}|^2 dr \leq \Vert \bar{l}\Vert_{L^2(T\times [0,\infty))}^2.
\]
Therefore, we have
\begin{equation*}
    e^{-r}(e^{2r}\hat{l}_{12}),\, e^{-r}(e^r\hat{l}_{i3}),\, e^{-r}\hat{l}_{33},\,e^{-r}\tr_{h_0}(\hat{l})\in L^2([0,\infty)).
\end{equation*}
Observe that any root with real part greater than or equal to $1$ is not square integrable. Therefore, we must have $a_1=a_2=b_1^i=c_1=d_1=0$.

Consider the remaining coefficients $b_2^i, c_2$ and $d_2$. 
Let $\Tilde{x}$ be a lift of $x$ in $\mathbb{H}^3$, and let $L:=\Tilde{\ovl{l}}$ be the lift of $\ovl{l}$ defined as $\Tilde{\ovl{l}}(\Tilde{x}):=\ovl{l}(x)$.
It follows that 
$$ -L=A_{h_0}(L)=\Delta L-Ric(L)-4L.$$
 Since $\Delta\left(| L|^2\right)=2\langle\Delta L,L\rangle+2|\nabla L|^2$, by Lemma 3.2 of \cite{Hamenstadt-Jackel}, we have \begin{align*}
        \frac12 \Delta\left(|L|^2\right)
        =& \langle\Delta L,L\rangle+|\nabla L|^2\\
        =&-|L|^2+\langle Ric(L),L\rangle+4|L|^2+|\nabla L|^2\\
        \geq & -|L|^2-6|L|^2+2\tr_{h_0}(L)^2+4|L|^2+|\nabla L|^2\\
        \geq& -3|L|^2+|\nabla L|^2.
    \end{align*}
    On the other hand, since $|\nabla(|L|)|\leq |\nabla L|$, \begin{equation*}
        \frac12 \Delta\left(| L|^2\right)= |L|\Delta(|L|)+\left|\nabla(|L|)\right|^2
        \leq |L|\Delta(|L|)+|\nabla L|^2.
    \end{equation*}
    Combining these two inequalities and assuming $L\neq0$, we obtain \begin{equation*}
        \Delta(|L|)\geq -3|L|,
    \end{equation*}
    this verifies the condition for the De Giorgi-Nash-Moser estimate (see Theorem 8.17 in \cite{Gilbarg-Trudinger} or Lemma 2.8 in \cite{Hamenstadt-Jackel}). This implies \begin{equation}\label{De Giorgi-Nash-Moser_0}
        |L|(\Tilde{x})\leq C\Vert L\Vert _{L^2(B(\Tilde{x}))},
    \end{equation}
     where $B(\Tilde{x})$ is the unit ball at $\Tilde{x}$ and $C$ is a constant, and assuming $L\neq 0$. As \eqref{De Giorgi-Nash-Moser_0} is stable under $C^2$ convergence, we can extend the inequality to arbitrary $L$.
    Applying it to the scalar function $|L|=|\Tilde{\ovl{l}}|$, we obtain the following inequality. 
    \begin{equation}\label{eq:GNMapplication_1}
        \vert\ovl{l}\vert(x)= \vert\Tilde{\ovl{l}}\vert(\Tilde{x})\lesssim \Vert \Tilde{\ovl{l}}\Vert_{L^2(B(\Tilde{x}))}.
    \end{equation}

     Let $x$ be a point that does not lie far out into the cusp, and let $y\in B(x)$, one can verify that the number of lifts of $y$ in $B(\Tilde{x})$ is bounded by $ce^{r(y)}$ for some constant $c$ 
     (see for instance \cite[corollary 7.7]{Hamenstadt-Jackel}). This leads to \begin{equation}\label{eq:L2_estimate_l}
        \int_{B(\Tilde{x})} |\Tilde{\ovl{l}}|^2\dvol \lesssim \int_{B(x)}e^{r(y)}|\ovl{l}|^2(y)\dvol\leq e\int_{B(x)}|\ovl{l}|^2(y)\dvol\lesssim \Vert\ovl{l}\Vert^2_{L^2(M)}.
    \end{equation}
    In particular, taking $x\in T(0)$ and combining it with \eqref{eq:GNMapplication_1},
    we obtain \begin{equation*}
         \Vert\hat{l}\Vert_{C^0(T(0))}\lesssim \Vert \ovl{l}\Vert_{C^0(T(0))}=O(\Vert\ovl{l}\Vert_{L^2(M)}).
    \end{equation*}
    Setting $r=0$ in \eqref{eq:ODEsolutions}, this shows that $b_2^i,c_2,d_2=O(\Vert\ovl{l}\Vert_{L^2(M)})$. 

Consequently, $e^{2r}\hat{l}_{12}=0$, $e^r\hat{l}_{i3}=O\left(\Vert\ovl{l}\Vert_{L^2(M)}e^{-(\sqrt{3}-1)r}\right)$, $\hat{l}_{33} = O\left(\Vert\ovl{l}\Vert_{L^2(M)}e^{-r}\right)$, and $\tr_{h_0}(\hat{l}) = O\left(\Vert\ovl{l}\Vert_{L^2(M)}e^{-r}\right)$. As a result, the ODEs corresponding to $e^{2r}\hat{l}_{ii}$, $i=1,2$, have the following form \begin{equation*}            (e^{2r}\hat{l}_{ii})''-2(e^{2r}\hat{l}_{ii})'+e^{2r}\hat{l}_{ii} = O\left(\Vert\ovl{l}\Vert_{L^2(M)}e^{-r}\right),\quad i=1,2.
\end{equation*}
A similar argument shows that $e^{2r}\hat{l}_{ii}=O\left(\Vert\ovl{l}\Vert_{L^2(M)}e^{-r}\right)$.

We conclude that \begin{equation}\label{eq:bound_average}
    \vert\hat{l}\vert=O\left(\Vert\ovl{l}\Vert_{L^2(M)}e^{-(\sqrt{3}-1)r}\right).
\end{equation}
(1) follows by
\begin{align*}
\Vert\ovl{l}\Vert_{L^1(T\times[r,\infty))}=&\int_r^\infty |\hat{l}|(s)\vol(T(s))\,ds\\
    =&\int_r^\infty O\left(\Vert\ovl{l}\Vert_{L^2(M)}e^{-(\sqrt{3}+1)s}\right)ds=O\left(\Vert\ovl{l}\Vert_{L^2(M)}e^{-(\sqrt{3}+1)r}\right).
\end{align*}

Next, we prove (2), it remain to evaluate $\ovl{l}-\hat{l}$. Consider the lifts of $\ovl{l}$ and $\hat{l}$ that are defined as $\Tilde{\ovl{l}}(\Tilde{x}):=\ovl{l}(x)$ and $\Tilde{\hat{l}}(\Tilde{x}):=\hat{l}(x)$, respectively. By applying the De Giorgi-Nash-Moser estimate \eqref{De Giorgi-Nash-Moser_0} again to the scalar function $|L|=|\Tilde{\ovl{l}}-\Tilde{\hat{l}}|$, we have
    \begin{equation}\label{eq:GNMapplication}
        \vert\ovl{l}-\hat{l}\vert(x)=\vert\Tilde{\ovl{l}}-\Tilde{\hat{l}}\vert(\Tilde{x}) \lesssim \Vert \Tilde{\ovl{l}}-\Tilde{\hat{l}}\Vert_{L^2(B(\Tilde{x}))}.
    \end{equation}
    In Proposition B.2 and equation (6.16) of \cite{Jiang-VargasPallete_RF}, setting $f=\ovl{l}$ and $\xi=0$, we obtain \begin{equation}\label{eq:l2_of_difference}
        \Vert \Tilde{\ovl{l}}-\Tilde{\hat{l}}\Vert^2_{L^2(B(\Tilde{x}))}\lesssim \int_{T\times [r(x)-1,r(x)+1]}|\ovl{l}|^2_{C^1}\dvol\lesssim \Vert\ovl{l}\Vert^2_{L^2(M)}.
    \end{equation}
    Combining \eqref{eq:bound_average}, \eqref{eq:GNMapplication}, and \eqref{eq:l2_of_difference}, we obtain \begin{equation*}
        \Vert\ovl{l}\Vert_{C^0(T\times[0,\infty))}\lesssim \Vert\ovl{l}\Vert_{L^2(M)}.
    \end{equation*}

     For any point $x$ in the thick part of $M$, observe that since we have a lower bound on injectivity radius, \eqref{eq:GNMapplication_1} and \eqref{eq:L2_estimate_l} imply $|\ovl{l}(x)|\lesssim \Vert\ovl{l}\Vert_{L^2(M)}$. This completes the proof of (2).
\end{proof}

\begin{Lemma}\label{lem:measureconv-1eigen}
    Let $M$ be a finite-volume hyperbolic $3$-manifold, $\Pi_i \in S_{\frac{1}{i},\mu_{Leb}}(M)$ a sequence so that $\Omega_*\delta_{\phi_i}$, the measures associated to the minimal surfaces representing $\Pi_i$, weakly converge for compactly supported functions to $\mu_{Leb}$ on $\mathcal{F}rM$. 
    Then $$\lim_{i\rightarrow \infty}\Omega_*\delta_{\phi_i}\big(e^{T_i}\theta(l(T_i))\big) = \mu_{Leb}(\theta(\ovl{l})).$$
\end{Lemma}
\begin{proof}

     For any given $\epsilon>0$, applying Lemma \ref{lem:areaincusps} to the hyperbolic metric $h_0$, we can find a compact set $K\subset M$ so that 
     \begin{equation}\label{equ_7.7_1}
        \Omega_*\delta_{\phi_i}(\mathcal{F}r(M\setminus K))<\epsilon.
     \end{equation}
     Moreover, it follows from \eqref{equ_theta_conv_L2_2} that, when $i$ is sufficiently large, we have 
     \begin{equation}
         \Vert e^{T_i}\theta(l(T_i))-\theta(\ovl{l})\Vert_{C^0(M(s_i))}<\epsilon.
     \end{equation}
     Since $\Omega_*\delta_{\phi_i}$ converges to $\mu_{Leb}$ on compact sets,
         \begin{equation}\label{equ_7.7_2}
             \Omega_*\delta_{\phi_i}(\mathcal{F}r K)<(1+\epsilon) \mu_{Leb}(\mathcal{F}r K).
         \end{equation}
     Combining \eqref{equ_7.7_1}-\eqref{equ_7.7_2} and using the fact that $\Omega_*\delta_{\phi_i}$ has support in $\mathcal{F}r(M(s_i))$, we obtain \begin{align}\label{equ_7.7_3}
&\left|\Omega_*\delta_{\phi_i}\big(e^{T_i}\theta(l(T_i))\big)-\Omega_*\delta_{\phi_i}(\theta(\ovl{l})|_K)\right|\\\nonumber
    \leq &\left|\Omega_*\delta_{\phi_i}\big(e^{T_i}\theta(l(T_i))\big)-\Omega_*\delta_{\phi_i}\big(e^{T_i}\theta(l(T_i)|_K)\big)\right|\\\nonumber
    &+\left|\Omega_*\delta_{\phi_i}\big(e^{T_i}\theta(l(T_i)|_K)\big)-\Omega_*\delta_{\phi_i}(\theta(\ovl{l}|_{K}))\right|\\\nonumber
   \leq & \left\Vert e^{T_i}\theta(l(T_i))\right\Vert_{C^0(M(s_i)\setminus K)}\cdot\Omega_*\delta_{\phi_i}(\mathcal{F}r(M(s_i)\setminus K))\\\nonumber
    &+ \left\Vert e^{T_i}\theta(l(T_i))-\theta(\ovl{l})\right\Vert_{C^0(K)} \cdot\Omega_*\delta_{\phi_i}(\mathcal{F}r K)\\\nonumber
     < & \left(\left\Vert \theta(\ovl{l})\right\Vert_{C^0(M(s_i))}+\epsilon\right)\cdot\Omega_*\delta_{\phi_i}(\mathcal{F}r(M\setminus K))+ \epsilon (1+\epsilon)\mu_{Leb}(\mathcal{F}r K)\\\nonumber
    <& \left(\frac12\left\Vert \ovl{l}\right\Vert_{C^0(M)}+\epsilon\right)\epsilon+\epsilon (1+\epsilon),
    \end{align}
    which tends to $0$ as $\epsilon\to 0$ due to Lemma \ref{lem:L2-1eigendecay} (2).
    
    Using Lemma \ref{lem:L2-1eigendecay} (1) and choosing a larger compact set $K$ if needed, we get 
    \begin{align}\label{equ_7.7_4}
        |\mu_{Leb}(\theta(\ovl{l}))-\mu_{Leb}(\theta(\ovl{l}|_K))|=&  \mu_{Leb}(\theta(\ovl{l}|_{M\setminus K}))=\frac12 \mu_{Leb}(\ovl{l}|_{M\setminus K})\\\nonumber
        \leq& \frac12\vol_{h_0}(M)^{-1}\Vert\ovl{l}\Vert_{L^1(M\setminus K)}<\epsilon. 
    \end{align}
    As $\Omega_*\delta_{\phi_i}(\theta(\ovl{l}|_{K}))$ converges to $\mu_{Leb}(\theta(\ovl{l}|_{K}))$, the lemma is derived by \eqref{equ_7.7_3} and \eqref{equ_7.7_4}.
\end{proof}

Choosing $\omega\in (\frac{1}{2},1)$, then it follows from \eqref{equ_delta_Omega_*_2} and Lemma \ref{lem:measureconv-1eigen} that \begin{equation*}
    0<\delta\leq \mu_{Leb}(\theta(\ovl{l})).
\end{equation*}
However, the equality of \eqref{equ_spectrum} implies $\tr_{h_0}(\ovl{l})=0$, and hence 
\begin{equation*}
    \mu_{Leb}(\theta(\ovl{l}))=\frac{1}{2}\mu_{Leb}(\ovl{l})=\frac{1}{6}\fint_M \tr_{h_0}(\ovl{l})\dvol_{h_0}=0,
\end{equation*}
leading to a contradiction.
This means that the assumption \eqref{contradiction_area2_2} is false, therefore the inequality stated in Proposition \ref{prop_2.2} must hold.

The rigidity result in Proposition \ref{prop_2.2} and the proof of Theorem \ref{Thm_sar>-6_c0perturbation} follow from arguments similar to those used in the previous section.

\bibliographystyle{plain} 
\bibliography{ref}   
\end{document}